\documentclass{amsart}
\usepackage[textsize=scriptsize]{todonotes}
\usepackage{amsmath}
\usepackage{amssymb}
\usepackage{amsthm}
\usepackage{amscd}
\usepackage{enumerate}
\usepackage[pdfusetitle,unicode,hidelinks]{hyperref}
\usepackage{bbm}
\usepackage{etoolbox}
\usepackage{tikz-cd}
\usepackage{comment}

\newtoggle{draft}

\iftoggle{draft} {
\usepackage[margin=1.5in]{geometry}
}{ 
\usepackage[margin=1in]{geometry}
\setlength{\marginparwidth}{0.75in}
}
\newtheorem{proposition}{Proposition}
\newtheorem{corollary}[proposition]{Corollary}
\newtheorem{lemma}[proposition]{Lemma}
\newtheorem{theorem}[proposition]{Theorem}
\newtheorem{meta-theorem}[proposition]{Meta-Theorem}
\newtheorem{meta-corollary}[proposition]{Meta-Corollary}

\newtheorem*{conjecture*}{Conjecture}
\newtheorem*{theorem*}{Theorem}
\newtheorem*{corollary*}{Corollary}
\newtheorem*{proposition*}{Proposition}
\newtheorem*{lemma*}{Lemma}
\theoremstyle{definition}
\newtheorem{definition}[proposition]{Definition}

\newtheorem{notation}[proposition]{Notation}

\newtheorem*{definition*}{Definition}
\newtheorem*{construction*}{Construction}
\theoremstyle{remark}
\newtheorem{remark}[proposition]{Remark}
\newtheorem{warning}[proposition]{Warning}
\newtheorem*{remark*}{Remark}
\newtheorem{question}[proposition]{Question}
\newtheorem{example}[proposition]{Example}
\newtheorem*{example*}{Example}

\numberwithin{proposition}{section}

\newcommand{\ideg}{\mathrm{deg}}
\newcommand{\unit}{\mathrm{unit}}
\newcommand{\counit}{\mathrm{counit}}
\newcommand{\Ex}{\mathrm{Ex}}
\def\L{\mathrm{L}}
\newcommand{\purity}{\mathfrak{p}}
\def\GL{\mathrm{GL}}
\def\SL{\mathrm{SL}}
\def\Sp{\mathrm{Sp}}
\def\M{\mathrm{M}}
\newcommand{\CHt}{\widetilde{\mathrm{CH}}}
\newcommand{\CH}{\mathrm{CH}}
\def\adj{\leftrightarrows}
\newcommand{\wequi}{\simeq}
\DeclareRobustCommand{\ul}{\underline}
\newcommand{\iHom}{\ul{\operatorname{Hom}}}
\newcommand{\Hom}{\operatorname{Hom}}
\newcommand{\Ext}{\operatorname{Ext}}
\newcommand{\Spec}{\operatorname{Spec}}
\newcommand{\Proj}{\operatorname{Proj}}
\newcommand{\MW}{\mathrm{MW}}
\newcommand{\Cor}{\mathrm{Corr}}
\def\fr{\mathrm{fr}}

\def\op{\mathrm{op}}
\newcommand{\id}{\operatorname{id}}
\newcommand{\RR}{\mathbb{R}}
\newcommand{\CC}{\mathbb{C}}
\newcommand{\Z}{\mathbb{Z}}
\newcommand{\Q}{\mathbb{Q}}
\newcommand{\N}{\mathbb{N}}
\newcommand{\red}{\mathrm{red}}
\newcommand{\calO}{\mathcal{O}}
\def\PSh{\mathcal{P}}

\let\scr=\mathcal
\let\bb=\mathbb
\def\A{\bb A}
\def\P{\bb P}
\newcommand{\1}{\mathbbm{1}}
\newcommand{\SH}{\mathcal{SH}}
\def\ph{\mathord-}
\let\cat=\mathrm
\def\Sm{{\cat{S}\mathrm{m}}}
\def\Aff{{\cat{A}\mathrm{ff}}}
\def\Sch{\cat{S}\mathrm{ch}{}}
\newcommand{\tr}{\mathrm{tr}}
\newcommand{\perf}{\mathrm{perf}}

\newcommand{\KO}{\mathrm{KO}}

\newcommand{\KGL}{\mathrm{KGL}}
\newcommand{\MSL}{\mathrm{MSL}}
\newcommand{\kgl}{\mathrm{kgl}}
\newcommand{\KW}{\mathrm{KW}}

\newcommand{\GW}{\mathrm{GW}}
\newcommand{\W}{\mathrm{W}}
\newcommand{\emW}{\mathrm{EM}(\mathrm{W})}
\newcommand{\K}{\mathrm{K}}

\newcommand{\BL}{\mathrm{BL}}
\newcommand{\Vect}{\mathrm{Vect}}
\newcommand{\Bil}{\mathrm{Bil}}
\newcommand{\GS}{\mathrm{GS}}
\newcommand{\PH}{\mathrm{PH}}
\newcommand{\BM}{\mathrm{BM}}
\newcommand{\SSform}{\langle - \vert - \rangle^{\mathrm{SS}}}
\newcommand{\image}{\mathrm{Im}}
\newcommand{\ind}{\mathrm{ind}}
\def\H{\mathrm{H}}
\newcommand{\naive}{\mathrm{naive}}
\newcommand{\lra}[1]{\langle #1 \rangle}
\newcommand{\eff}{\mathrm{eff}}

\def\CAlg{\mathrm{CAlg}}
\newcommand{\gr}{\mathrm{gr}}
\newcommand{\fil}{\mathrm{fil}}
\newcommand{\FSynor}{\mathrm{FSyn^{or}}}
\newcommand{\Rsignature}{\operatorname{sign}}
\newcommand{\Rrank}{\operatorname{rank}}
\newcommand{\disc}{\operatorname{disc}}
\newcommand{\Sym}{\operatorname{Sym}}
\newcommand{\Gr}{\operatorname{Gr}}
\newcommand{\tautbun}{\mathcal{S}}
\newcommand{\Stautbun}{S}
\newcommand{\Gtautbun}{\mathcal{S}}
\newcommand{\Fun}{\mathrm{Fun}}
\newcommand{\Pic}{\mathrm{Pic}}
\newcommand{\Jac}{\mathrm{Jac}}
\newcommand{\FF}{\mathbb{F}}
\newcommand{\tO}{\tilde{O}}
\newcommand{\tOpom}{\tilde{O}^{(-)}}
\newcommand{\qcoh}{\mathrm{qcoh}}

\def\dual{*}

\iftoggle{draft} {
\newcommand{\NB}[1]{\todo[color=gray!40]{#1}}
\newcommand{\tom}[1]{\todo[color=green]{#1}}

\newcommand{\kirsten}[1]{\todo[color=yellow]{#1}}
}{ 
\newcommand{\NB}[1]{}
\newcommand{\tom}[1]{}
\newcommand{\kirsten}[1]{}
\renewcommand{\todo}[1]{}
}

\title[$\A^1$-Euler classes]{$\A^1$-Euler classes: six functors formalisms, dualities, integrality and linear subspaces of complete intersections}

\date{\today}

\author{Tom Bachmann}
\address{Department of Mathematics, Massachusetts Institute of Technology,
Cambridge, MA, USA}
\email{\href{mailto:tom.bachmann@zoho.com}{tom.bachmann@zoho.com}}

\author{Kirsten Wickelgren}
\address{Current: K.~Wickelgren, Department of Mathematics, Duke University, Durham, NC, USA }
\email{\href{kirsten.wickelgren@duke.edu}{kirsten.wickelgren@duke.edu}}

\begin{document}

\maketitle

\begin{abstract}
We equate various Euler classes of algebraic vector bundles, including those of \cite{Barge_Morel-Euler_classes}, \cite{CubicSurface}, \cite{DJK}, and one suggested by M.J. Hopkins, A. Raksit, and J.-P. Serre. We establish integrality results for this Euler class, and give formulas for local indices at isolated zeros, both in terms of 6-functor formalism of coherent sheaves and as an explicit recipe in commutative algebra of Scheja and Storch. As an application, we compute the Euler classes enriched in bilinear forms associated to arithmetic counts of $d$-planes on complete intersections in $\P^n$ in terms of topological Euler numbers over $\RR$ and $\CC$. 
\end{abstract}

\tableofcontents

\section{Introduction}


For algebraic vector bundles with an appropriate orientation, there are Euler classes and numbers enriched in bilinear forms. We will start over a field $k$, and then discuss more general base schemes, obtaining integrality results. Let  $\GW(k)$ denote the Grothendieck--Witt group of $k$, defined to be the group completion of the semi-ring of non-degenerate, symmetric, $k$-valued, bilinear forms, see e.g. \cite{Lam-intro_quadratic_forms_over_fields}. Let $\langle a \rangle$ in $\GW(k)$ denote the class of the rank $1$ bilinear form $(x,y) \mapsto axy$ for $a$ in $k^*$.

For a smooth, proper $k$-scheme $f: X \to \Spec k $ of dimension $n$, coherent duality defines a trace map $\eta_f: \H^n(X, \omega_{X/k}) \to k$, which can be used to construct the following Euler number in $\GW(k)$. Let $V \to X$ be a rank $n$ vector bundle equipped with a relative orientation, meaning  a line bundle $\scr L$ on $X$ and an isomorphism $$\rho: \det V \otimes \omega_{X/k} \to \scr L^{\otimes 2}. $$ For $0 \leq i, j \leq n$, let $\beta_{i,j}$ denote the perfect pairing \begin{equation}\label{beta_ij_intro_def} \beta_{i,j}: \H^i(X, \wedge^j V^\dual \otimes \scr L) \otimes \H^{n-i}(X, \wedge^{n-j} V^\dual \otimes \scr L)\to k\end{equation} given by the composition $$ \H^i(X, \wedge^j V^\dual \otimes \scr L) \otimes \H^{n-i}(X, \wedge^{n-j} V^\dual \otimes \scr L) \stackrel{\cup}{\to} \H^n(X, \wedge^n V^\dual \otimes \scr L^{\otimes 2}) \stackrel{\rho}{\to}  \H^n(X, \omega_{X/k}) \stackrel{\eta_f}{\to} k.$$ For $i = n-i$ and $j=n-j$,   note that $\beta_{i,j}$ is a bilinear form on $ \H^i(X, \wedge^j V^\dual \otimes \scr L)$. Otherwise, $\beta_{i,j} \oplus \beta_{n-i,n-j}$ determines the bilinear form on $\H^i(X, \wedge^j V^\dual \otimes \scr L) \oplus \H^{n-i}(X, \wedge^{n-j} V^\dual \otimes \scr L)$.  The alternating sum $$n^{\GS}(V) : = \sum_{0 \leq i, j \leq n} (-1)^{i+j}\beta_{i,j} $$ thus determines an element of $\GW(k)$, which we will call the Grothendieck--Serre duality or coherent duality Euler number. Note that $\beta_{i,j} \oplus \beta_{n-i,n-j}$  in $\GW(k)$ is an integer multiple of $h$ where $h$ denotes the hyperbolic form $h = \langle 1 \rangle + \langle -1 \rangle$, with Gram matrix $$
   h=
  \left[ {\begin{array}{cc}
   0 & 1 \\
   1 & 0 \\
  \end{array} } \right].$$ This notion of Euler number was suggested by M.J. Hopkins, J.-P. Serre and A. Raksit, and developed by M. Levine and Raksit for the tangent bundle in \cite{levine2018motivic}.
  
For a relatively oriented vector bundle $V$ equipped with a section $\sigma$ with only isolated zeros, an Euler number $n^{\PH}(V, \sigma)$ was defined in \cite[Section 4]{CubicSurface} as a sum of local indices $$n^{\PH}(V, \sigma) = \sum_{x : \sigma(x) = 0} \ind^{\PH}_x \sigma .$$ The index $\ind^{\PH}_x \sigma$ can be computed explicitly with a formula of Scheja--Storch \cite{scheja} or Eisenbud--Levine/Khimshiashvili \cite{eisenbud77} \cite{khimshiashvili} (see Sections \ref{subsec:poincare-hopf-euler-number} and \ref{Scheja_Storch_Serre_Duality_section}) and is also a local degree \cite{kass2016class} \cite{Brazeltontrace} (this is discussed further in Section \ref{sec:A1-degrees}). For example, when $x$ is a simple zero of $\sigma$ with $k(x)=k$, the index is given by a well-defined Jacobian $\Jac~\sigma$ of $\sigma$, $$\ind^{\PH}_x \sigma= \langle \Jac~\sigma(x) \rangle,$$ illustrating the relation with the Poincar\'e--Hopf formula for topological vector bundles. (For the definition of the Jacobian, see the beginning of Section \ref{subsection:arithmetic_count_d-planes_ci}.) In \cite[Section 4, Corollary 36]{CubicSurface}, it was shown that  $n^{\PH}(V, \sigma) = n^{\PH}(V, \sigma')$ when $\sigma$ and $\sigma'$ are in a family over $\bb A^1_L$ of sections with only isolated zeros, where $L$ is a field extension with $[L:k]$ odd. We strengthen this result by equating $n^{\PH}(V, \sigma)$ and $n^{\GS}(V)$; this is the main result of \S\ref{sec:GS-PH}.

\begin{theorem}[see \S\ref{subsec:poincare-hopf-euler-number}]\label{thm:introd-ePH=eGS}
Let $k$ be a field, and $V \to X$ be a relatively oriented, rank $n$ vector bundle on a smooth, proper $k$-scheme of dimension $n$. Suppose $V$ has a section $\sigma$ with only isolated zeros. Then $$n^{\PH}(V, \sigma) = n^{\GS}(V).$$ In particular, $n^{\PH}(V, \sigma)$ is independent of the choice of $\sigma$.
\end{theorem}

\begin{remark}
Theorem \ref{thm:introd-ePH=eGS} strengthens Theorem 1 of \cite{Bethea}, removing hypothesis (2) entirely. It also simplifies the proofs of \cite[Theorem 1]{CubicSurface} and \cite[Theorems 1 and 2]{FourLines}: it is no longer necessary to show that the sections of certain vector bundles with non-isolated isolated zeros are codimension $2$, as in \cite[Lemmas 54,56,57]{CubicSurface}, and \cite[Lemma 1]{FourLines}, because $n^{\PH}(V, \sigma)$ is independent of $\sigma$.  
\end{remark}

\subsection{Sketch proof and generalizations}
The proof of the above theorem proceeds in three steps.
\begin{enumerate}
  \setcounter{enumi}{-1}
\item For a section $\sigma$ of $V$, we define an Euler number relative to the section using coherent duality and denote it by $n^\GS(V, \sigma, \rho)$.
  If $\sigma = 0$, we recover the absolute Euler number $n^\GS(V, \rho)$, essentially by construction.
\item For two sections $\sigma_1, \sigma_2$, we show that $n^\GS(V, \sigma_1, \rho) = n^\GS(V, \sigma_2, \rho)$. To prove this, one can use homotopy invariance of Hermitian K-theory, or show that $n^\GS(V, \sigma_1, \rho) = n^\GS(V, \rho)$ by showing an instance of the principle that alternating sums, like Euler characteristics, are unchanged by passing to the homology of a complex. 
\item If a section $\sigma$ has isolated zeros, then $n^\GS(V, \sigma, \rho)$ can be expressed as a sum of local indices $\ind_{Z/S}(\sigma)$, where $Z$ is (a clopen component of) the zero scheme of $\sigma$.
\item For $Z$ a local complete intersection in affine space, i.e. in the presence of coordinates, we compute the local degree explicitly, and identify it with the \emph{Scheja--Storch} form \cite[3]{scheja}.
\end{enumerate}
Taken together, these steps show that $n^\GS(V, \rho)$ is a sum of local contributions given by Scheja--Storch forms, which is essentially the definition of $n^\PH(V, \rho)$.

These arguments can be generalized considerably, replacing the Grothendieck--Witt group $\GW$ by a more general \emph{cohomology theory} $E$. We need $E$ to admit transfers along proper lci morphisms of schemes, and an $\SL^c$-orientation (see \S\ref{sec:coh-theories} for more details).
Then for step (0) one can define an Euler class $e(V, \sigma, \rho)$ as $z^*\sigma_*(1)$, where $z$ is the zero section.
Step (2) is essentially formal; the main content is in steps (1) and (3).
Step (1) becomes formal if we assume that $E$ is \emph{$\A^1$-invariant}.
In particular, steps (0)-(2) can be performed for $\SL$-oriented cohomology theories represented by motivic spectra; this is explained in \S\S\ref{sec:coh-theories},\ref{sec:representable-cohomology},\ref{ref:representable-theory-euler-classes}.

It remains to find a replacement for step (3).
We offer two possibilities: in \S\ref{sec:A1-degrees} we show that, again in the presence of coordinates, the local indices can be identified with appropriate $\A^1$-degrees.
On the other hand, in \S\ref{sec:Euler-KO} we show that for $E = \KO$ the motivic spectrum corresponding to Hermitian $K$-theory, the local indices are again given by Scheja--Storch forms.
This implies the following.

\begin{corollary}[see Corollary \ref{cor:KO-scheja} and Definition \ref{def:index-general}]
Let $S=Spec(k)$, where $k$ is a field of characteristic $\ne 2$.\footnote{Here and many times in the text, we restrict to bases $S$ with $1/2 \in \calO_S$ in order for the classical constructions of Hermitian $K$-theory to be well-behaved. Forthcoming work by other authors is expected to produce well-behaved Hermitian $K$-theory spectra in all characteristics, and then all our assumptions on the characteristic can be removed.}
Let $\pi: X \to k$ be smooth and $V/X$ a relatively oriented vector bundle with a non-degenerate section $\sigma$.
Write $\varpi: Z=Z(\sigma) \to k$ for the vanishing scheme (which need not be smooth).
Then \[ n^\PH(V, \sigma) = \varpi_*(1) \in \KO^0(k) = \GW(k). \]
\end{corollary}
Here we have used the lci pushforward \[ \varpi_*: \KO^0(Z) \stackrel{\rho}{\wequi} \KO^{L_\varpi}(Z) \to \KO^0(k) \] of D\'eglise--Jin--Khan \cite{DJK}.
If moreover $X$ is proper then $\varpi_*(1)$ also coincides with $\pi_*z^*z_*(1)$, where $z: X \to V$ is the zero section (see Corollary \ref{cor:meta-rep}, Corollary \ref{cor:forget-number} and Proposition \ref{prop:euler-class-six-functors}).
This provides an alternative proof that $n^\PH(V,\sigma)$ is independent of the choice of $\sigma$ (under our assumption on $k$).

Another important example is when $E$ is taken to be the motivic cohomology theory representing Chow--Witt groups. This recovers the Barge--Morel Euler class \cite{Barge_Morel-Euler_classes} $e^{\BM}(V)$ in $\CHt^r(X, \det V^*)$, which is defined for a base field of characteristic not $2$. Suppose that $\rho$ is a relative orientation of $V$ and $\pi: X \to \Spec k$ is the structure map. 

\begin{corollary}\label{co:pieBM=nPH}
Let $k$ be a field of  characteristic $\ne 2$. Then $\pi_* e^{\BM}(V, \rho) = n^{\GS}(V, \rho)$ in $\GW(k)$.
\end{corollary}
\begin{proof}
We have $e^{\BM}(V, \rho) = e(V, \rho, H\tilde\Z)$; indeed by Proposition \ref{prop:euler-class-six-functors} $e(V, \rho, H\tilde\Z)$ can be computed in terms of pushforward along the zero section of $V$, and the exactly the same is true for $e^{\BM}$ by definition \cite[\S2.1]{Barge_Morel-Euler_classes}.
We also have $n^\GS(V,\rho) = n(V, \rho, \KO)$; indeed the right hand side is represented by the natural symmetric bilinear form on the cohomology of the Koszul complex by Example \ref{KO-Euler_class_Koszul}, and this is essentially the definition of $n^\GS(V,\rho)$.

It thus suffices to prove that $n(V, \rho, H\tilde\Z) = n(V, \rho, \KO) \in \GW(k)$.
Consider the span of ring spectra $H\tilde\Z \leftarrow \tilde f_0 \KO \to \KO$ as in the proof of Proposition \ref{prop:application-bc}.
It induces an isomorphism on $\pi_0(\ph)(k)$, namely with $\GW(k)$ in all cases.
The desired equality follows from naturality of the Euler numbers.

(An alternative argument proceeds as follows.
It suffices to prove that $\pi_* e^\BM(V,\rho)$ and $n^\GS(V, \rho)$ have the same image in $\W(k)$ and $\Z$.
The image of $n^\GS(V, \rho)$ in $\W(k$) is given by $n(V, \rho, \KW)$; for this we need only show that $e(V,\rho,\KW)$ is represented by the Koszul complex, which is Example \ref{KW-Euler_class_Koszul}.
It will thus be enough to show that $n(V, \rho, H\Z) = n(V, \rho, \KGL)$ and $n(V, \rho, \ul{W}[\eta^{\pm}]) = n(V, \rho, \KW)$; this follows as before by considering the spans $H\Z \leftarrow \kgl \to \KGL$ and $\ul{W}[\eta^{\pm}] \leftarrow \KW_{\ge 0} \to \KW$.)\footnote{We include this alternative argument because we feel that Example \ref{KW-Euler_class_Koszul} is more fully justified in this paper than Example \ref{KO-Euler_class_Koszul}.}
\end{proof}

The left hand side is the Euler class studied by M. Levine in \cite{Levine-EC}. We do not compare these Euler classes with the obstruction theoretic Euler class of \cite[Chapter 8]{A1-alg-top}. Asok and Fasel show that the latter agrees with $\pi_* e^{\BM}(V, \rho)$ up to a unit in $\GW(k)$ \cite{Asok-Fasel_comparing_Euler_classes}.

\subsection{Applications} 

It is straightforward to see that Euler numbers for cohomology theories are stable under base change (see Corollary \ref{cor:Euler-number-base-change}). This implies that when considering vector bundles on varieties which are already defined over e.g. $\Spec(\Z[1/2])$, then the possible Euler numbers are constrained to live in $\GW(\Z[1/2]) = \Z[\lra{-1}, \lra{2}] \subset \GW(\Q)$. Using novel results on Hermitian K-theory \cite{CDHHLMNNS-3} allows one to use the base scheme $\Spec \Z$ as well. Proposition \ref{prop:application-bc} contains both of these cases, and the $\Z[1/2]$ case is independent of \cite{CDHHLMNNS-3}. It follows that for relatively oriented bundles over $\Z$ the Euler numbers can be read off from topological computations (Proposition \ref{nRRnCCaretopEnums}). Over $\Z[1/2]$ the topological Euler numbers of the associated real and complex vector bundles together with one further algebraic computation over some field in which $2$ is not a square determine the Euler number (and this is again independent of \cite{CDHHLMNNS-3}). See Theorem~\ref{two_values_e(bundlesoverZ[1/2])}. 

We use this to compute a weighted count of $d$-dimensional hyperplanes in a general complete intersection $$\{f_1 = \ldots = f_j \} \hookrightarrow \P^n_k$$ over a field $k$. This count depends only on the degrees of the $f_i$ and not the polynomials $f_i$ themselves: it is determined by associated real and complex counts, for any $d$ and degrees so that the expected variety of $d$-planes is $0$-dimensional and the associated real count is defined. This is Corollary \ref{arithmetic_count_d_planes_complete_intersection}. For example, combining with results of Finashin--Kharlamov over $\mathbb{R}$, we have that $160,839 \lra{1}+ 160,650 \lra{-1}$ and $$32063862647475902965720976420325 \lra{1}+ 32063862647475902965683320692800 \lra{-1}$$ are arithmetic counts of the $3$-planes in a $7$-dimensional cubic hypersurface and in a $16$-dimensional degree $5$-hypersurface respectively.  See Example \ref{3-plane-counts}. This builds on results of Finahin--Kharlamov \cite{finashin13}, J.L. Kass and the second-named author \cite{CubicSurface}, M. Levine \cite{Levine-Witt} \cite{Levine-EC}, S. McKean  \cite{McKean-Bezout}, Okonek--Teleman \cite{okonek14}, S. Pauli \cite{Pauli-Quintic_3_fold}, J. Solomon \cite{solomon06}, P.Srinivasan and the second-named author \cite{FourLines}, and M. Wendt \cite{Wendt-oriented_schubert}. 

\subsection{Acknowledgements} We warmly thank M. J. Hopkins for suggesting the definition of the Euler class using coherent duality. We likewise wish to thank A. Ananyevskiy and I. Panin for the reference to \cite{knus} giving the existence of Nisnevich coordinates, as well as A. Ananyevskiy, M. Hoyois and M. Levine for useful discussions. We thank Johannes Rau for pointing out that the previous version of Lemma~\ref{lemm:GW-generators} required correction.

Kirsten Wickelgren was partially supported by National Science Foundation Awards DMS-1552730 and 2001890.

\subsection{Notation and conventions}
\subsubsection*{Grothendieck duality}
We believe that if $f: X \to Y$ is a morphism of schemes which is locally of finite presentation, then there is a well-behaved adjunction \[ f_!: D_\qcoh(X) \adj D_\qcoh(Y): f^! \] between the associated derived ($\infty$-)categories of unbounded complexes of $\scr O_X$-modules with quasi-coherent homology sheaves.
Unfortunately we are not aware of any references in this generality.
Instead, whenever mentioning a functor $f^!$, \emph{we implicitly assume that $X, Y$ are separated and of finite type over some noetherian scheme $S$}.
In this situation, the functor $f^!$ is constructed for homologically bounded above complexes in \cite[Tag 0A9Y]{stacks-project} (see also \cite{hartshorne1966residues,conrad2000grothendieck}), and this is all we will use.

\subsubsection*{Vector bundles}
We identify locally free sheaves and vector bundles \emph{covariantly}, via the assignment \[ \scr E \leftrightarrow Spec(Sym(\scr E^\dual)). \]
While it can be convenient to (not) pass to duals here (as in e.g. \cite{DJK}), we do not do this, since it confuses the first named author terribly.

\subsubsection*{Regular sequences and immersions}
Following e.g. \cite{SGA6}, by a regular immersion of schemes we mean what is called a Koszul-regular immersion in \cite[Tag 0638]{stacks-project}, i.e. a morphism which is locally a closed immersion cut out by a Koszul-regular sequence.
Moreover, by a regular sequence we will always mean a Koszu-regular sequence \cite[Tag 062D]{stacks-project}, and we reserve the term \emph{strongly regular sequence} for the usual notion.
A strongly regular sequence is regular \cite[Tag 062F]{stacks-project}, whence a strongly regular immersion is regular.
In locally noetherian situations, regular immersions are strongly regular \cite[Tags 063L]{stacks-project}.

\subsubsection*{Cotangent complexes}
For a morphism $f: X \to Y$, we write $L_f$ for the cotangent complex.
Recall that if $f$ is smooth then $L_f \wequi \Omega_f$, whereas if $f$ is a regular immersion then $L_f \wequi C_f[1]$, where $C_f$ denotes the conormal bundle.

\subsubsection*{Graded determinants}
We write $\widetilde\det: K(X) \to Pic(D(X))$ for the determinant morphism from Thomason--Trobaugh $K$-theory to the groupoid of graded line bundles.
If $C$ is a perfect complex, then we write $\widetilde \det C$ for the determinant of the associated $K$-theory point.
We write $\det C \in \Pic(X)$ for the ungraded determinant.

Given an lci morphism $f$, we put $\omega_f = \det L_f$ and $\widetilde\omega_f = \widetilde\det L_f$.

We systematically use graded determinants throughout the text.
For example we have the following compact definition of a relative orientation.
\begin{definition} \label{def:relative-orientation}
Let $\pi: X \to S$ be an lci morphism and $V$ a vector bundle on $X$.
By a \emph{relative orientation} of $V/X/S$ we mean a choice of line bundle $\scr L$ on $X$ and an isomorphism \[ \rho: \iHom(\widetilde\det V^\dual, \widetilde \omega_{X/S}) \xrightarrow{\wequi} \scr L^{\otimes 2}. \]
\end{definition}
Note that if $\pi$ is smooth, this just means that the locally constant functions $x \mapsto rank(V_x)$ and $x \mapsto \dim \pi^{-1}(\pi(x))$ on $X$ agree, and that we are given an isomorphism $\scr L^{\otimes 2} \wequi \omega_{X/S} \otimes \det V$.
Hence we recover the definition from \cite[Definition 17]{CubicSurface}.

\section{Equality of coherent duality and Poincar\'e--Hopf Euler numbers} \label{sec:GS-PH}

We prove Theorem \ref{thm:introd-ePH=eGS} in this section.

\subsection{Coherent duality Euler Number} \label{subsec:serr-duality-euler-number}
Let $f:X \to \Spec k$ be a smooth proper $k$-scheme of dimension $n$, and let $V$ be a rank $n$ vector bundle, relatively oriented by the line bundle $\scr L$ on $X$ and isomorphism $\rho: \det V \otimes \omega_{X/k} \to \scr L^{\otimes 2}$. Let $\sigma: X \to V$ be a section. Let $K(\sigma)^\bullet$ denote the Koszul complex \[ 0 \to \wedge^n V^* \to  \wedge^{n-1} V^* \to \ldots \to V^* \to \calO \to 0, \] with $\calO$ in degree $0$ and differential of degree $+1$ given by $$d(v_1 \wedge v_2 \wedge \ldots \wedge v_j) =  \sum_{i=1}^j (-1)^{i-1} v_i(\sigma) v_1 \wedge \ldots \wedge v_{i-1} \wedge v_{i+1} \wedge \ldots \wedge v_j.$$
This choice of $K(\sigma)^\bullet$ is $\Hom_{\calO}(-,\calO)$ applied to the Koszul complex of \cite[17.2]{Eisenbud_CommutativeAlgebra}.
$K(\sigma)^\bullet$ carries a canonical multiplication \begin{equation}\label{defm}m: K(\sigma)^\bullet \otimes K(\sigma)^\bullet \to K(\sigma)^\bullet\end{equation} defined in degree $-p$ by $m = \oplus_{i+j = p} 1_{\wedge^i V^*} \wedge 1_{\wedge^j V^*}$. Composing $m$ with the projection $p: K(\sigma)^\bullet \to \det{V}^\dual[n]$ defines a non-degenerate bilinear form $$\beta_{(V,\sigma)}: K(V, \sigma) \otimes K(V, \sigma) \to \det{V}^\dual[n],$$ $$\beta_{(V,\sigma)} =p m .$$

Tensoring $\beta_{(V,\sigma)}$ by $\scr L^{\otimes 2}$ and reordering the tensor factors of the domain, we obtain a non-degenerate symmetric bilinear form on $K(V, \sigma) \otimes \scr L$ valued in $(\det{V}^\dual \otimes \scr L^{\otimes 2})[n]$. The orientation $\rho$ determines an isomorphism $(\det{V}^\dual \otimes \scr L^{\otimes 2})[n] \to \omega_{X/k}[n]$. Composing $\beta_{(V,\sigma)} \otimes \scr L^{\otimes 2}$ with this isomorphism produces a non-degenerate bilinear form $$ \beta_{(V,\sigma, \rho)}: (K(V, \sigma) \otimes \scr L)  \otimes (K(V, \sigma) \otimes \scr L) \to \omega_{X/k}[n].$$

Let $D(X)$ denote the derived category of quasi-coherent $\calO_X$-modules. Serre duality determines an isomorphism $R f_* \omega_{X/k}[n] \cong \calO_k$ \cite[III Corollary 7.2 and Theorem 7.6]{hartshorne2013algebraic}.
Since $Rf_*$ is lax symmetric monoidal (being right adjoint to a symmetric monoidal functor), we obtain a symmetric morphism \[ Rf_*\beta_{(V,\sigma,\rho)}: [Rf_*(K(V, \sigma) \otimes \scr L)]^{\otimes 2} \to Rf_* \omega_{X/k}[n] \wequi \scr O_k, \] in $D(k)$, which is non-degenerate by Serre duality. 

The derived category $D(k)$ is equivalent to the category of graded $k$-vector spaces, by taking cohomology\footnote{In this section, we treat all categories as $1$-categories, i.e. ignore the higher structure of $D(k)$ as an $\infty$-category.}. If $V$ is a (non-degenerate) symmetric bilinear form in graded $k$-vector spaces, denote by $V^{(n)} = V_n \oplus V_{-n}$ (for $n \ne 0$) and $V^{(0)} = V_0$ the indicated subspaces; observe that they also carry (non-degenerate) symmetric bilinear forms. 

\begin{definition} \label{def:GSwrtsigmaeuler-number}
For a relatively oriented rank $n$ vector bundle $V \to X$ with section $\sigma$ and orientation $\rho$, over a smooth and proper variety $f: X \to k$ of dimension $n$, the \emph{Grothendieck--Serre-duality Euler number with respect to $\sigma$} is \[ n^\GS(V, \sigma, \rho) = \sum_{i \ge 0} (-1)^i [(Rf_*\beta_{(V,\sigma,\rho)})^{(i)}] \in \GW(k). \]
\end{definition}
\begin{remark} \label{rmk:abuse}
In order to not clutter notation unnecessarily, we also write the above definition as \[ n^\GS(V, \sigma, \rho) = \sum_i (-1)^i [(Rf_*\beta_{(V,\sigma,\rho)})_i]. \]
We shall commit to this kind of abuse of notation from now on.
\end{remark}

Recall that $n^\GS(V, \rho) \in \GW(k)$ was defined in the introduction, in terms of the symmetric bilinear form on $\bigoplus_{i,j} H^i(X, \Lambda^j V^\dual \otimes \scr L)$.
\begin{proposition}\label{pr:eGSVsigmarho=eGSVrho}
For any section $\sigma$ we have $n^\GS(V, \sigma, \rho) = n^\GS(V, \rho) \in \GW(k)$.
\end{proposition}
 
To prove Proposition \ref{pr:eGSVsigmarho=eGSVrho}, we use the hypercohomology spectral sequence $E^{i,j}_r(K^{\bullet})$ associated to a complex $K^{\bullet}$ of locally free sheaves on $X$ with $$E^{i,j}_1(K^{\bullet}) : = \H^j(X, K^i) \Rightarrow R^{i+j} f_* K^{\bullet}$$ Let $F_i$ denote the resulting filtration on $R^{*} f_* K^{\bullet}$, so that $$ \ldots \supseteq F_i = \image(\H^*(X, K^{\bullet \geq i}) \to \H^*(X, K^{\bullet})) \supseteq F_{i+1} \supseteq \ldots. $$ Given a perfect symmetric pairing of chain complexes $\beta: K^{\bullet} \otimes K^{\bullet} \to \omega_{X/k} [n]$, cup product induces pairings \[ \beta': R^*f_* K^{\bullet} \otimes R^*f_* K^{\bullet} \to R^*f_* \omega_{X/k} [n] \to k \] and \[ \beta_1: E^{*,*}_1(K^{\bullet}) \otimes E^{*,*}_1(K^{\bullet}) \to k.\] The following properties hold:
\begin{enumerate}
\item Placing the $k$ in the codomain of $\beta_1$ in bidegree $(-n,n)$, $\beta_1$ is a map of bigraded vector spaces and satisfies the Leibniz rule with respect to $d_1$. It thus induces $\beta_2: E^{*,*}_2(K^{\bullet}) \otimes E^{*,*}_2(K^{\bullet}) \to k$. Then $\beta_2$ satisfies the Leibnitz rule with respect to $d_2$ and hence induces $\beta_3$, and so on.
\item \label{beta_iperfect} All the pairings $\beta_i$ are perfect.
\item \label{betacompatible} The pairing $\beta'$ is compatible with the filtration in the sense  that $\beta'(F_i,F_k) = 0$ if $i+k >-n$.
\item It follows that $\beta'$ induces a pairing on $\gr_{\bullet} R^*f_* K^{\bullet} $. Under the isomorphism $\gr_{\bullet} \wequi E_\infty$, it coincides with $\beta_\infty$.
\item \label{filtered_perfect} $\beta'$ is perfect in the filtered sense: the induced pairing $F_i \otimes  R^*f_* K^{\bullet}/F_{-n-i+1} \to k$ is perfect.  (In particular the pairing $\beta'$ is perfect.)
\end{enumerate}

\begin{remark}
We do not know a reference for these facts, and proving them would take us too far afield.
The main idea is that we have a sequence of duality-preserving functors \[ C^\perf(X) \xrightarrow{\sigma_\bullet} D(X)^\fil \xrightarrow{\pi_*} D(k)^\fil. \]
Here $C^\perf(X)$ denotes the category of bounded chain complexes of vector bundles, $D(X)^\fil$ is the filtered derived category \cite{gwilliam2018enhancing}, and $\sigma_\bullet$ is the ``stupid truncation'' functor (composed with forgetting to the filtered derived category).
The first duality is with respect to $\iHom(\ph, \omega[n])$, the second with respect to $\iHom(\ph, \sigma_\bullet(\omega[n])) = \iHom(\ph, \omega[n](-n))$, and the third with respect to $\iHom(\ph, k[0](-n))$.
There are further duality preserving functors \[ (\ph)^\gr: D(k)^\fil \to D(k)^\gr \text{ and } U: D(k)^\fil \to D(k), \] where $D(X)^\gr = \Fun(\Z, D(X))$, with $\Z$ viewed as a discrete category.
Hence any perfect pairing $C \otimes C \to k[0](-n) \in D(k)^\fil$ induces a perfect pairing on $H_*C^\gr \otimes H_*C^\gr \to k(-n,n)$, satisfying (1), and a pairing $H_*UC \otimes H_* UC \to k$, satisfying (3, 5).
Moreover there is a spectral sequence $E_1 = H_*C^\gr \Rightarrow H_* UC$, satisfying (1) and (4).
(2) is obtained from the fact that passage to homology is a duality preserving functor.

We apply this to $K^\bullet \in C^\perf(X)$; then $\gr_i \sigma_\bullet K^\bullet = K^i[i]$ and hence $\gr_i(\pi_*\sigma_\bullet K^\bullet) = \pi_* K^i[i]$.
\end{remark}

\begin{lemma} \label{associated_gr_same_Euler_char}
Let $X$ be a graded $k$-vector space with a finite decreasing filtration $$X \supset \ldots \supset X_\bullet \supset X_{\bullet +1} \supset \ldots.$$
Suppose $X \otimes X \to k$ is a perfect symmetric bilinear pairing, which is compatible with the filtration in the sense of \eqref{betacompatible} and \eqref{filtered_perfect}.
Let $X^i$ denote the $i$th graded subspace of $X$ and $X_{\bullet}^i$ denote the $i$th graded subspace of $X_{\bullet}$. Then in $\GW(k)$, there is equality \[ \sum_i (-1)^i [X^i] = \sum_i (-1)^i [\gr_{\bullet} X^i].\]
\end{lemma}
\begin{proof}
Note that (5) implies that the pairing $\gr_\bullet X$ is non-degenerate, so the statement makes sense (recall Remark \ref{rmk:abuse}).
On any graded symmetric bilinear form, the degree $i$ and $-i$ part for $i\ne 0$ assemble into a metabolic space, with Grothendieck-Witt class determined by the rank (see Lemma \ref{lemm:metabolic-witt}).
It is clear that the ranks on both sides of our equation are the same; hence it suffices to prove the Lemma in the case where $X^i=0$ for $i \ne 0$.
We may thus ignore the gradings.

Let $N$ be maximal with the property that $X_N \ne 0$.
We have a perfect pairing \[ X_{N+1} \otimes X/X_{-n-N} \to k. \]
Since $X_{N+1}=0$ we deduce that $X_{-n-N} = X$ and hence $X_j = X$ for all $j \le -n-N$.
If $-n-N \ge N$ then $X = X_N(N)$ and there is nothing to prove; hence assume the opposite.

We have the perfect pairing \[ X_N/X_{N+1} \otimes X_{-n-N}/X_{-n-N+1} \wequi X_N \otimes X/X_{-n-N+1} \to k. \]
Pick a sequence of subspaces $X \supset X'_{-n-N+1} \supset \dots \supset X'_{N-1}$ such that $X'_i \subset X_i$ and the canonical projection $X'_i \to X_i/X_{N}$ is an isomorphism.
Extend the filtration $X'$ by zero on the left and constantly on the right.
By construction, $X'^\gr_i = X^\gr_i$ for $i \ne N,-n-N$, and the pairing on $X' \subset X$ is perfect in the filtered sense.
By \cite[Lemma I.3.1]{milnor1973symmetric}, we have $X = X'\oplus (X')^\perp$.
By induction on $N$, we have $[X'] = [\gr_{\bullet} X']$.
It thus suffices to show that $[(X')^\perp] = [\gr_N X \oplus \gr_{-n-N} X]$.
This holds since both sides are metabolic of the same rank: $X_{-n-N}$ is an isotropic subspace of half rank on either side (see again Lemma \ref{lemm:metabolic-witt}).
\end{proof}

\begin{lemma} \label{Euler_char_unchanged_homology}
Let $E^\bullet$ be a chain complex with a non-degenerate, symmetric bilinear form $E^\bullet \otimes E^\bullet \to k[0]$.
Then \[ \sum_i (-1)^i [H^i(E)] = \sum_i [E^i] \in \GW(k). \]
\end{lemma}
\begin{proof}
Since passing to homology is a duality preserving functor, the statement makes sense.
Both sides have the same rank, so it suffices to prove equality in $\W(k)$ (see Lemma \ref{lemm:metabolic-witt}).
We have a perfect pairing $C^i \otimes C^{-i} \to k$ and similarly for homology.
Both are metabolic unless $i = 0$.
We can choose a splitting \[ C^{0} = H \oplus C'; \] here $H \subset ker(C^{0} \to C^{1})$ maps isomorphically to $H^{0}(C)$.
The restriction of the pairing on $C^{0}$ to $H$ is perfect by construction, and hence $C^{0} = H \oplus H^\perp$.
It suffices to show that $H^\perp$ is metabolic.
Compatibility of the pairing with the differential shows that $d(C^{-1}) \subset C^{0}$ is an isotropic subspace.
Self-duality shows that \[ im(d: C^{-1} \to C^{0}) \wequi im(d^\vee: (C^{1})^\vee \to (C^{0})^\vee) \wequi im(d: C^{0} \to C^{1})^\vee \] which implies that $d(C^{-1}) \subset H^\perp$ is of half rank.
This concludes the proof.
\end{proof}

\begin{proof}[Proof of Proposition \ref{pr:eGSVsigmarho=eGSVrho}]
Let $K^{\bullet} =K(V, \sigma)^{\bullet} \otimes \scr L $.
We compute
\begin{align*}
n^\GS(V, \rho) &\stackrel{\text{def.}}{=} \sum (-1)^{i+j} [E^{i,j}_1(K^{\bullet})] \\
  &\stackrel{\text{L.\ref{Euler_char_unchanged_homology}}}{=} \sum (-1)^{i+j} [E^{i,j}_{\infty}(K^{\bullet})] \\
  &\stackrel{\text{L.\ref{associated_gr_same_Euler_char}}}{=} \sum_i (-1)^i [R^if_* K^{\bullet}] \\
  &\stackrel{\text{def.}}{=} n^\GS(V, \sigma, \rho).
\end{align*}
This is the desired result.
\end{proof}

\begin{remark} \label{rmk:eGSVsigma-alternative}
Admitting a version of Hermitian $K$-theory which is $\A^1$-invariant on regular schemes and has proper pushforwards, one can given an alternative proof of Proposition \ref{pr:eGSVsigmarho=eGSVrho} by considering the Koszul complex with respect to the section $t\sigma$ on $\A^1 \times X$.
While we believe such a theory exists, at the time of writing there is no reference for this in characteristic $2$, so we chose to present the above argument instead.
\end{remark}
 
\subsection{Local indices for $n^\GS(V, \sigma, \rho)$} \label{subsec:local-indices}

Suppose that $\sigma$ is a section with only isolated zeros. Let $i$ denote the closed immersion $i: Z = Z(\sigma) \hookrightarrow X$ given by the zero locus of $\sigma$. We express $n^\GS(V, \sigma, \rho)$ as a sum over the points $z$ of $Z$ of a local index at $z$. To do this, we use a pushforward in a suitable context and show that $\beta_{(V,\sigma)}$ is a pushforward from $Z$.

For a line bundle $\scr L$ on a scheme $X$, denote by $\BL_\naive(D(X), \scr L[n])$ the set of isomorphism classes of non-degenerate symmetric bilinear forms on the derived category of perfect complexes on $X$, with respect to the duality $\iHom(\ph, \scr L[n])$.
For a proper, lci map $f: X' \to X$, coherent duality supplies us with a \emph{trace map} $\eta_{f,\scr L}: f_* f^!(\scr L)\to \scr L$.
We can use this to build a pushforward (see \cite[Theorem 4.2.9]{calmes2009tensor}) \begin{gather*} f_*: \BL_\naive(D(X'), f^!\scr L) \to \BL_\naive(D(X), \scr L), \\ [E \otimes E \xrightarrow{\phi} f^! \scr L] \mapsto [f_* E \otimes f_*E \to f_*(E \otimes E) \xrightarrow{f_* \phi} f_*(f^! \scr L) \xrightarrow{\eta_{f,\scr L}}  \scr L]. \end{gather*} 

\begin{remark}\label{uppershriek_and_tensor}
There is a canonical weak equivalence $f^!\scr L \wequi f^! \scr O_X \otimes f^* \scr L$ and $\eta_{f,\scr L}$ is given by the composition $$f_*(f^!\scr L) \wequi f_*( f^! \scr O_X \otimes f^* \scr L ) \wequi f_* f^! \scr O_X \otimes \scr L \xrightarrow{\eta_f \otimes \id_{\scr L}} \scr L,$$ where $\eta_f = \eta_{f, \scr O_S}$ \cite[Lemma 47.17.8]{stacks-project}.
\end{remark}

\begin{example}\label{ex:f_*betaVsigmarho=e}
Consider the case of a relatively oriented vector bundle $V$ on a smooth proper variety $f: X \to Spec(k)$.
Note that elements of $\BL^\naive(k)$ are just isomorphism classes of symmetric bilinear forms on graded vector spaces.
The orientation supplies us with an equivalence \[ f^!(\scr O_k) \wequi \omega_{X/k}[n] \wequi \det V^* [n] \otimes \scr L^{\otimes 2}. \]
Under the induced pushforward map we have \[ f_* [\beta_{(V,\sigma,\rho)}] = n^\GS(V, \sigma, \rho) \in \BL^\naive(k), \] where $\beta_{(V,\sigma,\rho)} \in \BL^\naive(X, \det V^* [n] \otimes \scr L^{\otimes 2})$ is the form on $K(V, \sigma) \otimes \scr L$ defined in \S\ref{subsec:serr-duality-euler-number}.
\end{example}

\begin{remark} \label{rmk:non-sensible}
A symmetric bilinear form $\phi$ on the derived category $D(S)$ is usually not a very sensible notion.
We offer three ways around this.\begin{enumerate}
\item If $1/2 \in S$, we could look at the image of $\phi$ in the Balmer-Witt group of $S$.
\item If $\phi$ happens to be concentrated in degree zero, it corresponds to a symmetric bilinear form on a vector bundle on $S$, which is a sensible invariant.
\item If $S = Spec(k)$ is the spectrum of a field, then $D(S)$ is equivalent to the category of graded vector spaces, and we can split $\phi$ into components by degree and consider \[ cl(\phi) := [H^0(\phi)] + \sum_{i > 0} (-1)^i [H^i(\phi) \oplus H^{-i}(\phi)] \in \GW(k). \]
\end{enumerate}
\end{remark}

Let $1_Z$ denote the element of $\BL_\naive(D(Z), \scr O_Z[0])$ represented by $\calO_Z \otimes \calO_Z \to \calO_Z$.

\begin{proposition}\label{Koszul_form_push_forward_from_support}
Let $X$ be a scheme, $V$ a vector bundle, and $\sigma \in \Gamma(X, V)$ a section locally given by a regular sequence.
Write $i: Z = Z(\sigma) \hookrightarrow X$ for the inclusion of the zero scheme.
Proposition \ref{prop:f!-Lf} yields a canonical equivalence $i^!\det(V^\dual)[n] \wequi \scr O_Z[0]$, where $n$ is the rank of $V$; under the induced map \[ i_*: \BL_\naive(D(Z), \scr O_Z[0]) \to \BL_\naive(D(X), \det(V^\dual)[n]) \] we have $i_*(1_Z) = \beta_{(V,\sigma)}$, where \[ \beta_{(V,\sigma)}: K(V, \sigma) \otimes K(V, \sigma) \to \det(V^\dual)[n] \] is the canonical pairing on the Koszul complex as in \S\ref{subsec:serr-duality-euler-number}.
\end{proposition}
\begin{proof}
Because $\sigma$ locally corresponds to a regular sequence, the canonical map $r:K(V, \sigma)^{\bullet} \to i_*\calO_Z$ is an equivalence in $D(X)$.
The canonical projection $i_* \scr O_Z \wequi K(V, \sigma) \to \det(V^\dual)[n]$ induces by adjunction a map $\scr O_Z \to i^!\det(V^\dual)[n]$.
We claim that this is the equivalence of Proposition \ref{prop:f!-Lf}.
The proof of said proposition shows that the problem is local on $Z$, so we may assume that $V$ is trivial.
Then this map is precisely the isomorphism constructed in \cite[Proposition III.7.2 and preceeding pages]{hartshorne1966residues}, which is also the isomorphism employed in the proof of Proposition \ref{prop:f!-Lf}.

Now we prove that $i_*(1_Z) = \beta_{(V,\sigma)}$.
Consider the following diagram
\begin{equation*}
\begin{CD}
K(V, \sigma) \otimes K(V, \sigma) @>{r \otimes r}>> i_*\scr O_Z \otimes^L i_* \scr O_Z \\
@V{m_K}VV                                           @V{m_Z}VV   \\
K(V, \sigma)                      @>{r}>>           i_* \scr O_Z \wequi i_* i^! \det(V^\dual)[n] @>{\tr}>> \det(V^\dual)[n].
\end{CD}
\end{equation*}
The map $m_K: K(V, \sigma) \otimes K(V, \sigma) \to K(V, \sigma)$ is the canonical multiplication (see \eqref{defm} \S\ref{subsec:serr-duality-euler-number}), and $m_Z: i_*\scr O_Z \otimes^L i_* \scr O_Z \to i_* \scr O_Z \otimes i_* \scr O_Z \to i_* \scr O_Z$ is equivalently given by either the multiplication in $\scr O_Z$ or the lax monoidal witness transformation of $i_*$.
The former interpretation shows that the left hand square commutes.
The pairing $i_*(1)$ is given by the composite from the top right hand corner to the bottom right hand corner.
To prove the claim it suffices to show that the bottom row composite $K(V, \sigma) \to\det(V^\dual)[n]$ is the canonical projection.
This follows by adjunction from our choice of equivalence $\scr O_Z \wequi i^! \det(V^\dual)[n]$.

This concludes the proof.
\end{proof}

Proposition \ref{Koszul_form_push_forward_from_support} is an example of a more general phenomenon given in Meta-Theorem \ref{thm:meta}.

\begin{lemma}[\cite{calmes2009tensor}, Theorem 5.1.9] \label{lemm:pushforward-functorial}
Let $g: Z \to Y$ and $f: Y \to X$ be proper maps.\footnote{Recall our convention that since we are invoking a functor $f^!$, $Z,Y,X$ are also finite type and separated over a noetherian base $S$. Without this we should add the hypothesis that $f,g$ are locally (so globally) of finite presentation.}
Given equivalences $f^! \scr L \wequi \scr M [n]$ and $g^! \scr M [n] \wequi \scr N$, the canonical equivalence $(f g)^! \wequi g^! f^!$ produces a weak equivalence $(f g)^! \scr L \wequi \scr N$, and consequently push forward maps \[g_*:  \BL_\naive(D(Z),  \scr N) \to  \BL_\naive(D(Y),  \scr M [n]) \]  \[f_*:  \BL_\naive(D(Y),  \scr M [n]) \to \BL_\naive(D(X), \scr L) \]    \[(f g)_*:  \BL_\naive(D(Z),  \scr N) \to \BL_\naive(D(X), \scr L). \]
There is a canonical equivalence $(f g)_* \wequi f_* g_*$.
\end{lemma}
\begin{proof}
The main point is that $\eta_{f, \scr L} \circ f_*(\eta_{g, \scr M [n]}) = \eta_{fg, \scr L}$.
The categorical details are worked out in the reference.
\end{proof}

Now we get back to our Euler numbers.
Let $X/k$ be smooth and proper, $V$ a relatively oriented vector bundle, $\sigma$ a section of $V$ with only isolated zeros.
Write $i: Z = Z(\sigma) \hookrightarrow X$ for the inclusion of the zero scheme of $\sigma$.
Let $\varpi: Z \to \Spec k$ and $f: X \to \Spec k$ denote the structure maps, so that $\varpi = f i $.

The weak equivalence $i^!\det(V^\dual)[n] \wequi \scr O_Z[0]$ of Proposition \ref{Koszul_form_push_forward_from_support} and Remark \ref{uppershriek_and_tensor} produce a weak equivalence $i^! (\det V^*[n] \otimes \scr L^{\otimes 2}) \cong i^* \scr L^{\otimes 2}$. The orientation $\rho$ gives an isomorphism $\det V^*[n] \otimes \scr L^{\otimes 2}\cong \omega_{X/k}[n]$. Combining, we have a chosen weak equivalence \[ i^! (\omega_{X/k}[n]) \cong i^* \scr L^{\otimes 2}. \]
Since also $f^! \calO_k \wequi \omega_{X/k} [n]$ (see e.g. Proposition \ref{prop:f!-Lf}), we  therefore obtain a canonical equivalence \[ \varpi^! \calO_k \wequi  i^* \scr L^{\otimes 2}. \]
We use this equivalence to define \[ \varpi_*: \BL^\naive(Z, i^* \scr L^{\otimes 2}) \to \BL^\naive(k). \]
\begin{corollary}
With this notation, we have \[ n^\GS(V, \sigma, \rho) =\varpi_* (i^* \scr L \otimes i^* \scr L \to i^* \scr L^{\otimes 2}). \]
\end{corollary}
\begin{proof}
By Lemma \ref{lemm:pushforward-functorial} we have $\varpi_* = f_* i_*$.
Proposition \ref{Koszul_form_push_forward_from_support} and the projection formula imply that $i_*(i^* \scr L \otimes i^* \scr L \to i^* \scr L^{\otimes 2}) = \beta_{V,\sigma,\rho}$.
We conclude by Example~\ref{ex:f_*betaVsigmarho=e}.
\end{proof}

Suppose that $\sigma$ has isolated zeros, or in other words that the support of $\sigma$ is a disjoint union of points.
Then $n^\GS(V, \sigma, \rho)$ can be expressed as a sum of local contributions. Namely, for each point $z$ of $Z$, let $i_z: Z_z \hookrightarrow X$ denote the chosen immersion coming from the connected component of $Z$ given by $z$. Let $\varpi_z: Z_z \to \Spec k$ denote the structure map. Then $$ n^\GS(V, \sigma, \rho) = \sum_{z \in Z}\varpi_{z*} (i_z^* \scr L \otimes i_z^* \scr L \to i_z^* \scr L^{\otimes 2}) .$$ 

In light of this we propose the following.
\begin{definition} \label{def:local-index-eGS}
For a relatively oriented vector bundle with a section as above, and $z \in Z(\sigma)$, we define \[ \ind_z(\sigma) = \ind_z(V, \sigma, \rho) = \varpi_{z*}(i_z^* \scr L \otimes i_z^* \scr L \to i_z^* \scr L^{\otimes 2}) \in \BL^\naive(k). \]
\end{definition}
The above formula then reads \begin{equation} \label{eq:eGS-formula} n^\GS(V, \sigma, \rho) = \sum_{z \in Z} \ind_z(V, \sigma, \rho). \end{equation}

In the next two subsections, we compute the local contributions $\ind_z(\sigma)$ as an explicit bilinear form constructed by Scheja and Storch \cite{scheja}, appearing in the Eisenbud--Levine--Khimshiashvili signature theorem \cite{eisenbud77} \cite{khimshiashvili}, and used as the local index of the Euler class constructed in \cite[Section 4]{CubicSurface}.

\subsection{Scheja-Storch and coherent duality}\label{Scheja_Storch_Serre_Duality_section}
Let $S$ be a scheme, $\pi: X \to S$ a smooth scheme of relative dimension $n$, and $Z \subset X$ closed with $\varpi: Z \to S$ finite.
Suppose given the following data:
\begin{enumerate}
\item Sections $T_1, \dots, T_n \in \scr O(X)$ such that $T_i \otimes 1 - 1 \otimes T_i$ generate the ideal of $X \subset X \times_S X$.
\item Sections $f_1, \dots, f_n \in \scr O(X)$ such that $Z = Z(f_1, \dots, f_n)$.
\end{enumerate}
\begin{remark} \label{rmk:SS-regular}
Since $Z \to X$ is quasi-finite, Lemma \ref{lemm:automatic-regularity} shows that $f_1, \dots, f_n$ is a regular sequence and $Z \to X$ is flat, so finite locally free (being finite and finitely presented \cite[Tag 02KB]{stacks-project}).
\end{remark}
Choose $a_{ij} \in \scr O(X \times_S X)$ such that \[ f_i\otimes 1 - 1 \otimes f_i = \sum_j a_{ij}(T_j \otimes 1 - 1 \otimes T_j). \]
Let $\Delta \in \scr O(Z \times_S Z)$ be the image of the determinant of $a_{ij}$.
Since $\varpi$ is finite locally free, $\Delta$ determines an element $\tilde \Delta$ of \[ \Hom_{\scr O_S}(\scr O_S, (\varpi \times_S \varpi)_* \scr O_{Z \times_S Z}) \wequi \Hom_{\scr O_S}(\scr O_S, \varpi_* \scr O_Z \otimes \varpi_* \scr O_Z) \wequi \Hom_{\scr O_S}((\varpi_* \scr O_Z)^\dual, \varpi_* \scr O_Z). \]

\begin{remark} \label{rmk:scheja-storch-form}
We can make $\tilde \Delta$ explicit: if $\Delta = \sum_i b_i \otimes b_i'$, then \[ \tilde\Delta(\alpha) = \sum_i \alpha(b_i)b_i'. \]
\end{remark}
\begin{remark} \label{rmk:scheja-differentiation}
By construction, the pullback of $\Delta$ along the diagonal $\delta: Z \to Z \times_S Z$ is the determinant of the differentiation map $C_{Z/X} \to \Omega_X|_Z$ with respect to the canonical bases.
In other words this is the \emph{Jacobian}: \[ \delta^*(\Delta) = \Jac F := \det \left(\frac{\partial f_i}{\partial T_j}\right)_{i,j=1}^n. \]
\end{remark}

\begin{theorem} \label{thm:introd-2}
Under the above assumptions, the map \[ \tilde\Delta: (\varpi_* \scr O_Z)^\dual \to \varpi_* \scr O_Z \] is a symmetric isomorphism and hence determines a symmetric bilinear structure on $\varpi_* \scr O_Z$.
This is the same structure as $\varpi_*(1)$, i.e. \[ \varpi_*(\scr O_Z) \otimes \varpi_*(\scr O_Z) \to \varpi_*(\scr O_Z) \wequi \varpi_*(\varpi^! \scr O_S) \xrightarrow{\eta_\varpi} \scr O_S. \]
\end{theorem}
Here the isomorphism $\varpi^! \scr O_Z \wequi \scr O_Z$ arises from \[ \varpi^!(\scr O_Z) \wequi \widetilde\det L_\varpi \wequi \omega_{Z/X} \otimes \omega_{X/S} \wequi \scr O, \] with the first isomorphism given by Proposition \ref{prop:f!-Lf}, and the third isomorphism given by the sections $(T_i)$ and $(f_i)$.

\begin{remark}
The theorem asserts in particular that the isomorphism $\tilde\Delta$, and hence the section $\Delta$, are independent of the choice of the $a_{ij}$.
\end{remark}

We begin with some preliminary observations before delving into the proof.
The problem is local on $S$, so we may assume that $S = Spec(A)$; then $Z = Spec(B)$.
Since $\varpi$ is finite, there is a canonical isomorphism \cite[III \S 8 Theorem 8.7 (3), or Ideal Theorem (3) p. 6]{hartshorne1966residues} \[ \varpi^! \wequi \iHom_A(B, \ph): D(A) \to D(B). \]
In particular \[ \varpi^!(A) \wequi \iHom_A(B, A) \] and the trace map takes the form \cite[Ideal theorem 3) pg 7]{hartshorne1966residues} \[ \varpi_* \varpi^! A \wequi \Hom_A(B, A) \to A, \eta \mapsto \eta(1). \]

\begin{proof}[Proof of Theorem \ref{thm:introd-2}]
The isomorphisms \[  B^\dual = \iHom_A(B,A) \wequi \varpi^!(A) \wequi B \] determine an element $\Delta' \in \Hom_A(B^\dual, B)$.
The theorem is equivalent to showing that $\tilde\Delta = \Delta'$.

We thus need to make explicit the isomorphism \[ \iHom_A(B, A) \wequi \varpi^!(\scr O_A) \wequi i^! \pi^! (\scr O_A) \wequi \omega_{Z/X} \otimes \omega_{X} \wequi \scr O. \]
Tracing through the definitions (including the proof of \cite[III Proposition 8.2]{hartshorne1966residues}), one finds that this isomorphism arises by computing \[ \Ext^n_X(B, \scr O_X) \] in two ways.
One the one hand, the kernel of the surjection $\scr O_X \to B$ is generated by $f_1, \dots, f_n$, which is a regular sequence by Remark \ref{rmk:SS-regular}; let $K_A(f)^\bullet$ denote the corresponding Koszul complex.
On the other hand we can consider the embedding $Z \hookrightarrow X \times Z$; its ideal is generated by the strongly regular sequence $T_i - t_i$, where $t_i$ is the image of $T_i$ in $B$.
We thus obtain a resolution $K_B(T-t)^\bullet \to B$ over $X \times Z$.
Since $p: X \times Z \to X$ is finite, $p_*K_B(T-t)^\bullet \to p_*B=B$ is still a resolution.
We shall conflate $K_B(T-t)$ and $p_* K_B(T-t)$ notationally.
We can thus compute \[ \Ext^n_{X}(B, \scr O_X) \wequi coker(\Hom_{X}(K_B(T-t)^{n-1}, \scr O_X) \to \Hom_{X}(K_B(T-t)^{n}, \scr O_X)). \]
Since $\Hom_{X}(B \otimes \scr O_X, \scr O_X) \wequi \Hom_A(B, \scr O_X) \wequi \Hom_A(B, A) \otimes_A \scr O_X,$ there is a natural map $\xi: \Hom_A(B, A) \to \Hom_{X}(K_B(T-t)^{n}, \scr O_X)$ (sending $\alpha$ to $\alpha \otimes 1$).
One checks that this induces $\Hom_A(B, A) \wequi coker(\dots) \wequi Ext^n_{X}(B, \scr O_X)$.

We can write down a map of resolutions $\zeta: K_A(f) \to K_B(T-t)$ as follows.
The kernel of $B \otimes \scr O_X \to B$ is by construction generated by $\{T_i - t_i\}_i$, but it also contains $f_i$.
Note that $f_i = \sum_j \bar a_{ij}(T_j - t_j)$, where we write $\bar a_{ij}$ for the image of $a_{ij}$ in $\scr O_X \otimes B$.
Letting $K_A(f)$ be the exterior algebra on $\{e_1, \dots, e_n\}$ and $K_B(T-t)$ the exterior algebra on $\{e_1', \dots, e_n'\}$, the map $\zeta$ is specified by $\zeta(e_i) = \sum_j \bar a_{ij} e_j'$.
The isomorphism \[ \Hom_A(B, A) \wequi h^n \Hom_{X}(K_B(T-t)^\bullet, \scr O_X) \wequi h^n \Hom_{X}(K_A(f)^\bullet, \scr O_X) \wequi B \] is thus given by
\begin{gather*}
  \Hom_A(B, A) \xrightarrow{\xi} \Hom_{X}(K_B(T-t)^{n}, \scr O_X) \xrightarrow{\det(a_{ij})^*} \Hom_{X}(K_A(f)^{n}, \scr O_X) \\ \wequi \Hom_{X}(\scr O_X, \scr O_X) \wequi \scr O_X \to B.
\end{gather*}
Write the image of $\det(a_{ij})$ in $B \otimes B$ as $\sum_k b_k \otimes b_k'$. Tracing through the definitions, we find that the above composite sends $\alpha \in \Hom_A(B, A)$ to $\sum_k \alpha(b_k)b_k'$.
By Remark \ref{rmk:scheja-storch-form}, this is precisely $\tilde\Delta$.

This concludes the proof.
\end{proof}

\begin{definition} \label{def:SSform}
If $X=U \subset \A^n_S$, $(T_i)$ are the standard coordinates, and $F=(f_1, \dots, f_n)$, we denote the symmetric bilinear form constructed above by \[\SSform = \SSform(U, F, S). \]
\end{definition}
This form was first constructed, without explicitly using coherent duality, by Scheja and Storch \cite[3]{scheja}.

\begin{example}\label{F=0_is_one_point}
Suppose that $Z \to S$ is an isomorphism (where $Z=Z(F)$ as above), so that the diagonal $\delta: Z \to Z \times_S Z$ is also an isomorphism.
Then $\SSform$ is just the rank $1$ bilinear form corresponding to multiplication by $\delta^*(\Delta) \in \scr O_Z \wequi \scr O_S$.
In other words, using Remark \ref{rmk:scheja-differentiation}, $\SSform$ identifies with $(x,y) \mapsto (\Jac~F) xy$.
\end{example}

\subsection{The Poincar\'e-Hopf Euler number with respect to a section}
\label{subsec:poincare-hopf-euler-number}

In this subsection, we recall the Euler class defined in \cite[Section 4]{CubicSurface} and prove Theorem \ref{thm:introd-ePH=eGS}. To distinguish this Euler class from the others under consideration, here we call it the {\em Poincar\'e--Hopf Euler number}, because it is a sum of local indices as in the Poincar\'e--Hopf theorem for the Euler characteristic of a manifold. It is defined using local coordinates. 

Let $k$ be a field, and let $X$ be an $n$-dimensional smooth $k$-scheme. Let $z$ be a closed point of $X$.

\begin{definition} \label{def:NiscoordinatesBasek}
(cf. \cite[Definition 17]{CubicSurface}) By a system of \emph{Nisnevich coordinates around $z$} we mean a Zariski open neighborhood $U$ of $z$ in $X$, and an \'etale map $\varphi: U \to \A^n_k$ such that the extension of residue fields $k(\varphi(z)) \subseteq k(z)$ is an isomorphism.
\end{definition}

\begin{proposition}
When $n>0$, there exists a system of Nisnevich coordinates around every closed point $z$ of $X$.
\end{proposition}

\begin{proof}
When $k$ is infinite, this follows from \cite[Chapter 8. Proposition 3.2.1]{knus}. When $k(z)/k$ is separable, for instance when $k$ is finite, this is \cite[Lemma 18]{CubicSurface}.
\end{proof}

As above, let $V$ be a relatively oriented, rank $n$ vector bundle on $X$. Let $\sigma$ be a section with only isolated zeros, and let $Z \hookrightarrow X$ denote the closed subscheme given by the zero locus of $\sigma$. Let $z$ be a point of $Z$. The Poincar\'e--Hopf local index or degree $$ \ind_z^\PH \sigma \in \GW(k) $$ was defined in \cite[Definition 30]{CubicSurface} as follows. Choose a system of Nisnevich coordinates $\varphi: U \to \A^n_k$ around $z$. After possibly shrinking $U$, the restriction of $V$ to $U$ is trivial and we may choose an isomorphism $\psi: V \vert_{U} \to \calO_{U}^n$ of $V$. The local trivialization $\psi$ induces a distinguished section of $\det V(U)$. The system of local coordinates $\varphi$ induces a distinguished section of $\det T_X (U)$, and we therefore have a distinguished section of $(\det V \otimes \omega_X)(U)$. As in \cite[Definition 19]{CubicSurface}, the local coordinates $\varpi$ and local trivialization $\psi$ are said to be {\em compatible with the relative orientation} if the distinguished element is the tensor square of a section of $\scr L(U)$. By multiplying $\psi$ by a section in $\calO(U)$, we may assume this compatibility. 

Under $\psi$, the section $\sigma$ can be identified with an $n$-tuple $(f_1, \ldots, f_n)$ of regular functions, $\psi(\sigma) = (f_1, \ldots, f_n) \in \oplus_{i=1}^n \mathcal{O}_{U}$. Let $m$ denote the maximal ideal of $\mathcal{O}_{U}$ corresponding to $z$. Since $z$ is an isolated zero, there is an integer $n$ such that $m^n=0$ in $\mathcal{O}_{Z,z}$. For any $N$, it is possible to choose $(g_1, \ldots, g_n)$ in $\oplus_{i=1}^n m^{N}$ such that $(f_i + g_i)\vert_{U}$ is in the image of $\varphi^*: \mathcal{O}_{\A^r_S} \to \varphi_* \mathcal{O}_U$ after possibly shrinking $U$ \cite[Lemma 22]{CubicSurface}. For $N=2n$, choose such $(g_1, \ldots, g_n)$ and let $F_i$ in  $\mathcal{O}_{\A^d_k}(\A^d_k)$ be the functions such that $\varphi^*(F_i) = f_i + g_i$. Then $\varphi$ induces an isomorphism $\calO_{Z,z} \cong k[t_1, \dots, t_n]_{m_{\varphi(z)}}/(F_1, \dots, F_n)$ \cite[Lemma 25]{CubicSurface}, and $\ind_z^\PH \sigma$ is defined to be the associated form of Scheja--Storch $\SSform(\varphi(U), F, k)$ (see \S\ref{Scheja_Storch_Serre_Duality_section} and Definition \ref{def:SSform} for the definition of $\SSform(\varphi(U), F, k)$). The local index $\ind_z^\PH \sigma$ is well-defined by \cite[Lemma 26]{CubicSurface}. Then the Poincar\'e--Hopf Euler number is defined to be the sum of the local indices:

\begin{definition}\label{def:ePH}
The {\em Poincar\'e--Hopf Euler number} $n^{\PH}(V, \sigma)$ of $V$ with respect to $\sigma$ is $n^{\PH}(V, \sigma) = \sum_{z \text{ such that } \sigma(z) = 0} \ind_z^\PH \sigma$.
\end{definition}

\begin{proof}[Proof of Theorem \ref{thm:introd-ePH=eGS}]
By Proposition \ref{pr:eGSVsigmarho=eGSVrho} we have $n^{\GS}(V) = n^\GS(V, \sigma)$, where the orientation has been suppressed from the notation, but is indeed present. Using Formula \eqref{eq:eGS-formula}, it is thus enough to show that $\ind_z(\sigma) = \ind_z^\PH(\sigma)$.
This follows from Theorem \ref{thm:introd-2}.
One needs to be careful about the trivializations used in defining the various pushforward maps; this is ensured precisely by the condition that the tautological section is a square.
The details of this argument are spelled out more carefully in the proof in Proposition \ref{prop:meta-compute} in the next section.
\end{proof}

One can extend the comparison of local degrees $\ind_z^\PH \sigma = \ind_z \sigma$ to work over a more general base scheme $S$.
This was done for $S = \A^1_k$  with $k$ a field in \cite[Lemma 33]{CubicSurface}, but in more generality, it is useful to pick the local coordinates using knowledge of both $\sigma$ and $X$ as follows.

\begin{definition} \label{def:non-deg-section}
Let $X$ be a scheme, $V$ a vector bundle on $X$ and $\sigma$ a section of $V$.
\begin{enumerate}
\item We call $\sigma$ {\em non-degenerate} if it locally corresponds to a regular sequence.
\item Given another scheme $S$ and a morphism $\pi: X \to S$, we call $\sigma$ \emph{very non-degenerate} (with respect to $\pi$) if it non-degenerate and the zero locus $Z(\sigma)$ is finite and locally free over $S$.
\end{enumerate}
\end{definition}
\begin{remark} \label{rmk:non-degeneracy-crit}
Suppose that $X$ is smooth over $S$ and $rk(V) = \dim X/S$.
\begin{enumerate}
\item If $S=Spec(k)$ is the spectrum of a field, then $Z(\sigma) \to Spec(k)$ is quasi-finite if and only if it is finite locally free, if and only if $\sigma$ is locally given by a regular sequence.
  In other words $\sigma$ is non-degenerate if and only if it is very non-degenerate, if and only if $Z(\sigma) \to Z$ is quasi-finite.
\item In general, $\sigma$ is non-degenerate as soon as $Z(\sigma) \to S$ is quasi-finite, and very non-degenerate if and only if $Z(\sigma) \to S$ is finite.
  See Lemma \ref{lemm:automatic-regularity}.
\end{enumerate}
\end{remark}

\begin{example} \label{ex:cotangent-cx-regular-section}
If $\sigma$ is a non-degenerate section, then precomposition with $\sigma$ induces an isomorphism $\Hom(V, \scr O) \wequi C_{Z/X}$.
In particular $N_{Z/X} \wequi V$ and $L_{Z/X} \wequi V^\dual[1]$.
\end{example}

\begin{definition} \label{def:coordinates}
Let $X$ be a smooth $S$-scheme, and let $V \to X$ be a vector bundle, relatively oriented by $\rho$. Let $\sigma$ be a very non-degenerate section of $V$, and let $Z$ be a closed and open subscheme of the zero locus $Z(\sigma)$ of $\sigma$.
By a \emph{system of coordinates for $(V,X,\sigma,\rho,Z)$} we mean an open neighbourhood $U$ of $Z$ in $X$, an \'etale map $\varphi: U \to \A^n_S$, a trivialization $\psi: V\vert_U \wequi \scr O_U^n$, and a section $\sigma' \in \scr O_{\A_S^n}^n(\varphi(U))$, such that the following conditions hold:
\begin{enumerate}
\item\label{varphi_iso_Z} $Z = Z(\sigma|_U) \cong Z(\sigma')$,
\item $\det(\sigma|_Z) = \det(\varphi^* \sigma'|_Z) \in \det N_{Z/X}$, and
\item the canonical section of $\omega_{X/S} \otimes \det V|_Z \cong \scr L^{\otimes 2}|_Z$ determined by $\psi$ and $\varphi$ corresponds to the square of a section of $\scr L|_Z$.
\end{enumerate}
Here for (2) and (3) we used Example \ref{ex:cotangent-cx-regular-section}.
\end{definition}

Suppose that $X$ is dimension $n$ over $S$, so that the rank of $V$ is also $n$. Let $Z \subset Z(\sigma)$ be a clopen component and write $\varpi: Z \to S$ for the structure map.
The local index generalizes straightforwardly from Definition \ref{def:local-index-eGS}:
\begin{definition}
We call \[ \ind_{Z}(\sigma) = \ind_{Z}(V, \sigma, \rho) = \varpi_*(i^* \scr L \otimes i^* \scr L \to i^* \scr L^{\otimes 2}) \in \BL^\naive(S) \] the local index at $Z$.
\end{definition}
\begin{remark}
Since $\varpi$ is finite locally free, $\varpi_*$ preserves vector bundles.
In particular, $\ind_{Z}(\sigma) \in \BL^\naive(S)$ is a symmetric bilinear form on a vector bundle, as opposed to just on a complex up to homotopy. (See also Remark~\ref{rmk:non-sensible}.)
\end{remark}

A system of coordinates for $(V,X,\sigma,\rho,Z)$ determines a presentation $Z=Z(\sigma|_U)=Z(\sigma') \subset \A^n_S$, where $\sigma': \A^n_S \supset \phi(U) \to \A^n$.
Hence Definition \ref{def:SSform} supplies us with a symmetric bilinear form $\SSform(\varphi(U), \sigma, S) \in \BL^\naive(S)$.

\begin{proposition}
The form $\SSform(\varphi(U), \sigma', S)$ coincides (up to isomorphism) with $\ind_Z(\sigma)$.
In particular its isomorphism class is independent of the choice of coordinates.
\end{proposition}
\begin{proof}
The argument is the same as in the proofs of Theorem \ref{thm:introd-ePH=eGS} and Proposition \ref{prop:meta-compute}.
\end{proof}

In contrast, our proof that the Euler number (sum of indices) is independent of the choice of section (i.e. Proposition \ref{pr:eGSVsigmarho=eGSVrho}) does not generalize immediately; in fact this will not hold in $\BL^\naive(S)$ but rather in some quotient (like $\GW(k)$ in the case of fields).
As indicated in Remark \ref{rmk:eGSVsigma-alternative}, one situation in which it is easy to see this independence is if the quotient group satisfies homotopy invariance.
This suggests studying Euler numbers valued in more general homotopy invariant cohomology theories for algebraic varieties, which is what the remainder of this work is concerned with.

\section{Cohomology theories for schemes}
\label{sec:coh-theories}
\subsection{Introduction}
\label{subsec:coh-theories}
In order to generalize the results from the previous sections, we find it useful to introduce the concept of a \emph{cohomology theory twisted by $K$-theory}.
We do not seek here to axiomatize all the relevant data, but just introduce a common language for similar phenomena.

\begin{definition} \label{def:C-L}
Let $S$ be a scheme and $\scr C \subset \Sch_S$ a category of schemes.
Denote by $\scr C^L$ the category of pairs $(X, \xi)$ where $X \in \scr C$ and $\xi \in K(X)$ (i.e. a point in the $K$-theory space of $X$), and morphisms those maps of schemes compatible with the $K$-theory points.\footnote{Technically speaking, this means ``coherently compatible'', so $\scr C^L$ is an $\infty$-category. However we will only need its homotopy $1$-category, so for us compatible means ``together with a homotopy class of paths joining the two $K$-theory points''.}
By a \emph{cohomology theory} $E$ over $S$ (for schemes in $\scr C$) we mean a presheaf of sets on $\scr C^L$, i.e. a functor \[ E: (\scr C^L)^\op \to \cat{S}\mathrm{et}, (X, \xi) \mapsto E^\xi(X). \]
\end{definition}

To illustrate the flavor of cohomology theory we have in mind, we begin with two examples.
\begin{example} \label{ex:first-examples}
We can put
\begin{enumerate}
\item $E^\xi(X) = \CH^{rk(\xi)}(X)$, the Chow group of algebraic cycles up to rational equivalence of the appropriate codimension, or
\item $E^\xi(X) = \GW(X, \widetilde\det\, \xi)$, the Grothendieck-Witt group of symmetric bilinear perfect complexes for the duality $\iHom(\ph, \widetilde\det\, \xi)$ (see e.g. \cite{schlichting2010mayer}).
\end{enumerate}
\end{example}

\begin{warning} \label{warn:autom}
For cohomology theories with values in a $1$-category (like sets), in the above definition we can safely replace $K(X)$ by its trunctation $K(X)_{\le 1}$, i.e. the ordinary $1$-groupoid of virtual vector bundles.
However, we can in general \emph{not} replace it by just the set $K_0(X)$.
In other words, if (say) $V$ is a vector bundle on $X$ and $\phi$ an automorphism of $V$, then there is an induced automorphism \[ E(\phi): E^{V}(X) \to E^{V}(X) \] which may or may not be trivial.
For example, in the case $E = \GW$ as above, if $V = \scr O$ is trivial and $\phi$ corresponds to $a \in \scr O^\times(X)$, then $E(\phi)$ is given by multiplication by $\lra{a} \in \GW(X)$.
\end{warning}

\subsection{Features of cohomology theories}
\label{subsec:coh-features}
Many cohomology theories that occur in practice satisfy additional properties beyond the basic ones of the above definition, and many come with more data.
We list here some of those relevant to the current paper.
\begin{description}
\item[Morphisms of theories] Cohomology theories form a category in an evident way, with morphisms given by natural transformations.
\item[Trivial bundles] We usually abbreviate $E^{\scr O^n}(X)$ to $E^n(X)$.
\item[Additive and multiplicative structure] Often, $E$ takes values in abelian groups.
  Moreover, often $E^0(X)$ is a ring and $E^\xi(X)$ is a module over $E^0(X)$.
  Typically all of this structure is preserved by the pullback maps.
\item[Disjoint unions] Usually $E$ converts finite disjoint unions into products, i.e. $E(\emptyset) = *$ and $E(X \coprod Y) = E(X) \times E(Y)$.
  If $E$ takes values in abelian groups, this is usually written as $E(X \coprod Y) = E(X) \oplus E(Y)$.
\item[Orientations] In many cases, the cohomology theory $E$ factors through a quotient of the category $\scr C^L$, built using a quotient $q: K(X) \to K'(X)$ of the $K$-theory groupoid.
  In other words, one has canonical isomorphisms $E^\xi(X) \wequi E^{\xi'}(X)$ for certain $K$-theory points $\xi, \xi'$.
  More specifically:
\item[$\GL$-orientations] In the above situation, if $K'(X) = \Z$ and $q$ is the rank map, then we speak of a $\GL$-orientation.
  In other words, in this situation we canonically have $E^\xi(X) \wequi E^{rk(\xi)}(X)$.
  In particular, Warning \ref{warn:autom} does not apply: all automorphisms of vector bundles act trivially on $E$.
  This happens for example if $E = \mathrm{CH}$ (see Example \ref{ex:first-examples}(1)).
\item[$\SL$-orientations] If instead $K'(X) = Pic(D(X))$ via the determinant, then we speak of an $\SL$-orientation.
  In other words, in this situation $E^\xi(X)$ only depends on the rank and (ungraded) determinant of $\xi$.
  We write $E^{rk(\xi)}(X, \det(\xi))$ for this common group.
  This happens for example if $E=\GW$ (see Example \ref{ex:first-examples}(1)).
\item[$\SL^c$-orientations] This is a further strengthening of the concept of an $\SL$-orientation, where in $K'(X) = Pic(D(X))$ we mod out (in the sense of groupoids) by the squares of line bundles.
  In other words, if $\scr L_1, \scr L_2, \scr L_3$ are line bundles on $X$, then any isomorphism $\scr L_1 \wequi \scr L_2 \otimes \scr L_3^{\otimes 2}$ induces \[ E^n(X, \scr L_1) \wequi E^n(X, \scr L_2). \]
  Note that in particular then $E^n(X, \scr L) \wequi E^n(X, \scr L^\dual)$.
  This also happens for $E = \GW$, essentially by construction.
\item[Supports] Often, for $Z \subset X$ closed there is a cohomology with support, denoted $E^\xi_Z(X)$.
  It enjoys further functorialities which we do not list in detail here.
\item[Transfers] In many theories, for appropriate morphisms $p: X \to Y$ and $\xi \in K(Y)$, there exists $tw(p, \xi) \in K(X)$ and a transfer map \[ p_*: E^{tw(p, \xi)}(X) \to E^\xi(Y), \] compatible with composition.
  Typically $p$ is required to be lci, and \[ tw(p, \xi) = p^*\xi + L_p, \] where $L_p$ is the \emph{cotangent complex} \cite{illusie1971complexe}.
  Furthermore, typically $p$ is required to be proper, or else we need to fix $Z \subset X$ closed and proper over $Y$ and obtain $p_*: E^{tw(p, \xi)}_Z(X) \to E^\xi(Y)$.
  Finally, usually $E$ takes values in abelian groups and satisfies the disjoint union property, and transfer from a disjoint union is just the sum of the transfers.
\end{description}

\begin{remark}
\begin{enumerate}
\item We have defined a morphism of cohomology theories as a natural transformation of functors valued in sets.
  Whether or not such a transformation respects additional structure (abelian group structures, orientations, transfers etc.) must be investigated in each case.
\item In many cases (in particular in the presence of homotopy invariance), $\SL$-oriented theories are also canonically $\SL^c$-oriented.
  See Proposition \ref{prop:SL-oriented-trick}.
\end{enumerate}
\end{remark}

\subsection{Some cohomology theories}
\label{subsec:13}
We now introduce a number of cohomology theories that can be used in this context.
\begin{description}
\item[Hermitian $K$-theory $\GW$] This is the theory from Example \ref{ex:first-examples}(2).
  It is $\SL^c$-oriented.
  We believe that it has transfers for (at least) smooth proper morphisms and regular immersions, but we are not aware of a reference for this in adequate generality.
  If $X$ is regular and $1/2 \in X$ one can use the comparison with $\KO$-theory (see below).

\item[Naive derived bilinear forms $\BL_\naive$] See \S\ref{subsec:local-indices}.

\item[Cohomology theories represented by motivic spectra] Let $\SH(S)$ denote the motivic stable $\infty$-category.
  Then any $E \in \SH(S)$ defines a cohomology theory on $\Sch_S$, automatically satisfying many good properties; for example they always have transfers along smooth and proper morphisms, as well as regular immersions.
  For a lucid introduction, see \cite{EHKSY2}.
  We recall some of the main points in \S\ref{sec:representable-cohomology}.

\item[Orthogonal  $K$-theory spectrum $\KO$]
  This spectrum is defined and stable under arbitrary base change if $1/2 \in S$ \cite{panin2010motivic,schlichting2015geometric}.
  Over regular bases, it represents Hermitian $K$-theory $\GW$; in general it represents a homotopy invariant version.

\item[Generalized motivic cohomology $\H\tilde\Z$]
  This can be defined as $\pi_0^{\mathrm{eff}}(\1)$; see for example \cite{bachmann-very-effective}.
  Over fields (of characteristic not $2$) it represents generalized motivic cohomology in the sense of Calm\`es-Fasel \cite{calmes2014finite,calmes2017comparison,bachmann-criterion}; it is unclear if this theory is useful in this form over more general bases.
\end{description}

\subsection{The yoga of Euler numbers}
\label{subsec:yoga}
Let $E$ be a cohomology theory.
\begin{definition}
We will say that $E$ has Euler classes if, for each scheme $X$ over $S$ and each vector bundle $V$ on $X$ we are supplied with a class \[ e(V, E) \in E^{V^\dual}(X). \]
\end{definition}
\begin{remark}
The twist by $V^\dual$ (instead of $V$) in the above definition may seem peculiar.
This ultimately comes from our choice of covariant (instead of contravariant) equivalence between locally free sheaves and vector bundles, whereas a contravariant equivalence is used in the motivic Thom spectrum functor and hence in the definition of twists.
\end{remark}
Typically, the Euler classes will satisfy further properties, such as stability under base change; we do not formalize this here.

Now suppose that $\pi: X \to S$ is smooth and proper, $V$ is relatively oriented, $E$ has transfers for smooth proper maps and is $\SL^c$-oriented.
In this case we have a transfer map \[ E^{V^\dual}(X) \wequi E^n(X, \det V) \stackrel{\rho}{\wequi} E^n(X, \omega_{X/S}) \wequi E^{L_\pi}(X) \xrightarrow{\pi_*} E^0(S). \]
\begin{definition}
In the above situation, we call \[ n(V, \rho, E) = \pi_* e(V, E) \] the \emph{Euler number} of $V$ in $E$ with respect to the relative orientation $\rho$.
\end{definition}

\begin{example}
Let $E = \GW$.
We can define a family of Euler classes by \[ e(V) = [K(V, 0)] \in \GW(X, \widetilde\det V) \wequi \GW^{V^\dual}(X); \] here we use the Koszul complex from \S\ref{subsec:serr-duality-euler-number}.
This depends initially on a choice of section, but we shall show that the Grothendieck-Witt class often does not.
In any case, here we chose the $0$-section for definiteness.
Assuming that $\GW$ has transfers (of the expected form) in this context, we find that \[ n(V, \rho, \GW) = n^{\GS}(V, 0, \rho). \]
\end{example}

Now let $\sigma$ be a non-degenerate section of $V$ (in the sense of Definition \ref{def:non-deg-section}) and write $i: Z=Z(\sigma) \hookrightarrow X$ for the inclusion of the zero-scheme.
Thus $i$ is a regular immersion.
In this case one has (see Example \ref{ex:cotangent-cx-regular-section}) \[ [L_i] \wequi -[N_{Z/X}^\dual] \wequi -[V^\dual|_Z] \] and consequently, if $E$ has pushforwards along regular immersions, there is a transfer map \[ i_*: E^0(Z) \wequi E^{[V^\dual]|_Z - [V^\dual|_Z]}(Z) \wequi E^{i^*V^\dual + L_i}(Z) \to E^{V^\dual}(X). \]
The following result is true in all cases that we know of; but of course it cannot be proved from the weak axioms that we have listed.
\begin{meta-theorem} \label{thm:meta}
Let $\sigma$ be a non-degenerate section of a vector bundle $V$ over a scheme $X$.
Let $E$ be a cohomology theory with Euler classes and pushforwards along regular immersions, such that $E^0(S)$ has a distinguished element $1$ (e.g. is a ring).
Then \[ e(V, E) = i_*(1), \] where we use the identification above for the pushforward.
\end{meta-theorem}

Going back to the situation where $X$ is smooth and proper over $S$, $V$ is relatively oriented and $E$ is $\SL^c$-oriented and has transfers along proper lci morphisms, we also have the pushforward \[ \varpi_*: E^0(Z) \wequi E^0(Z, \scr L^{\otimes 2}|_Z) \stackrel{\rho}{\wequi} E^0(Z, \det V^\dual \otimes \omega_{X/S}|_Z) \wequi E^{L_\varpi}(Z) \to E^0(S). \]
More generally, if $Z' \subset Z$ is a clopen component, then we have a similar transfer originating from $E^0(Z')$.
\begin{definition} \label{def:index-general}
For $V, \sigma, X, E$ as above, for any clopen component $Z' \subset Z$ we denote by \[ \ind_{Z'}(\sigma, \rho, E) = \varpi'_*(1) \in E^0(S) \] the \emph{local index of $\sigma$ around $Z'$ in $E$}.
Here $\varpi': Z' \to S$ is the restriction of $\varpi$ to $Z'$.
\end{definition}

\begin{meta-corollary} \label{cor:meta}
Let $\sigma$ be a non-degenerate section of a relatively oriented vector bundle $V$ over $\pi: X \to S$.
Let $E$ be an $\SL^c$-oriented cohomology theory with Euler classes and pushforwards along proper lci morphisms, such that Meta-Theorem \ref{thm:meta} applies.
Then \[ n(V, \rho, E) = \sum_{Z'} \ind_{Z'}(\sigma, \rho, E). \]
\end{meta-corollary}
\begin{proof}
By assumption, transfers are compatible with composition, and additive along disjoint unions.
The result follows.
\end{proof}

\begin{example}
If $S = Spec(k)$ is the spectrum of a field, then $Z$ is zero-dimensional, and hence decomposes into a finite disjoint union of ``fat points''.
In particular, the Euler number is expressed as a sum of local indices, in bijection with the zeros of our non-degenerate section.
\end{example}

Recall the notion of coordinates from Definition \ref{def:coordinates}.
The following result states that indices may be computed in local coordinates.
\begin{proposition} \label{prop:meta-compute} 
Let $E$ be an $\SL^c$-oriented cohomology theory with Euler classes and pushforwards along proper lci morphisms.
Let $(\psi,\varphi,\sigma_2)$ be a system of coordinates for $(V,X,\sigma_1,\rho_1,Z)$.
Then \[ \ind_Z(\sigma_1, \rho_1, E) = \ind_Z(\sigma_2, \rho_2, E), \] where $\rho_2$ is the canonical relative orientation of $\scr O_{\A^n}^n/\A^n$.
\end{proposition}
\begin{proof}
Let $\varpi: Z \to S$ denote the canonical map.
Then both sides are obtained as $\varpi_*(1)$, but conceivably the orientations used to define the transfer could be different; we shall show that they are not.
In other words, we are given two isomorphisms \[ \widetilde \det L_\varpi \stackrel{\alpha_i}{\wequi} \scr L_i^{\otimes 2} \] and we need to exhibit $\scr L_1 \stackrel{\beta}{\wequi} \scr L_2$ such that $\alpha_1^{-1} \alpha_2 = \beta^{\otimes 2}$.
The isomorphisms $\alpha_i$ arise as \[ \det L_\varpi \wequi \det N_{Z/X} \otimes \omega_{X/S}|_Z \stackrel{\sigma_i}{\wequi} \det V \otimes \omega_{X/S}|_Z \stackrel{\rho_i}{\wequi} \scr L_i^{\otimes 2}|_Z; \] here $\scr L_2 = \scr O$, and for $i=2$ we implicitly use $\varphi$ and $\psi$ as well.
We first check that the two isomorphisms $\det N_{Z/X} \wequi \det V|_Z$ are the same.
Indeed $V|_U \wequi \scr O^n$ via $\psi$, and up to this trivialization the isomorphism is given by the trivialization of $C_{Z/X}$ by $\sigma_i$; these are the same by assumption (2).
Now we deal with the second half of the isomorphism.
By construction we have an isomorphism $\scr O \wequi \scr O^{\otimes 2} \stackrel{\alpha_1^{-1}\alpha_2}{\wequi} \scr L_1^{\otimes 2}$; what we need to check is the corresponding global section of $\scr L_1^{\otimes 2}$ is a tensor square.
Unwinding the definitions, this follows from assumption (3).
\end{proof}

\section{Cohomology theories represented by motivic spectra}
\label{sec:representable-cohomology}
We recall some background material about motivic extraordinary cohomology theories, i.e. theories represented by motivic spectra.
We make essentially no claim to originality.
\subsection{Aspects of the six functors formalism}
We recall some aspects of the six functors formalism for the motivic stable categories $\SH(\ph)$, following the exposition in \cite{EHKSY2}.
\subsubsection{Adjunctions}
For every scheme $X$, we have a symmetric monoidal, stable category $\SH(X)$.
For every morphism $f: X \to Y$ of schemes we have an adjunction \[ f^*: \SH(Y) \adj \SH(X): f_*. \]
If no confusion can arise, we sometimes write $E_Y := f^* E$.
If $f$ is smooth, there is a further adjunction \[ f_\#: \SH(X) \adj \SH(Y): f^*. \]
If $f$ is locally of finite type, then there is the exceptional adjunction \[ f_!: \SH(X) \adj \SH(Y): f^!. \]
There is a natural transformation $\alpha: f_! \to f_*$
If $f$ is proper, then $\alpha$ is an equivalence.

The assignments $f \mapsto f^*, f_*, f^!, f_!, f_\#$ are functorial.
In particular, given composable morphisms $f, g$ of the appropriate type, we have equivalences $(fg)_* \wequi f_*g_*$, and so on.

\subsubsection{Exchange transformations}
Suppose given a commutative square of categories
\begin{equation*}
\begin{CD}
\scr C @<F<< \scr D \\
@AGAA        @AG'AA \\
\scr C' @<F'<< \scr D', \\
\end{CD}
\end{equation*}
i.e. a natural isomorphism $\gamma: FG' \wequi GF'$.
If the functors $G', G$ have right adjoints $H', H$, then we have the natural transformation \[ F' H' \xrightarrow{\unit}  HG F'H' \stackrel{\gamma}{\wequi}  HFG'H' \xrightarrow{\counit} HF \] called the associated \emph{exchange transformation}.
Similarly if $G', G$ have left adjoints $K', K$, then we have the exchange transformation \[ KF \xrightarrow{\unit} KF G'K' \stackrel{\gamma}{\wequi} KGFK' \xrightarrow{\counit} FK'. \]

Suppose we have a commutative square of schemes
\begin{equation} \label{eq:square-of-schemes}
\begin{CD}
X' @>f'>> Y' \\
@Vg'VV  @VgVV \\
X @>f>> Y. \\
\end{CD}
\end{equation}
Then we have an induced commutative square of categories 
\begin{equation*}
\begin{CD}
\SH(X') @<f'^*<< \SH(Y') \\
@Ag'^*AA         @Ag^*AA \\
\SH(X) @<f^*<< \SH(Y). \\
\end{CD}
\end{equation*}
Passing to the right adjoints, we obtain the exchange transformation \[ \Ex^*_*: f^*g_* \to g'_* f'^*. \]
Similarly there is $\Ex_\#^*: g'_\# f'^* \to f^*g_\#$ (for $g$ smooth; this is in fact an equivalence if \eqref{eq:square-of-schemes} is cartesian), and so on.

\subsubsection{Exceptional exchange transformation}
Given a cartesian square of schemes as in \eqref{eq:square-of-schemes}, with $g$ (and hence $g'$) locally of finite type, there is a canonical equivalence \[ \Ex^*_!: f^*g_! \wequi g'_! f'^*. \] Passing to right adjoints, we obtain \[ \Ex^{*!}: f'^*g^! \to g'^!f^*. \]

\subsubsection{Thom transformation}
Given a perfect complex $\scr E$ of vector bundles on $X$, the motivic J-homomorphism $K(X) \to Pic(\SH(X))$ \cite[\S16.2]{bachmann-norms} provides us with an invertible spectrum $\Sigma^\scr{E} \1 \in \SH(X)$.
We denote by $\Sigma^{\scr E}: \SH(X) \to \SH(X), E \mapsto E \wedge \Sigma^\scr{E} \1$ the associated invertible endofunctor.
If $\scr E$ is a vector bundle (concentrated in degree zero), then $\Sigma^\scr{E} \1$ is the suspension spectrum on the Thom space $\scr E^\dual/\scr E^\dual \setminus 0$.\footnote{Indeed, in \cite[\S2.1.2]{EHKSY2}, the transformation $\Sigma^{\scr E}$ is built out of $Spec(Sym(\scr E))$---which is the vector bundle corresponding to $\scr E^\dual$, in our convention.}

\begin{lemma} \label{lemm:shriek-thom}
The functor $f^!$ commutes with Thom transforms.
\end{lemma}
\begin{proof}
This follows from the projection formula \cite[A.5.1(6)]{triangulated-mixed-motives} and invertibility of $\Sigma^{(\ph)}\1$: \[ f^!(\Sigma^{V} X) \wequi f^! \iHom(\Sigma^{-V}\1, X) \wequi \iHom(f^* \Sigma^{-V} \1, f^!X) \wequi \Sigma^{f^* V} f^! X. \]
\end{proof}

\subsubsection{Purity transformation}
Let $f: X \to Y$ be a smoothable \cite[\S2.1.21]{EHKSY2}  lci morphism.
Then the cotangent complex $\L_f$ is perfect, and there exists a canonical \emph{purity transformation} \[ \purity_f: \Sigma^{\L_f} f^* \to f^!. \]

\subsection{Cohomology groups and Gysin maps}
\label{subsec:gysin}
\subsubsection{}
Let $S$ be a scheme and $E \in \SH(S)$. Given $(\pi: X \to S) \in \Sch_S$, $i: Z \hookrightarrow X$ closed, $\xi \in K(Z)$, we define the \emph{$\xi$-twisted $E$-cohomology of $X$ with support in $Z$} as \[ E_Z^\xi(X) = [\1, \pi_* i_! \Sigma^\xi i^! \pi^* E]_{\SH(S)}. \]
This assignment forms a cohomology theory in the sense of \S\ref{subsec:coh-theories}.
It takes values in abelian groups, has supports, and satisfies the disjoint union property.
We shall see that it has transfers for proper lci maps.
It need not be orientable in general.

If $Z = X$, we may omit it from the notation and just write $E^\xi(X)$. As before if $\xi$ is a trivial virtual vector bundle of rank $n \in \Z$ then we also write $E^n_Z(X)$ instead of $E^\xi_Z(X)$.

\begin{example} \label{ex:cohomology-with-support}
Suppose that $\xi = i^*V$, where $V$ is a vector bundle on $X$.
We have \[ E_Z^\xi(X) = [\1, \pi_* i_! \Sigma^\xi i^! \pi^* E]_{\SH(S)} \wequi [\1, i^! \Sigma^\xi \pi^* E]_{\SH(Z)} \wequi [i_* \1,  \Sigma^\xi \pi^* E]_{\SH(X)}, \]
where we have used that $i_* \wequi i_!$ and Lemma \ref{lemm:shriek-thom}.
Using the localization sequence $j_\#j^* \to \id \to i_*i^*$ \cite[Theorem 6.18(4)]{hoyois-equivariant} to identify $i_* \1 \wequi X/X \setminus Z$ we find that \[ E_Z^\xi(X) \wequi [ X/(X \setminus Z), \Sigma^\xi \pi^* E]_{\SH(X)}. \]
\end{example}

\begin{remark}
The final expression in the above example only depends on $Z \subset X$ as a subset, not subscheme.
This also follows directly from the definition, since $\SH(Z) \wequi \SH(Z_\red)$; this is another consequence of localization.
\end{remark}

\subsubsection{Functoriality in $E$} \label{subsub:functoriality-E}
If $\alpha: E \to F \in \SH(S)$ is any morphism, then there is an induced morphism $\alpha_*: E^\xi_Z(X) \to F^\xi_Z(X)$.
This just follows from the fact that $\pi^*$ etc. are functors.

\subsubsection{Contravariant functoriality in $X$}
Let $f: X' \to X \in \Sch_S$.
Then there is a pullback map \[ f^*: E^\xi_Z(X) \to E^{f^* \xi}_{f^{-1}(Z)}(X'), \] coming from the morphism
\begin{gather*}
  \pi_* i_! \Sigma^\xi i^! \pi^* E \xrightarrow{\unit} \pi_* f_*f^* i_! \Sigma^\xi i^! \pi^* E \wequi \pi'_* f^* i_! \Sigma^\xi i^! \pi^* E
  \stackrel{\Ex_!^*}{\wequi} \pi'_* i'_! f'^* \Sigma^\xi i^! \pi^* E \\
  \wequi \pi'_* i'_! \Sigma^{f^*\xi} f'^* i^! \pi^* E \xrightarrow{\Ex^{*!}} \pi'_* i'_! \Sigma^{f^*\xi} i'^* f^* \pi^* E \wequi \pi'_* i'_! \Sigma^{f^*\xi} i'^* \pi'^* E.
\end{gather*}
Here the $\Ex_!^*$ and $\Ex^{*!}$ come from the cartesian square
\begin{equation*}
\begin{CD}
Z' @>f'>>    Z \\
@Vi'VV       @ViVV \\
X' @>f>>     X \\
\end{CD}
\end{equation*}

\begin{lemma}
Let $f: X' \to X \in \Sch_X$, $Z \subset X$, $\xi \in K(Z)$ and $\alpha: E \to F \in \SH(S)$.
The following square commutes
\begin{equation*}
\begin{CD}
E_{Z'}^{\xi'}(X') @>{\alpha_*}>> F_{Z'}^{\xi'}(X') \\
@A{f^*}AA                        @A{f^*}AA         \\
E_{Z}^{\xi}(X) @>{\alpha_*}>> F_{Z}^{\xi}(X'), \\
\end{CD}
\end{equation*}
where $\xi' = f^*(\xi)$ and $Z' = f^{-1}(Z)$.
\end{lemma}
\begin{proof}
This is just an expression of the fact that the exchange transformations used to build $f^*$ are indeed natural transformations.
\end{proof}

\subsubsection{Covariant functoriality in $X$} \label{sec:gysin}
Suppose given a commutative square in $\Sch_S$
\begin{equation} \label{eq:gysin-square}
\begin{CD}
Z_1 @>i>> X    \\
@VgVV    @VfVV \\
Z_2 @>k>> Y,   \\
\end{CD}
\end{equation}
where $f$ is smoothable lci, $i, k$ are closed immersions and $g$ is proper.
For every $\xi \in K(Z_2)$, there is the \emph{Gysin map} \[ f_*:E_{Z_1}^{g^* \xi + i^* \L_f}(X) \to E_{Z_2}^\xi(Y) \]
coming from the morphism
\[ f_* i_! \Sigma^{g^*\xi + i^* \L_f} i^! f^* E_Y \xrightarrow{\purity_f} f_* i_! \Sigma^{g^* \xi} i^! f^! E_Y \wequi k_! g_! \Sigma^{g^* \xi} g^!k^! E_Y \wequi k_! g_! g^!\Sigma^{\xi} k^! E_Y \xrightarrow{\counit} k_! \Sigma^\xi k^! E_Y, \]
where we used Lemma \ref{lemm:shriek-thom} to move $\Sigma^{L_f}$ through $i^!$ and $\Sigma^\xi$ through $g^!$, and also used $i_* \wequi i_!, k_* \wequi k_!, g_* \wequi g_!$.
\begin{remark}
In \cite{EHKSY2}, the Gysin map is denoted by $f_!$, to emphasize that it involves the purity transform.
We find the notation $f_*$ more convenient.
\end{remark}

\begin{lemma} \label{lemm:gysin-compat}
The following hold.
\begin{enumerate}
\item Suppose given a square as in \eqref{eq:gysin-square} and $\alpha: E \to F \in \SH(S)$.
  Then the following square commutes
  \begin{equation*}
  \begin{CD}
  E_{Z_1}^{g^* \xi + i^* \L_f}(X) @>{\alpha_*}>> F_{Z_1}^{g^* \xi + i^* \L_f}(X) \\
  @V{f_*}VV                                      @V{f_*}VV                       \\
  E_{Z_2}^\xi(Y)                  @>{\alpha_*}>> F_{Z_2}^\xi(Y).
  \end{CD}
  \end{equation*}

\item Suppose given a square as in \eqref{eq:gysin-square} and $s: Y' \to Y$ such that $s, f$ are tor-independent.
  Let $X' = X \times_Y Y'$, $Z_i' = Z_i \times_Y Y'$, $i', g', k', f'$ the induced maps, and so on.
  Then the following square commutes
  \begin{equation*}
  \begin{CD}
  E_{Z_1}^{g^* \xi + i^* \L_f}(X) @>{s^*}>> E_{Z_1'}^{g'^* \xi' + i'^* \L_{f'}}(X) \\
  @V{f_*}VV                                      @V{f'_*}VV                       \\
  E_{Z_2}^\xi(Y)                  @>{s^*}>> E_{Z_2'}^{\xi'}(Y').
  \end{CD}
  \end{equation*}

\item Suppose given a commutative diagram in $\Sch_S$ as follows
  \begin{equation*}
  \begin{CD}
  Z_1 @>i>> X    \\
  @VgVV    @VfVV \\
  Z_2 @>k>> Y    \\
  @Vg'VV    @Vf'VV \\
  Z_3 @>l>> W.   \\
  \end{CD}
  \end{equation*}
  Then given $\xi \in K(Z_3)$, we have \[ f'_* f_* = (f'f)_*: E^{(g'g)^* \xi + i^* L_{f'f}}_{Z_1}(X) \to E_{Z_3}^\xi(W). \]
  Here we use the equivalence $L_{f'f} \wequi f'^* L_f + L_{f'} \in K(X)$, coming from the cofiber sequence $f'^* L_f \to L_{f'f} \to L_{f'}$.
\end{enumerate}
\end{lemma}
\begin{proof}
(1) Since $f_*: E_{Z_1}^{g^* \xi + i^* \L_f}(X) \to E_{Z_2}^\xi(Y)$ is obtained as $[1, \zeta_E]$ where $\zeta$ is a natural transformation of endofunctors of $\SH(S)$, this is clear.

(2) Consider the following diagram.
\begin{equation*}
\begin{tikzcd}
f_*i_!\Sigma^{g^*\xi+i^*L_f}i^!f^* \ar[r, "\unit_X"] \ar[d, "\wequi"]  &  f_*s_{X*}s_X^*i_!\Sigma^{g^*\xi+i^*L_f}i^!f^* \ar[r, "\Ex"] \ar[d, "\wequi"]  &  f_*s_{X*}\Sigma^{g'^*\xi'+i'^*L_{f'}}i'^!f'^*s^* \ar[d, "\wequi"] \\
f_*i_!\Sigma^{g^*\xi}i^!\Sigma^{L_f}f^* \ar[r, "\unit_X"] \ar[d, "\purity_f"]  &  f_*s_{X*}s_X^*i_!\Sigma^{g^*\xi}i^!\Sigma^{L_f}f^* \ar[r, "\Ex"] \ar[d, "\purity_f"]  &  f_*s_{X*}i'_!\Sigma^{g'^*\xi'}i'^!\Sigma^{L_{f'}}f'^*s^* \ar[d, "\purity_{f'}"] \ar[ld, phantom, "(a)"] \\
f_*i_!\Sigma^{g^*\xi}i^!f^! \ar[r, "\unit_X"] \ar[rd, "\unit"'] \ar[dd, "\wequi"]  &  f_*s_{X*}s_X^*i_!\Sigma^{g^*\xi}i^!f^! \ar[r, "\Ex"] \ar[d, phantom, "(b)"]  &  f_*s_{X*}i'_!\Sigma^{g'^*\xi'}i'^!f'^!s^* \ar[d, "\wequi"] \\
  &  s_*s^*f_*i_!\Sigma^{g^*\xi}i^!f^! \ar[r, "\Ex"] \ar[d, "\wequi"]  &  s_*f_*i'_!\Sigma^{g'^*\xi'}i'^!f'^!s^* \ar[d, "\wequi"] \\
k_!g_!\Sigma^{g^*\xi}g^!k^! \ar[r, "\unit"] \ar[d, "\wequi"]  &  s_*s^*k_!g_!\Sigma^{g^*\xi}g^!k^! \ar[d, "\wequi"] \ar[r, "\Ex"]  &  s_*k'_!g'_!\Sigma^{g'^*\xi'}g'^!k'^!s^* \ar[d, "\wequi"] \\
k_!g_!g^!\Sigma^{\xi}k^! \ar[r, "\unit"] \ar[d, "\counit"]  &  s_*s^*k_!g_!g^!\Sigma^{\xi}k^! \ar[d, "\counit"] \ar[r, "\Ex"]  &  s_*k'_!g'_!g'^!\Sigma^{\xi'}k'^!s^* \ar[d, "\counit'"] \ar[dl, phantom, "(b)"] \\
k_!\Sigma^{\xi}k^! \ar[r, "\unit"]  &  s_*s^*k_!\Sigma^{\xi}k^! \ar[r, "\Ex"]  &  s_*k'_!\Sigma^{\xi'}k'^!s^* \\
\end{tikzcd}
\end{equation*}
Here $s_X: X' \to X$ is the canonical map and $\xi' = (Z_2' \to Z_2)^* \xi$.
All the unlabelled equivalences arise from moving Thom transforms through (various) pullbacks, compatibility of pullbacks and pushforwards with composition, and equivalences of the form $p_* \wequi p_!$ for $p$ proper.
All the maps labelled $\Ex$ are exchange transformations expressing compatibility of $p_*, p_!, p^!$ with base change.
Denote the diagram by $\scr D$.
Then $[\1, \scr D E_Y]$ yields a diagram of abelian groups.
The outer square of that diagram identifies with the square we are trying to show commutes.
It hence suffices to show that $\scr D$ commutes.
All cells commute for trivial reasons, except for (a) which commutes by \cite[Proposition 2.2.2(ii)]{EHKSY2} and (b) which commute by stability of the counit transformations under base change.

(3) Consider the following diagram.
\begin{equation*}
\begin{tikzcd}
  &  f'_*f_* i_! \Sigma^{g^*g'^* \xi + i^*L_f + i^*f'^*L_{f'}} i^!f^*f'^* \ar[d, "\wequi"] \\
  & l_! g'_! g_! \Sigma^{g^*g'^*\xi} i^! \Sigma^{L_f}f^* \Sigma^{L_{f'}}f'^* \ar[r, "\wequi"] \ar[d, "\purity_f"] &  l_!(g'g)_! \Sigma^{(g'g)^* \xi} i^! \Sigma^{L_{f'f}} (f'f)^* \ar[dd, "\purity_{f'f}"] \ar[ddl, "(a)", phantom] \\
  &  l_! g'_! g_! \Sigma^{g^*g'^*\xi} i^!f^! \Sigma^{L_{f'}}f'^* \ar[ld, "\wequi"] \ar[d, "\purity_{f'}"] \\
l_! g'_! g_! g^! \Sigma^{g'^* \xi} k^! \Sigma^{L_{f'}} f'^* \ar[d, "\counit_g"] \ar[dr, "\purity_{f'}"]  &  l_! g'_! g_! \Sigma^{g^*g'^*\xi} i^!f^!f'^! \ar[r, "\wequi"] \ar[d, "\wequi"]  &  l_!(g'g)_!\Sigma^{(g'g)^*\xi}i^!(f'f)^! \ar[dd, "\wequi"] \\
l_! g'_! \Sigma^{g'^*\xi} k^! \Sigma^{L_{f'}} f'^* \ar[d, "\purity_{f'}"]  &  l_! g'_! g_! g^! \Sigma^{g'^*\xi} k^! f'^! \ar[dl, "\counit_{g}"] \ar[d, "\wequi"] \\
l_! g'_! \Sigma^{g'^*\xi} k^! f'^! \ar[d, "\wequi"]  &  l_! g'_! g_! g^! g'^! \Sigma^\xi l^! \ar[dl, "\counit_g"] \ar[r, "\wequi"]  &  l_!(g'g)_!(g'g)^! \Sigma^\xi l^! \ar[dl, "\counit_{g'g}"] \\
l_! g'_! g'^! \Sigma^\xi l^! \ar[r, "\counit_{g'}"']  &  l_! \Sigma^\xi l^! \ar[u, "(b)", phantom] \\
\end{tikzcd}
\end{equation*}
All the unlabelled equivalences arise from moving Thom transforms through (various) pullbacks, compatibility of pullbacks and pushforwards with composition, and equivalences of the form $p_* \wequi p_!$ for $p$ proper.
Denote the diagram by $\scr D$.
Then $[\1, \scr D E_W]$ yields a diagram of abelian groups.
Going from the top to the bottom middle via the leftmost past, we obtain $f'_* f_*$; going instead via the rightmost path we obtain $(f'f)_*$.
It hence suffices to show that $\scr D$ commutes.
All cells commute for trivial reasons, except for (a) which commutes by \cite[Proposition 2.2.2(i)]{EHKSY2}, and (b) which commutes by compatibility of the counit transformations with composition.
\end{proof}

\begin{example} \label{ex:forget-support-gysin}
Consider a commutative square as in \eqref{eq:gysin-square}, with $X=Y$ and $f = \id$, so that $g: Z_1 \hookrightarrow Z_2$ is a closed immersion.
Then $f_*: E^{g^* \xi}_{Z_1}(X) \to E^\xi_{Z_2}(X)$ is the ``extension of support'' map.
In particular taking $Z_2 = X$ as well, we obtain the map $E^{i^* \xi}_{Z_1}(X) \to E^\xi(X)$ ``forgetting the support''.
Lemma \ref{lemm:gysin-compat}(3) now in particular tells us that given a proper map $f: X \to Y$, a closed immersion $i: Z \hookrightarrow X$ and $\xi \in K(Y)$, the following diagram commutes
\begin{equation*}
\begin{CD}
E^{i^*f^* \xi}_Z(X) @>>> E^{f^* \xi}(X) \\
@V{f_*}VV                @V{f_*}VV      \\
E^\xi_X(X)          @=   E^\xi(X),
\end{CD}
\end{equation*}
where the upper horizontal map forgets the support.
\end{example}

\subsection{Orientations}
\subsubsection{Product structures}
By a ring spectrum (over $S$) we an object $E \in \SH(S)$ together with homotopy classes of maps $u: \1 \to E$ and $m: E \wedge E \to E$ satisfying the evident identities.
If $E$ is a ring spectrum, then there are multiplication maps \[ E^\xi_{Z_1}(X) \times E^{\xi'}_{Z_2}(X) \to E^{\xi + \xi'}_{Z_1 \cap Z_2}(X) \]  induced by \[ (\Sigma^\xi E) \wedge (\Sigma^{\xi'} E) \wequi \Sigma^{\xi + \xi'} E \wedge E \xrightarrow{m} \Sigma^{\xi + \xi'} E \] and the diagonal\footnote{One easily checks that the diagonal $X \to X \times X$ induces a map as indicated.} \[ X/X \setminus (Z_1 \cap Z_2) \to X/X \setminus Z_1 \wedge X/X \setminus Z_2. \]

\begin{lemma} \label{lemm:product-pullback-compat}
The multiplicative structure on $E$-cohomology is compatible with pullback: given $Z_1, Z_2 \subset X$, $\xi, \xi' \in K(X)$, $f: X' \to X$, the following diagram commutes
\begin{equation*}
\begin{CD}
E^\xi_{Z_1}(X) \times E^{\xi'}_{Z_2}(X) @>>> E^{\xi + \xi'}_{Z_1 \cap Z_2}(X) \\
@V{f^* \times f^*}VV                         @Vf^*VV \\
E^{f^*\xi}_{f^{-1}(Z_1)}(X') \times E^{f^*\xi'}_{f^{-1}(Z_2)}(X') @>>> E^{f^*\xi + f^*\xi'}_{f^{-1}(Z_1) \cap f^{-1}(Z_2)}(X') \\
\end{CD}
\end{equation*}
\end{lemma}
\begin{proof}
Immediate from the definitions.
\end{proof}

\subsubsection{Thom spectra} \label{sec:thom-spectra}
Let $G = (G_n)_n$ be a family of finitely presented $S$-group schemes, equipped with a morphism of associative algebras $G \to (\GL_{nk,S})_n$ (for the Day convolution symmetric monoidal structure on $\mathrm{Fun}(\N, \mathrm{Grp}(\Sch_S))$).
Then there is a notion of (stable) vector bundle with structure group $G$, the associated $K$-theory space $K^G(X)$, and the associated \emph{Thom spectrum} $\M G$, which is a ring spectrum \cite[Example 16.22]{bachmann-norms}.

\begin{example}
If $G_n = \GL_n$, then $K^G(X) = K(X)$ and $\M \GL$ is the algebraic cobordism spectrum \cite[Theorem 16.13]{bachmann-norms}.
If $G_n = \SL_n$ (respectively $\Sp_n$), then $K^G(X)$ is the $K$-theory of oriented (respectively symplectic) vector bundles in the usual sense, and $\M \SL$ (respectively $\M \Sp$) is the Thom spectrum as defined in \cite{panin2010algebraic}.
\end{example}

In order to work effectively with $\M G$, one needs to know that it is stable under base change.
This is easily seen to be true for  $\M\GL, \M\SL, \M\Sp$ \cite[Example 16.23]{bachmann-norms}.
We record the following more general result for future reference.
\begin{proposition} \label{prop:MSL-stable-under-base-change}
The Thom spectrum $\M G$ is stable under base change, provided that each $G_n$ is flat and quasi-affine.
\end{proposition}
\begin{proof}
We have a presheaf $K^G \in \PSh(\Sch_S)$ and a map $K^G \to K$.
For $f: X \to S \in \Sch_S$, denote by $K^G_X \in \PSh(\Sm_X)$ and $j_X: K^G_X \to K|_{\Sm_X}$ the restrictions.
Then by definition $\M G_X = M_X(j_X)$, where $M_X: \PSh(\Sm_X)_{/K} \to \SH(X)$ is the motivic Thom spectrum functor \cite[\S16.1]{bachmann-norms}.
Let $L K^G_S \in \PSh(\Sch_S)$ denote the left Kan extension of $K^G_S$.
We claim that $L K^G_S \to K^G$ is a Nisnevich equivalence.
Assuming this, we deduce that $f^* K^G_S \wequi (L K^G_S)|_{\Sm_X} \to K^G_X$ is a Nisnevich equivalence.
Since $M_X$ inverts Nisnevich equivalences \cite[Proposition 16.9]{bachmann-norms}, this implies that $f^* \M G_S \wequi \M G_X$, which is the desired result.

To prove the claim, we first note that by \cite[Lemma 3.3.9]{EHKSY3}, we may assume $S$ affine, and it suffices to prove that the restriction of $K^G$ to $\Aff_S$ is left Kan extended from smooth affine $S$-schemes.
By definition $K^G = (\Vect^G)^{gp}$, where $\Vect^G = \coprod_{n \ge 0} BG_n$ (here the coproduct is as stacks, i.e. fppf sheaves).
The desired result now follows from \cite[Proposition A.0.4 and Lemma A.0.5]{EHKSY3} (noting that the coproduct of stacks is the same as coproduct of $\Sigma$-presheaves, and Kan extension preserves $\Sigma$-presheaves).
\end{proof}

Now let $\xi \in K^G(X)$.
Then there is a canonical equivalence \cite[Example 16.29]{bachmann-norms} \[ \Sigma^\xi \M G_X \wequi \Sigma^{|\xi|} \M G_X. \]
We denote by $t_\xi \in \M G^{\xi - |\xi|}(X)$ the class of the map \[ \1 \xrightarrow{u} \M G_X \wequi \Sigma^{\xi - |\xi|} \M G_X. \]

\subsubsection{Oriented ring spectra}
\label{subsub:oriented-rings}
\begin{definition} \label{def:orientation}
Let $E \in \SH(S)$ be a ring spectrum and $G = (G_n)_n$ a family of group schemes as in \S\ref{sec:thom-spectra}.
By a \emph{strong $G$-orientation} of $E$ we mean a ring map $\M G \to E$.
\end{definition}

\begin{example}
The spectrum $\KO$ is strongly $\SL$-oriented; see Corollary \ref{cor:KO-orient}.
\end{example}

Note that if $E \in \SH(S)$ is strongly $G$-oriented, then there is no reason a priori why $E_X$ should be strongly $G_X$-oriented.
This is true if $\M G$ is stable under base change, so for most reasonable $G$ by Proposition \ref{prop:MSL-stable-under-base-change}.
We will not talk about strong $G$-orientations unless $\M G$ is stable under base change, so assume this throughout.

Given $\xi \in K^G(X)$, the map $\M G \to E$ provides us with $t_\xi = t_\xi(E) \in E^{\xi - |\xi|}(X)$.

\begin{proposition} \label{prop:thom-iso-E}
Let $E$ be strongly $G$-oriented and $\xi \in K^G(X)$.
\begin{enumerate}
\item The classes $t_\xi(E)$ are stable under base change: for $f: X' \to X$ we have $t_{f^*\xi}(E) = f^*t_\xi(E)$.
\item Multiplication by $t_\xi(E)$ induces an equivalence $\Sigma^{|\xi|} E \wequi \Sigma^\xi E$ and an isomorphism $t: E^{|\xi|}_Z(X) \wequi E^\xi_Z(X)$, called the \emph{Thom isomorphism}.
\item The Thom isomorphism is compatible with base change: $f^*(t(x)) = t(f^*(x))$.
\end{enumerate}
In particular, $E$ is $G$-oriented.
\end{proposition}
\begin{proof}

(1) follows from the same statement for $\M G$, where it holds by construction.
For the first half of (2), it suffices to show that $t_\xi(E)$ is a unit in the Picard-graded homotopy ring of $E$.
This follows from the same statement for $\M G$. The second half of (2) follows.
(3) immediately follows from (1).
\end{proof}

\begin{example}
Let $E$ be strongly $\GL$-oriented.
Then for any $\xi \in K(X)$ we obtain $E^\xi(X) \wequi E^{rk(\xi)}(X)$, so $E$ is oriented in the sense of \S\ref{subsec:coh-features}
\end{example}

\begin{definition}
For $X \in \Sch_S$ and $\scr L$ a line bundle on $X$ put \[ E^n_Z(X, \scr L) = E^{n-1+\scr L}_Z(X). \]
\end{definition}
\begin{example} \label{ex:SL-thom-iso}
Let $E$ be strongly $\SL$-oriented and $\xi \in K(X)$.
Then $\xi' := \xi - (|\xi| - 1 + \det\xi) \in K(X)$ lifts canonically to $K^\SL(X)$, whence by Proposition \ref{prop:thom-iso-E} we get a canonical (Thom) isomorphism $E^\xi(X) \wequi E^{|\xi|}(X, \det\xi)$.
In particular, $E$ is $\SL$-oriented in the sense of \S\ref{subsec:coh-features}.
\end{example}

\begin{remark}
If $E$ is strongly $\SL$-oriented, then since $\det(\scr L_1 \oplus \scr L_2) \wequi \scr L_1 \otimes \scr L_2$, by Example \ref{ex:SL-thom-iso} the product structure on $E$-cohomology twisted by line bundles takes the form $E^n(X, \scr L_1) \times E^m(X, \scr L_2) \to E^{n+m}(X, \scr L_1 \otimes \scr L_2)$.
\end{remark}

\begin{remark} \label{rmk:strong-orientation-permanence}
Strong $G$-orientations have better permanence properties than ordinary ones (provided that $\M G$ is stable under base change): they are stable under base change and taking (very) effective covers, for example.
\end{remark}

\subsection{$\SL^c$-orientations}
A. Ananyevskiy has done important work on $\SL$ and $\SL^c$ orientations.
We shall make use of the following result; see e.g. \cite[Theorem 1.1]{ananyevskiy2019sl}.
\begin{proposition}[Ananyevskiy] \label{prop:SL-oriented-trick}
Let $E \in \SH(S)$ be $\SL$-oriented and $\scr L_1, \scr L_2, \scr L_3$ be line bundles on $X$.
Suppose given an isomorphism $\scr L_1 \wequi \scr L_2 \otimes \scr L_3^{\otimes 2}$.
\begin{enumerate}
\item There is a canonical equivalence $\Sigma^{\scr L_1} E \wequi \Sigma^{\scr L_2} E$, compatible with base change.
\item There is a canonical isomorphism $E^n_Z(X, \scr L_1) \wequi E^n_Z(X, \scr L_2)$, compatible with base change.
\end{enumerate}
In particular, the cohomology theory represented by $E$ is $\SL^c$-oriented.
\end{proposition}
\begin{proof}
Note that (2) follows from (1).
Let $\scr L = \scr L_3$.
It suffices to exhibit a canonical equivalence $\Sigma^{\scr L^{\otimes 2}} E \wequi \Sigma^{\scr O} E$.
We have canonical equivalences
\[ \Sigma^{\scr L + \scr L} E \wequi \Sigma^{\scr O + \scr L^{\otimes 2}} E, \quad  \Sigma^{\scr L + \scr L^*} E \wequi \Sigma^{\scr O^2}E, \quad \Sigma^{\scr L + \scr L} \wequi \Sigma^{\scr L + \scr L^*}, \]
by \cite[Corollary 3.9, Lemma 4.1]{ananyevskiy2019sl}.
Consequently $\Sigma^{\scr O + \scr L^{\otimes 2}} E \wequi \Sigma^{\scr O^2}E$, whence the claim.
\end{proof}

\section{Euler classes for representable theories} \label{ref:representable-theory-euler-classes}
\subsection{Tautological Euler class}\label{subsection_with_tautological_Euler_class}
Let $E \in \SH(S)$ be a ring spectrum, $X \in \Sch_S$ and $V$ a vector bundle on $X$.
\begin{definition} \label{def:taut-euler-class}
We denote by $e(V) = e(V, E) \in E^{V^\dual}(X)$ the \emph{tautological Euler class of $V$}, defined as the composite \[ \1_X \wequi \Sigma^\infty_+ V \to \Sigma^\infty V/(V \setminus 0 )\wequi \Sigma^{V^\dual} \1 \xrightarrow{u} \Sigma^{V^\dual} E|_X \in \SH(X). \]
\end{definition}

\begin{lemma} \label{lemm:Euler-class-base-change}
Let $f: X' \to X \in \Sch_S$.
Then \[ f^* e(V, E) = e(f^*V, E) \in E^{f^* V}(X'). \]
\end{lemma}
\begin{proof}
Immediate.
\end{proof}

If $E$ is strongly $\SL$-oriented in the sense of \S\ref{subsub:oriented-rings}, and hence $\SL^c$-oriented in the sense of \S\ref{subsec:coh-features}, then for any relatively oriented vector bundle $V$ over a smooth and proper scheme $X/S$ we obtain an Euler number $n(V, \rho, E) \in E^0(S)$. See \S\ref{subsec:yoga}.

\subsection{Integrally defined Euler numbers} \label{subsec:integrally-defined}
\begin{corollary}[Euler numbers are stable under base change] \label{cor:Euler-number-base-change}
Let $E$ be a strongly $\SL$-oriented cohomology theory and let $V$ be vector bundle $V$ over a smooth and proper scheme $X/S$, relatively oriented by $\rho$. Let $f: S' \to S$ be a morphism of schemes.
Then \[ f^* n(V, \rho, E) = n(f^*V, f^*o, f^*E) \in E^0(S'). \]
\end{corollary}
\begin{proof}
This holds since all our constructions are stable under base change.
See in particular Lemma \ref{lemm:gysin-compat}(2) (for compatibility of Gysin maps with pullback, which applies since $X \to S$ is smooth), Proposition \ref{prop:SL-oriented-trick} and Proposition \ref{prop:thom-iso-E}(3) (ensuring that the identification $E^{V^\dual}(X) \wequi E^{L_\pi}(X)$ is compatible with base change) and Lemma \ref{lemm:Euler-class-base-change} (for compatibility of Euler classes with base change).
\end{proof}

\begin{proposition} \label{prop:application-bc}
Let $X/\Z[1/d!]$ be smooth and proper for some non-negative integer $d$. Let $V/X$ be a relatively oriented vector bundle.
Then for any field $k$ with $2d! \in k^\times$ we have \[ n(V_k, \rho, \H\tilde\Z) \in \Z[\lra{-1}, \lra{2}, \dots, \lra{d}] \subset \GW(k). \]
In fact there is a formula \[ n(V_k, \rho, \H\tilde\Z) = \sum_{a \in \Z[1/d!]^\times} n_a \lra{a} \] which holds over any such field, with the coefficients $n_a \in \Z$ independent of $k$ (and zero for all but finitely many $a$).
\end{proposition}

\begin{remark}\label{rmk:dependence_on_CDH+}
If $d=1$, Proposition \ref{prop:application-bc} relies on the novel results about Hermitian $K$-theory of the integers from \cite{CDHHLMNNS-3}. In the below proof, this is manifested in the dependence of \cite[Lemma 3.38(2)]{bachmann-eta} on these results. We will later use Proposition \ref{prop:application-bc}  for the $d=1$ case of Theorem \ref{two_values_e(bundlesoverZ[1/2])}, whence this result is also using  \cite{CDHHLMNNS-3} in an essential way. For $d \geq 2$, the proof is independent of \cite{CDHHLMNNS-3}.
\end{remark}

Note that here the assumption that the rank of $V$ equals the dimension of $X$ is included in the hypothesis that $V/X$ a relatively oriented vector bundle (see Definition \ref{def:relative-orientation}).
\begin{proof}
Recall the very effective cover functor $\tilde f_0$ and the truncation in the effective homotopy $t$-structure $\pi_0^\eff$, for example from \cite[\S\S3,4]{bachmann-very-effective}.
We have a diagram of spectra \[ \KO_k \leftarrow \tilde f_0 \KO_k \to \pi_0^\eff \KO_k \leftarrow \pi_0^\eff \MSL_k \wequi H\tilde\Z; \] see \cite[Example 16.34]{bachmann-norms} for the last equivalence.
The functors $\tilde f_0$ and $\pi_0^\eff$ are lax monoidal in an appropriate sense, so this is a diagram of ring spectra.
Moreover all of the ring spectra are strongly $\SL$-oriented (via the ring map $\MSL_k \to \pi_0^\eff \MSL_k$); see also Remark \ref{rmk:strong-orientation-permanence}.
Finally all the maps induce isomorphisms on $[\1, \ph]$, essentially by construction.
It follows that $n(V_k, \rho, \KO) = n(V_k, \rho, H\tilde\Z) \in \GW(k)$.
We may thus as well prove the result for $n(V_k, \rho, \KO)$ instead.

If $d\geq 2$, then we have $\KO \in \SH(\Z[1/d!])$, and by Corollary \ref{cor:Euler-number-base-change} we see that $n(V_k, \rho, \KO) \in im(\GW(\Z[1/d!]) \to \GW(k))$.
The result thus follows from Lemma \ref{lemm:GW-generators} below.
If $d=1$, we employ the $\SL$-oriented ring spectrum $\KO' \in \SH(\Z)$ from \cite[\S3.8.3]{bachmann-eta}.
We find that $n(V_k, \rho, \KO)$ is the image of $n(V, \rho, \KO') \in \KO'^0(\Z)$.
This latter group is isomorphic to $\GW(\Z)$ by \cite[Lemma 3.38(2)]{bachmann-eta}, whence the result.
\end{proof}

\begin{lemma}\label{lemm:GW-generators}
As a ring, $\GW(\Z[1/d!])$ is generated by $\lra{-1}, \lra{2}, \dots, \lra{d}$.
\end{lemma}
\begin{proof}
Let $A = \Z[1/d!]$.
If $d=1$, this is well-known \cite[Theorem II.4.3]{milnor1973symmetric}.
Hence $1/2 \in A$ and by \cite[Lemma I.6.3]{milnor1973symmetric} it suffices to show the same statement for $W(A)$.
For any prime $p$, we have a map $\partial^p: W(\Q) \to W(\mathbb{F}_p)$, assembling to an exact sequence \cite[Corollary IV.3.3]{milnor1973symmetric} \[ 0 \to W(\Z[1/e!]) \to W(\Q) \to \bigoplus_{p > e} W(\mathbb{F}_p). \]
Let $W' \subset W(A) \subset W(\Q)$ denote the subring generated by $\lra{i}$ for $2 \le i \le d$.
The proof of \cite[Theorem VI.4.1, see (4.3) p. 177]{Lam-intro_quadratic_forms_over_fields} shows that \[ \partial_{\le d} := \sum_{p \le d} \partial^p: W' \to \bigoplus_{p \le d} W(\mathbb{F}_p) \] is surjective.
Now let $\phi \in W(A)$.
We can find $\phi' \in W'$ such that $\partial_{\le d}(\phi) = \partial_{\le d}(\phi')$.
It follows that $\partial^p(\phi - \phi') = 0$ for all $p$ (since $\partial^p(\phi) = 0 = \partial^p(\phi')$ for $p > d$) and hence $\phi - \phi' \in \W(\Z)$.
This reduces to $d=1$, which we have already dealt with.
\end{proof}

\begin{remark}
The statement that  $\GW(\Z[1/d])$ is generated as a ring by $\{ \lra{-1}, \lra{p}: p \vert d\}$ analogous to Lemma~\ref{lemm:GW-generators} is false in general. For example, let $\theta= \lra{5\cdot 3} -\lra{29\cdot 3}$. We claim $\theta$ lies in $\GW(\Z[\frac{1}{5 \cdot 29}])$ but is not in the image of $\Z[\lra{-1}, \lra{5},\lra{29}]$. Since $5$ and $29$ are congruent mod $3$, the residue $\partial^3 (\theta) = \lra{5} -\lra{29} = 0 \in W(\FF_3)$, whence $\theta$ is an element of $\GW(\Z[\frac{1}{5 \cdot 29}])$. Note that $\partial^5 \theta = \lra{3}$, which has non-trivial discriminant. On the other hand, the discriminant of the residue $\partial^5$ at $5$ is trivial on the image of $\Z[\lra{-1}, \lra{5},\lra{29}]$ because $\partial^5 \lra{-5} = 1$, $\partial^5  \lra{5} = 1$, $\partial^5 \lra{29\cdot 5} = 1$, $\partial^5 \lra{-29\cdot 5} = 1$, and $\partial^5\lra{-1} =\partial^5\lra{-29} =\partial^5 \lra{29}= 0$. 
\end{remark}

\begin{notation}
Let $A \hookrightarrow \RR$ and $V$ be a relatively oriented, rank $n$ vector bundle on a smooth proper $n$-dimensional scheme $X$ over $A$.
We have $\GW(\RR) \wequi \Z \oplus \Z\lra{-1}$.
There are thus unique integers $n_{\RR}, n_{\CC} \in \Z$ such that \[ n(V_\RR, \rho, H\tilde\Z) = \frac{ n_{\CC} + n_{\RR} }{2} + \frac{n_{\CC} -  n_{\RR}}{2}\lra{-1}. \]
\end{notation}

The integers $n_\RR$ and $n_\CC$ are the Euler numbers of the corresponding real and complex topological vector bundles respectively, at least when $X$ is projective, justifying the notation. To show this, consider the cycle class map $\CH^\ast (X) \to H^\ast(X(\CC),\Z)$ from the Chow ring of a smooth $\CC$-scheme $X$ to the singular cohomology of the complex manifold $X(\CC)$. See \cite[Chapter 19]{fulton-intersection}. Furthermore, there are real cycle class maps from oriented Chow of $\RR$-smooth schemes $X$ to the singular cohomology of the real manifold $X(\RR)$), discussed in \cite{jacobson-fundamental-ideal} and  \cite{Hornbostel_real_cycle_class_map} (as well as more refined real cycle class maps defined in \cite{Benoist_WittenbergI} and \cite{Hornbostel_real_cycle_class_map}): For a smooth $\RR$-scheme $X$ and a line bundle $\mathcal{L} \to X$, consider the real cycle class map $$\CHt^\ast (X, \mathcal{L})\cong H^n(X, \K^{\MW}_n(\omega_{X/k})) \cong   H\tilde\Z^{L_{\pi}}(X) \to H^\ast(X(\CC),\Z(\mathcal{L}))$$ from the oriented Chow groups twisted by $\mathcal{L}$ to the singular homology of the associated local system on $X(\RR)$. We use results on the real cycle class map due to Hornbostel, Wendt, Xie, and Zibrowius \cite{Hornbostel_real_cycle_class_map}, including compatibility with pushforwards. Since they only had need of this compatibility in the case of the pushforward by a closed immersion, we first extend this slightly.

\begin{lemma}\label{smooth_proj_pushforwards_commute_real_cycle_class_map}
Let $\pi: X \to \Spec k$ be the structure map of a smooth, projective scheme $X$ of dimension $n$ over the real numbers $\RR$. Then the following square commutes, \begin{equation*}
  \begin{CD}
  H\tilde\Z^{L_{\pi}}(X) @>{}>> H^n(X(\RR), \Z(\omega_{X/k})) \\
  @V{\pi_*}VV                                      @V{}VV                       \\
\GW(\RR)                  @>{}>> H^0(\ast,\Z).
  \end{CD} 
  \end{equation*}
  where the vertical maps are the canonical pushforwards and the horizontal maps are the real cycle class maps.
\end{lemma}

\begin{proof}
$\pi$ is the composition of a closed immersion $i: X \hookrightarrow \P^n_k$ and the structure map $p: \P^n_k \to \Spec k$. The algebraic pushforward $\pi_*$ is the composition $p_* \circ i_*$ and the analogous statement holds for the topological pushforward by classical algebraic topology. By \cite[Theorem 4.7]{Hornbostel_real_cycle_class_map}, the pushforward $i_*$ commutes with real realization. We may thus reduce to the case where $X = \P^n_k$ is projective $n$-space over a field $k$. In this case, $\pi_*$ is an isomorphism by \cite[Prop 6.3, Theorem 11.7]{Fasel-projective_bundle}. Let $s : \Spec K \to \P^n_k$ be the closed immersion given by the origin. Since $\pi s = 1$, and the real realization maps commute with the algebraic and topological pushforwards of $s$ by  \cite[Theorem 4.7]{Hornbostel_real_cycle_class_map}, the real realization maps also commute with the algebraic and topological pushforwards of $\pi$.

\end{proof}

\begin{proposition}\label{nRRnCCaretopEnums}
Let $A \hookrightarrow \RR$ and $V$ be a relatively oriented, rank $n$ vector bundle on a smooth projective $n$-dimensional scheme $X$ over $A$. Then $n_\RR$ and $n_\CC$ are the Euler numbers of the corresponding real and complex topological vector bundles respectively, i.e., $n_{\RR}$ is the topological Euler number of the relatively oriented topological $\RR^n$-bundle $V(\RR)$ associated to the real points of $V$ on the real $n$-manifold $X(\RR)$, and $n_{\CC}$ is the analogous topological Euler number on $V(\CC) \to X(\CC)$.
\end{proposition}

\begin{proof}
By \cite[Proposition 6.1]{Hornbostel_real_cycle_class_map}, the $\A^1$-Euler class $e(V, H\tilde\Z)$ of $V_{\RR}$ in the oriented Chow group $H\tilde\Z^{V^*}(X_{\RR}) $ maps to the topological Euler class of $V(\RR)$ under the real cycle class map. By Lemma~\ref{smooth_proj_pushforwards_commute_real_cycle_class_map},  it follows that the image of the Euler number $n(V,H\tilde\Z)$ under the real cycle class map is the topological Euler number of $V(\RR)$. Under the canonical isomorphism $H^0(\ast, \Z) \cong \Z$, the real cycle class map $\GW(\RR) \to H^0(\ast, \Z) \cong \Z$ is the map taking a bilinear form over $\RR$ to its signature. It follows that $n_{\RR}$ is the topological Euler number of $V(\RR)$.

Let $\gamma: \CHt^\ast (X, \det V^*) \cong H\tilde\Z^{V^*}(X_{\RR}) \to \CH^{\ast}(X) \to \CH^{\ast}(X_{\CC}) \to H^{\ast}(X(\CC), \Z)$ denote the composition of the canonical map to Chow followed by the (usual) cycle class map, and we similarly have $$\gamma: \GW(\RR) \cong \CHt^{0}(\Spec \RR) \to H^{0}(\Spec \RR (\CC), \Z) \cong H^0(\ast, \Z) \cong \Z,$$ which sends the class of a bilinear form in $\GW(\RR)$ to its rank. The cycle class map is compatible with Chern classes (\cite[Proposition 19.1.2]{fulton-intersection}), whence the image of $e(V, H\tilde\Z)$ under $\gamma$ is the topological Euler class of $V(\CC)$. The cycle class map  commutes with the relevant pushforwards and pullbacks \cite[Corollary 19.2, Example 19.2.1]{fulton-intersection}, so we have that $\gamma(n(V,H\tilde\Z)) = n_{\CC}$ is the topological Euler number of $V(\CC)$. 
\end{proof}

\begin{remark}
If proper pushforwards in algebra and topology commute with real realization, as predicted by \cite[4.5 Remark]{Hornbostel_real_cycle_class_map}, then Proposition \ref{nRRnCCaretopEnums} holds more generally for smooth, proper schemes. It seems that proving this would take us too far afield, however. Alternatively, if $V \otimes \RR$ admits a non-degenerate section, then the results of \S\ref{sec:Euler-KO} (showing that the Euler number can be computed in terms of Scheja-Storch forms) imply that Proposition \ref{nRRnCCaretopEnums} holds for smooth, proper schemes, arguing as in \cite[Lemma 5]{FourLines}.
\end{remark}

\begin{theorem}\label{two_values_e(bundlesoverZ[1/2])}
Suppose $X$ is smooth and proper over $\Z[1/2]$. Let $V$ be a relatively oriented vector bundle on $X$ and let $V_k$ denote the base change of $V$ to $k$ for any field $k$. Then either 
\begin{equation}\label{eoverZ[1/2]no<2>} n^{\GS}(V_k, \rho)  = \frac{ n_{\CC} +n_{\RR} }{2} + \frac{n_{\CC} -  n_{\RR}}{2}\lra{-1}\end{equation} or\begin{equation}\label{eoverZ[1/2]with<2>} n^{\GS}(V_k, \rho)  = \frac{ n_{\CC} +n_{\RR} }{2} + \frac{n_{\CC} -  n_{\RR}}{2}\lra{-1}  +  \lra{2} -1,\end{equation} where the \emph{same} formula holds for \emph{all} fields $k$ of characteristic $\ne 2$.

If instead $X$ is smooth proper over $\Z$, then \eqref{eoverZ[1/2]no<2>} holds for \emph{any} field $k$ (including fields $k$ of characteristic two).
\end{theorem}

Recall that for the last claim regarding $X$ smooth and proper over $\Z$, we rely on \cite{CDHHLMNNS-3}. See Remark \ref{rmk:dependence_on_CDH+}.

\begin{proof}
First assume that $char(k) \ne 2$.  
By Corollary \ref{co:pieBM=nPH} we have $n^{\GS}(V_k, \rho) =  n(V_k, \rho, H\tilde\Z)$.
If the base is $\Z$, then we learn from Proposition \ref{prop:application-bc} that there exist $a, b \in \Z$ (independent of $k$!) such that  \[ n(V_k, \rho, \H\tilde\Z) = a + b\lra{-1}. \]
In $\GW(\Z[1/2]) \hookrightarrow \GW(\Q)$ we have the relations \[ \lra{-2} = 1 + \lra{-1} - \lra{2} \text{ and } 2\lra{2} = 2 \] (the former because $\lra{a} + \lra{-a} = \lra{1} + \lra{-1}$ for any $a$, and see e.g. \cite[Lemma 42]{bachmann-gwtimes} for the latter).
Hence if the base is $\Z[1/2]$, there exist $a, b \in \Z, c \in \{0,1\}$ (independent of the choice of field $k$!) such that \[ n(V_k, \rho, \H\tilde\Z) = a + b\lra{-1} + c \lra{2}. \]
If the base is $\Z$, let us put $c=0$.
By construction we have $$n_{\RR} = \Rsignature n(V_{\RR}, \rho, \H\tilde\Z)=(a+c)-b$$ and $$ n_{\CC} = \Rrank n(V_{\CC}, \rho, \H\tilde\Z)=(a+c)+b,$$ which determines $a+c$ and $b$, so that there only remain at most two possible values for $n(V_k,\rho,\H\tilde\Z)$.

Now suppose that $char(k) = 2$ (so that in particular the base is $\Z$).
Since $\lra{1} = \lra{-1}$ over fields of characteristic $2$, we need to show that $n^{\GS}(V_k, \rho) = n_\CC$.
We may as well assume that $k=\FF_2$.
The rank induces an isomorphism $\GW(\FF_2) \cong \Z$ (see e.g. Corollary \ref{cor:GW(F2)}).
Considering the canonical maps \[ \GW(\FF_2) \to \KGL^0(\FF_2) \leftarrow \KGL^0(\Z) \to \KGL^0(\Z[1/2]) \leftarrow \GW(\Z[1/2]) \] in which all but the right-most one are isomorphisms, we get the string of equalities \[ n(V_{\FF_2}) = n(V_{\FF_2}, \KGL) = n(V_{\Z}, \KGL) = n(V_{\Z[1/2]}, \KGL) = rk(n(V_{\Z[1/2]}, \KO)). \]
The result follows.
\end{proof}

\begin{remark}\label{rmk:distinguishing_betwee_<2>_and_0_in_Euler_classes}
The difference of the two values given in Theorem \ref{two_values_e(bundlesoverZ[1/2])} is  $\lra{2} - 1$, so the two possibilities can be distinguished by the value of $\disc n(V_k)$ in $k^*/(k^*)^2$ for any field $k$ in which $2$ is not a square, such as $\mathbb{F}_3$ or $\mathbb{F}_5$. Such discriminants can be evaluated by a computer, as shown to the second named author by Anton Leyton and Sabrina Pauli.
\end{remark}

\subsection{Refined Euler classes and numbers}
\label{subsec:refined-euler}

\begin{definition}[refined Euler class] \label{def:refined-euler-class}
Let $E \in \SH(S)$ be a homotopy ring spectrum, $X \in \Sch_S$, $V \to X$ a vector bundle and $\sigma: X \to V$ a section with zero scheme $Z = Z(\sigma)$.
We denote by $e(V, \sigma) = e(V, \sigma, E) \in E^{V^\dual}_Z(X)$ the class corresponding to the composite \[ X/X \setminus Z \xrightarrow{\sigma} V/V \setminus 0 \wequi \Sigma^{V^\dual} \1 \xrightarrow{u} \Sigma^{V^\dual} E|_X \in \SH(X); \] see Example \ref{ex:cohomology-with-support}.
\end{definition}

\begin{remark} \label{rmk:refined-euler-bc}
It is clear by construction that refined Euler classes are stable under base change.
\end{remark}

\begin{remark} \label{rmk:refined-euler-levine}
In \cite[Definition 3.9]{levine2018motivic} the authors define for an $\SL$-oriented ring spectrum $E$ and a vector bundle $p: V \to X$ the \emph{canonical Thom class} $th(V) \in E^{p^*V}_0(V)$, which one checks coincides with $e(p^*V, \sigma_0)$, where $\sigma_0$ is the tautological section of $p^*V$. Since $e(V, \sigma) = \sigma^* e(p^*V, \sigma_0)$ we deduce that \[ e(V, \sigma) = \sigma^* th(V). \]
\end{remark}

\begin{lemma} \label{lemm:forget-refinement}
The ``forgetting support'' map $E_Z^{V^\dual}(X) \to E^{V^\dual}(X)$ sends $e(V, \sigma)$ to $e(V)$.
\end{lemma}
\begin{proof}
This follows from the commutative diagram
\begin{equation*}
\begin{CD}
V @>>>       V/V \setminus 0 @>>> \Sigma^{V^\dual} E \\
@A{\sigma}AA @A{\sigma}AA                    \\
X @>>>       X/X\setminus Z.
\end{CD}
\end{equation*}
Indeed the composite from the bottom left to the top right along the top left represents $e(V)$, by Definition \ref{def:taut-euler-class} (note that $\sigma$ is a homotopy inverse to the projection $V \to X$), whereas the composite along the bottom right represents $e(V, \sigma)$ with support forgotten, by Definition \ref{def:refined-euler-class} and Example \ref{ex:cohomology-with-support}.
\end{proof}

\begin{definition}[refined Euler number]
Suppose that $\pi: X \to S$ is smoothable lci, $Z \to S$ is finite and $V$ is relatively oriented (so in particular $rk(V) = rk(\L_\pi)$).
Suppose further that $E$ is $\SL$-oriented.
Then we put \[ n(V, \sigma, \rho) = n(V, \sigma, \rho, E) = \pi_* e(V, \sigma) \in E^0(S). \]
\end{definition}

\begin{corollary} \label{cor:forget-number}
Suppose that additionally $\pi: X \to S$ is smooth and proper.
Then $n(V, \sigma, \rho) = n(V, \rho) \in E^0(S)$.
\end{corollary}
\begin{proof}
Combine Lemma \ref{lemm:forget-refinement} and Example \ref{ex:forget-support-gysin}.
\end{proof}

\subsection{Refined Euler classes and the six functors formalism}
We now relate our Euler classes to the six functors formalism.
The following result shows that our refined Euler class coincides with the one defined by D\'eglise--Jin--Khan \cite[Remark 3.2.10]{DJK}.
\begin{proposition} \label{prop:euler-class-six-functors}
Let $E \in \SH(S)$ be a homotopy ring spectrum, $X \in \Sm_S$, $V$ a vector bundle over $X$ and $\sigma$ a section of $V$.
Then \[ e(V, \sigma, E) = \sigma^*z_*(1), \] where $z: X \to V$ is the zero section and we use the canonical isomorphism $N_z \wequi V$ to form the pushforward $z_*$.
\end{proposition}
\begin{proof}
Let $p: V \to X$ be the projection and $s_0: V \to p^* V$ the canonical section.
Then $\sigma^*(p^*V, \sigma_0) = (V, s)$ and hence, using Remark \ref{rmk:refined-euler-bc}, it suffices to show that $z_*(1) = e(p^*V, \sigma_0, E)$.
By \cite[last sentence of \S2.1.1]{EHKSY2} we know that \[ z_*: E^0(X) \wequi E^{z^* p^*V^\dual + L_z} \to E^{p^*V^\dual}_X(V) \] is the purity equivalence.
In the case of the zero section of a vector bundle, it just takes the tautological form \[ [\1, E]_X \wequi [V/V \setminus 0, V/V \setminus 0 \wedge E]_X \wequi [V/V \setminus 0, p^*V/p^*V \setminus 0 \wedge E_V]_V, \] and hence indeed sends $1$ to $e(p^*V, \sigma_0)$.
\end{proof}

It follows from the above that Euler classes of vector bundles are determined by Euler classes of vector bundles with non-degenerate sections.
\begin{example}\label{KW-Euler_class_Koszul}
Let $V \to X$ be a vector bundle and $\sigma$ a section.
Suppose $1/2 \in S$.
The Euler Class of $V$ in $\KW = \KO[\eta^{-1}]$ is given by the Koszul complex with its canonical symmetric bilinear form.
Indeed we may assume that $V$ has a regular section, in which case this follows from the existence of the morphism of cohomology theories $\BL^\naive \to \KW$ and Proposition \ref{Koszul_form_push_forward_from_support}.
\end{example}

We deduce that the ``Meta-Theorem'' of \S\ref{subsec:yoga} holds in this setting, even in a slightly stronger form with supports:
\begin{corollary} \label{cor:meta-rep}
If $\sigma$ is a non-degenerate section (i.e. locally given by a regular sequence), then $i_*(1) = e(V, \sigma) \in E_Z^{V^\dual}(X)$.
In particular, forgetting supports (taking the image along $E_Z^{V^\dual}(X) \to E^{V^\dual}(X)$) we have $i_*(1) = e(V) \in E^{V^\dual}(X)$, i.e. Theorem \ref{thm:meta} holds in this situation.
\end{corollary}
\begin{proof}
The second statement follows from the first and Lemma \ref{lemm:forget-refinement}; hence we shall prove the first.
Consider the following cartesian square
\begin{equation*}
\begin{CD}
Z @>i>> X \\
@ViVV  @VzVV \\
X @>\sigma>> V.
\end{CD}
\end{equation*}
A well-known consequence of regularity of $\sigma$ is that this square is tor-independent.\footnote{To see this note that $z_* \scr O_X$ can be resolved by the Koszul complex $K(p^* V, \sigma_0)$ for the tautological section $\sigma_0$ of the pullback $p^* V$ of $V$ along $p: V \to X$.  It follows now from the projection formula that $z_* \scr O_X \otimes^L \sigma_* \scr O_X \wequi \sigma_*\sigma^*z_* \scr O_X \wequi \sigma_*\sigma^* K(p^* V, \sigma_0) \wequi \sigma_* K(V, \sigma).$ Since by definition $\sigma$ locally corresponds to a regular sequence (and $\sigma$ is affine), the claim follows.}
Since $i_*$ is compatible with tor-independent base change (Lemma \ref{lemm:gysin-compat}(2)), we deduce from Proposition \ref{prop:euler-class-six-functors} that \[ e(V, \sigma, E) = \sigma^*z_*(1) = i_*i^*(1) = i_*(1). \]
This was to be shown.
\end{proof}

It follows that, as explained in \S\ref{subsec:yoga}, both the ordinary and refined Euler numbers in $E$-cohomology can be computed as sums of local indices.
The remainder of this paper is mainly concerned with determining these indices, for certain examples of $E$.

\begin{remark}
These results can be generalized slightly.
Fix a scheme $S$ and an $\SL$-oriented ring spectrum $E \in \SH(S)$.
\begin{enumerate}
\item Let $X/S$ smoothable lci, $V/X$ a vector bundle, relatively oriented in the sense of Definition \ref{def:relative-orientation}.
  If $X/S$ is in addition proper, then we can transfer the Euler class along the structure morphism (see \S\ref{sec:gysin}) as before to obtain an Euler number $n(V, \rho, E) \in E^0(S)$.
  More generally, without assuming $X/S$ proper, given a section $\sigma$ with zero scheme $Z$ proper over $S$, we obtain the refined Euler number $n(V, \sigma, \rho, E)$.
\item Let $X/S$ be arbitrary, $V$ a vector bundle, $\sigma$ a section of $V$ with zero scheme $i: Z \to X$, and suppose that $i$ is a regular immersion (but $\sigma$ need not be a non-degenerate section, i.e. $Z$ could have higher than expected dimension).
  In this case there is an \emph{excess bundle} $\scr E = cok(N_{Z/X} \to V|_Z)$.\footnote{The map $N_{Z/X} \to V|_Z$ is always injective, and is an isomorphism precisely if the section is non-degenerate.}
  A straightforward adaptation of the proof of Corollary \ref{cor:meta-rep}, using the excess intersection formula \cite[Proposition 3.3.4]{DJK}, shows that \[ e(V, \sigma, E) = i_*(e(\scr E, E)). \]
\item Putting everything together, let $X/S$ be proper smoothable lci, $V$ a relatively oriented vector bundle and $\sigma$ a section with zero scheme $Z$ regularly immersed in $X$.
  Then \[ n(V, \rho, E) = \sum_{Z' \subset Z} n(\scr E|_{Z'}, \rho', E). \]
  Here the sum is over clopen components $Z'$ of $Z$ and $\rho'$ denotes the induced relative orientation of $\scr E$.
  Note that if $\sigma$ is non-degenerate on $Z'$, i.e. $\scr E|_{Z'} = 0$, then $e(\scr E|_{Z'}, E) = 1 \in E^0(Z')$ and $n(\scr E|_{Z'}, \rho', E) = \ind_{Z'}(\sigma)$ as before.
\end{enumerate}
\end{remark}

\section{$d$-Dimensional planes on complete intersections in projective space} \label{subsec:enumerative-app}
\subsection{Some Euler numbers of symmetric powers on Grassmannians}
Grassmannians and flag varieties are smooth and proper over $\Z$, and the Euler classes of many of their vector bundles have interesting interpretations in enumerative geometry. Computations over $\RR$ and $\CC$ are available in the literature in connection with enumerative results or accessible with localization techniques in equivariant cohomology \cite{Atiyah-Bott} \cite{GP-loc}. Integrality of Euler classes, as in Theorem \ref{two_values_e(bundlesoverZ[1/2])} and Proposition \ref{prop:application-bc}, can leverage such results to all fields. We do this now using the $\Z[1/2]$ case of Theorem \ref{two_values_e(bundlesoverZ[1/2])} and a characteristic class argument. This is independent of the recent work \cite{CDHHLMNNS-3} on Hermitian K-theory over $\Z$ and the characteristic class argument may be of some independent interest. One can alternatively deduce Corollary \ref{co:eoplussym} from the $\Z$ case of Theorem \ref{two_values_e(bundlesoverZ[1/2])} and \cite{CDHHLMNNS-3}.

\begin{remark}\label{changing_relative_orientations}
Suppose $X$ is a smooth, proper $\Z$-scheme with geometrically connected fibers and $\Pic X$ torsion free, for example, $X$ a Grassmannian or projective space. For a relatively orientable vector bundle $V \to X$ defined over $\Z$, there are at most two isomorphism classes of relative orientations. Namely, by assumption, there is a line bundle $L \to X$ and isomorphism $\rho: L^{\otimes 2} \stackrel{\cong}{\longrightarrow} \omega_{X/\Z} \otimes \det V$. Since $\Pic X$ is torsion free, any relative orientation is an isomorphism $L^{\otimes 2} \stackrel{\cong}{\longrightarrow} \omega_{X/\Z} \otimes \det V$, whence two such differ by a global section of $\Hom(L^{\otimes 2},L^{\otimes 2}) \cong \calO(X)^*$. By hypothesis on $X$, the fibers of the pushforward of $\calO_X$ all have rank $1$, whence this pushforward is $\calO_{\Z}$ and $\calO(X)^* \cong \Z^* = \{\pm 1\}$. Thus any relative orientation is isomorphic to $\rho$ or $-\rho$. We then have $n(V, \rho) = \lra{-1} n(V, -\rho)$.  

Consequently we suppress the choice of orientation and just write $n(V)$.
Beware that this does not mean that every vector bundle is relatively orientable, though!
\end{remark}

\begin{remark}\label{canonical_rel_orientation_sym}
There is a canonical relative orientation for certain classes of vector bundles on  Grassmannians $X = \Gr(d,n)$: a point $p$ of $X$ with residue field $L$ corresponds to a dimension $d+1$ subspace of $L^{n+1}$. A choice of basis $\{e_0,\ldots,e_n\}$ of $L^{n+1}$ such that the span of $\{e_0,\ldots,e_d\}$ is $p$ defines canonical local coordinates for an affine chart isomorphic to $\A^{(n-d)(d+1)}$ with $p$ as the origin (see for example \cite[Definition 42]{CubicSurface}). This defines local trivializations of the tautological and quotient bundles on $X$, and therefore also of their tensor, symmetric, exterior powers and their duals. A vector bundle $V$ formed from such operations on the tautological and quotient bundles and which is relatively orientable on $X$ inherits a canonical relative orientation $\rho$ such that the local coordinates and trivializations just described are compatible with $\rho$ in the sense of \cite[Definition 21]{CubicSurface}. This is described in \cite[Proposition 45]{CubicSurface} in a special case, but the argument holds in the stated generality. (One only needs the determinants of the clutching functions to be squares which follows from the relative orientability of $V$. Together with the explicit coordinates this gives the relative orienation.) This relative orientation has the property that it is defined over $\Z$ and for any very non-degenerate section $\sigma$, the data just described gives a system of coordinates in the sense of Definition \ref{def:coordinates}.
\end{remark}

\begin{corollary}\label{co:eoplussym}
Let $d \leq n$ be positive integers, and let $X=\Gr(d,n)$ be the Grassmannian of $d$-planes $\P^d$ in $\P^n$. Let $V = \oplus_{i=1}^j \Sym^{n_i} \tautbun^*$, where $\tautbun$ denotes the tautological bundle on $X$, and $n_1,\ldots,n_j$ are positive integers such that $\Rrank V = \dim X$ and $V$ is relatively orientable, i.e. such that $\sum_{i=1}^j {n_i + d \choose d}= (d+1)(n-d) $ and $\sum_{i=1}^j \frac{n_i}{d+1}{n_i + d \choose d} + n+1$ is even. Then $$n(V) = \frac{ n_{\CC} +n_{\RR} }{2} + \frac{ n_{\CC} -n_{\RR} }{2}\lra{-1}$$ over any ring in which $2$ is invertible (where we interpret $n(V)$ as $n(V,KO)$), or any field (where we interpret $n(V)$ as $n^{\GS}(V)$). 
\end{corollary}

\begin{proof}
Let $\calO(1)$ denote the generator of $\Pic X$ given by the pullback of the tautological bundle under the Pl\"ucker embedding. Then $\omega_{X/\Z} \cong \calO(-n-1)$ and $\det \Sym^{n_i} \tautbun^* \cong \calO(\frac{n_i}{d+1}{n_i + d \choose d})$. Thus $V$ is relatively orientable as a bundle over $\mathbb{Z}$ and Remarks \ref{changing_relative_orientations} and \ref{canonical_rel_orientation_sym} apply.

By the $\Z[1/2]$ case of Theorem \ref{two_values_e(bundlesoverZ[1/2])}, it is enough to show that the discriminant of $n(V_{\mathbb{F}_p})$ is trivial for some prime $p$ congruent to $1$ mod $4$, and such that $2$ is not a square. See Remark \ref{rmk:distinguishing_betwee_<2>_and_0_in_Euler_classes}.
Let $\emW$ denote the Eilenberg--MacLane spectrum of the $\eta$-inverted Milnor--Witt sheaves $\K^{\MW}_*[\eta^{-1}]$, cf. \cite[Remark 3.1]{Levine-Witt}, and consider the associated Euler class $e(V,\emW)$.
Then $n(V,\emW)$ determines the Witt class of $n(V)$, and hence the discriminant of $n(V)$ as well.

{\bf $d$ is even:} Suppose that $d$ is even. Let $\pi$ denote the structure map of $X$. Since $n(V,\emW) = \pi_*\prod_{i=1}^j e(\Sym^{n_i} \tautbun^*, \emW)$, it is enough to show that $e(\Sym^{n_1} \tautbun^*, \emW) =0$. By the Jouanolou device, we may assume any vector bundle is pulled back from the universal bundle. It is therefore enough to show the same for the dual tautological bundle on the universal Grassmannian $B\GL_{d+1}$, i.e., let $\tautbun_{d+1}^*$ denote the dual of the tautological bundle on $B\GL_{d+1}$; we show that $e(\Sym^{n_1} \tautbun_{d+1}^*, \emW)=0$. By Ananyevskiy's splitting principle \cite[Theorem 6]{Ananyevskiy-SL_projective_bundle_thm} and its extension due to M. Levine \cite[Theorem 4.1]{Levine-Witt}, we may show the vanishing of the $\emW$-Euler class of $\Sym^{n_1} \tautbun_{d+1}^*$ after pullback to $$B\SL_2 \times B\SL_2 \times \ldots \times B\SL_2 \times B\SL_1$$ via the map classifying the external Whitney sum of the tautological bundles. Here we use that $d+1$ is odd. This pullback of $\Sym^{n_1} \tautbun_{d+1}^*$ contains the odd-rank summand $\Sym^{n_1} \tautbun_{1}^*$, and therefore its $\emW$-Euler class is $0$ as desired \cite[Lemma 3]{Ananyevskiy-SL_projective_bundle_thm} \cite[Lemma 4.3]{Levine-Witt}. 

{\bf $d$ is odd:} Let $k$ be a finite field whose order is prime to $2\prod_{i=1}^j (n_i)!$, congruent to $1$ mod $4$ (so $-1$ is a square), and such that $2$ is not a square. By Theorem \ref{two_values_e(bundlesoverZ[1/2])}, it suffices to show that the discriminant of $n(V,\Gr(d,n)) \in \W(k) \cong \GW(k)/\mathbb{Z}h$ is trivial, cf. Remark \ref{rmk:distinguishing_betwee_<2>_and_0_in_Euler_classes}. 

Define $r$ in $\Z$ so that $d=2r+1$. Let $\Stautbun_{d+1}^*$ denote the dual tautological bundle on $B\SL_{d+1}$ and let $p_1,\ldots,p_r,p_{r+1}$ and $e$ in $\emW^*(B\SL_{d+1}) $ denote its Pontryagin and Euler classes respectively. (Often one would let $p_i$ be the Pontryagin classes of the tautological bundle, not its dual, but this is more convenient here.) By Lemma \ref{e(Sym^ninZ[pie])}, $e(\Sym^{n_i} \Stautbun_{d+1}^*, \emW)$ is in the image of $\Z[p_1,\ldots,p_r,e] \to \emW^*(B\SL_{d+1})$. (Note that we have omitted $p_{r+1}$ as $p_{r+1} = e^2$ \cite[Corollary 3]{Ananyevskiy-SL_projective_bundle_thm}.) Therefore $e(V,\Gr(d,n))$ can be expressed as a polynomial with integer coefficients in the Pontryagin classes and Euler class of the dual tautological bundle $\tautbun^*$ on $\Gr(d,n)$. By Lemma \ref{discpi_*monomial=1}, it follows that $\disc n(V,\Gr(d,n))=1$ in $k^*/(k^*)^2$ as desired. 
\end{proof}  

M. Levine \cite{Levine-Witt} uses the normalizer $N$ of the standard torus of $\SL_{2}$ $$1 \to \left\{\begin{pmatrix}
t & 0  \\
0 & t^{-1} 
\end{pmatrix} \right\} \to N \to  \left\{\begin{pmatrix}
0 & 1  \\
-1 & 0 
\end{pmatrix} \right\} \to 1$$  and bundles $\tO(a)$ and $\tO^{-}(a)$ for $a$ in $\Z$ corresponding to the representations \begin{align*}\begin{pmatrix}
t & 0  \\
0 & t^{-1} 
\end{pmatrix} \mapsto \begin{pmatrix}
t^a & 0  \\
0 & t^{-a} 
\end{pmatrix} \\ 
\begin{pmatrix}
0 & 1  \\
-1 & 0 
\end{pmatrix} \mapsto \begin{pmatrix}
0 & 1  \\
(-1)^a & 0 
\end{pmatrix}
\end{align*} and \begin{align*}\begin{pmatrix}
t & 0  \\
0 & t^{-1} 
\end{pmatrix} \mapsto \begin{pmatrix}
t^a & 0  \\
0 & t^{-a} 
\end{pmatrix} \\ 
\begin{pmatrix}
0 & 1  \\
-1 & 0 
\end{pmatrix} \mapsto \begin{pmatrix}
0 & -1  \\
(-1)^{a+1} & 0 
\end{pmatrix},
\end{align*} respectively, to compute characteristic classes, and we use his technique. We will use the notation $\tOpom (a)$ to mean either $\tO(a)$ and $\tO^{-}(a)$ when a claim holds for both possibilities. We likewise use the $\emW$-Pontryagin (or Borel) classes of a vector bundle with trivialized determinant of Panin and Walter \cite{panin2010quaternionic}. See \cite[Introduction, Section 3]{Ananyevskiy-SL_projective_bundle_thm} or \cite[Section 3]{Levine-Witt} \cite[Section 2]{Wendt-oriented_schubert} for background on these classes.

Let $e_i \in \emW^*(N^{r+1})$ denote the pullback of $e(\Stautbun_{2}^* ,\emW)$ under the $i$th projection $BN^{r+1} \to BN$ composed with the canonical map $BN \to B\SL_2$.

Given vector bundles $V$ and $E$ on schemes $X$ and $Y$, respectively, let $V \boxtimes E$ denote the vector bundle on $X \times Y$ given by the tensor product of the pullbacks of $V$ and $E$.
\begin{lemma} \label{eboxtimesOa_is_in_Zei}
Suppose $r \geq 1$ and $a_1,\ldots,a_{r+1}$ are integers, and that our base scheme is a field $k$ with characteristic not dividing $2 \prod_{i=1}^{r+1} a_i$. The Euler class $e( \boxtimes_{i=1}^{r+1} \tOpom (a_i), \emW)$ and $\emW$-Pontryagin classes $p_j(\boxtimes_{i=1}^{r+1} \tOpom (a_i), \emW)$ for $j=1,\ldots, 2^r$ are in the image of the map $$\Z[e_1^2,\ldots, e_{r+1}^2] \to \emW^*(BN^{r+1}).$$ 
\end{lemma}

\begin{proof}
Proceed by induction on $r$. By Ananyevskiy's splitting principle \cite[Theorem 6]{Ananyevskiy-SL_projective_bundle_thm}, there exists a map $\pi:Y \to BN^{r}$ such that the pullback of $\boxtimes_{i=2}^{r+1} \tOpom (a_i)$ is a direct sum of rank $2$ bundles $V_i$ on $Y$, $$\pi^* \boxtimes_{i=2}^{r+1} \tOpom (a_i) \cong \oplus_{i=1}^{2^{r-1}} V_i,$$ and the map $$\emW^*(BN \times BN^{r}) \to \emW^*(BN \times Y) $$ is injective. We pull back the vector bundle $\boxtimes_{i=1}^{r+1} \tOpom (a_i)$ along the map $1_{BN} \times \pi : BN \times Y \to BN^{r+1}$ and obtain an isomorphism \begin{equation}\label{split_boxtimes}(1_{BN} \times \pi)^* \boxtimes_{i=1}^{r+1} \tOpom (a_i) \cong \oplus_{i=1}^{2^{r-1}} (\tOpom (a_1) \boxtimes V_i).\end{equation}  By \cite[Theorem 7.1]{Levine-Witt}, and the equality $e(\Stautbun_{2},\emW)^2 = e(\Stautbun_{2}^*,\emW)^2 $ which follows from \cite[Theorem 11.1]{Levine-EC}, \begin{equation}\label{e(Opm(a))}
e(\tOpom (a_1),\emW)^2 = a_1^2 e(\Stautbun_{2}^*,\emW)^2. 
\end{equation}
Since $\tOpom (a_1)$ and $V_i$ both have rank $2$, \cite[Proposition 9.1]{Levine-Witt} and Equation~\eqref{e(Opm(a))} imply that \begin{equation}\label{eOaboxtimesVi}e(\tOpom (a_1) \boxtimes V_i ,\emW) = a_1^2 e_1^2 - e(V_i, \emW)^2\end{equation} and \begin{equation}\label{pOaboxtimesVi}p_1(\tOpom (a_1) \boxtimes V_i ,\emW) = 2(a_1^2 e_1^2 + e(V_i, \emW)^2).\end{equation} This establishes the claim when $r=1$, because $p_2=e^2$.

We now assume the claim holds for $r-1$. By Equation~\ref{split_boxtimes}, \begin{equation}\label{eindexboxtimes}e((1_{BN} \times \pi)^* \boxtimes_{i=1}^{r+1} \tOpom (a_i), \emW) = \prod_{i=1}^{2^{r-1}} e(\tOpom (a_1) \boxtimes V_i, \emW) \end{equation} and  \begin{equation}\label{pindexboxtimes}p((1_{BN} \times \pi)^* \boxtimes_{i=1}^{r+1} \tOpom (a_i), \emW) = \prod_{i=1}^{2^{r-1}} p(\tOpom (a_1) \boxtimes V_i, \emW),\end{equation} where $p$ denotes the total Pontryagin class. 

Combining Equations~\eqref{eOaboxtimesVi} and \eqref{eindexboxtimes} shows that $$ e((1_{BN} \times \pi)^* \boxtimes_{i=1}^{r+1} \tOpom (a_i), \emW) = \sum_{i=0}^{2^{r-1}} (a_1 e_1)^{2i} \sigma_{2^{r-1} - i} ( e(V_1, \emW)^2, \ldots,  e(V_{2^{r-1}}, \emW)^2),$$ where $\sigma_i$ denotes the $i$th elementary symmetric function. Since the $V_i$ have rank $2$, $$e(V_1, \emW)^2 = p_1(V_1, \emW),$$ and the Whitney sum formula for Pontryagin classes implies that $$ p_{i} (\oplus_{j=1}^{2^{r-1}} V_i, \emW) =\sigma_{i} ( e(V_1, \emW)^2, \ldots,  e(V_{2^{r-1}}, \emW)^2).$$ Since $\oplus_{j=1}^{2^{r-1}} V_i \cong \pi^* \boxtimes_{i=2}^{r+1} \tOpom (a_i)$, it follows by induction that $e( \boxtimes_{i=1}^{r+1} \tOpom (a_i), \emW)$ is in the image of $\Z[e_1^2,\ldots, e_{r+1}^2]$.

Combining equations~\eqref{pOaboxtimesVi} \eqref{eOaboxtimesVi} and \eqref{pindexboxtimes}, we have $$p((1_{BN} \times \pi)^* \boxtimes_{i=1}^{r+1} \tOpom (a_i), \emW) = \prod_{i=1}^{2^{r-1}} (1+ 2(a_1^2 e_1^2 + e(V_i, \emW)^2)+(a_1^2 e_1^2 - e(V_i, \emW)^2)^2 ).$$ Because the elementary symmetric polynomials generate all symmetric polynomials, it follows that $p((1_{BN} \times \pi)^* \boxtimes_{i=1}^{r+1} \tOpom (a_i), \emW) $ is in the image of $$\Z[e_1^2, \sigma_{i} ( e(V_1, \emW)^2, \ldots,  e(V_{2^{r-1}}, \emW)^2): i = 1,\ldots, 2^{r-1}].$$ As above, these elementary symmetric functions are the Pontryagin classes of the pullback of $\boxtimes_{i=2}^{r+1} \tOpom (a_i)$, finishing the proof by induction.
\end{proof}

 Let $p_1,\ldots,p_r$ and $e$ in $\emW^*(B\SL_{d+1}) $ denote the Pontryagin and Euler classes respectively of the dual tautological bundle $\Stautbun_{d+1}^*$ on $B\SL_{d+1}$.

\begin{lemma}\label{e(Sym^ninZ[pie])}
Let $d = 2r+1$ be an odd integer and let $n$ be a positive integer. Let $k$ be a field of characteristic not dividing $2n!$. Then $e(\Sym^{n} \Stautbun_{d+1}^*, \emW)$ is in the image of $\Z[p_1,\ldots,p_r,e] \to \emW^*(B\SL_{d+1})$.
\end{lemma}

\begin{proof}
Let $f$ denote the composite $$BN^{r+1} \to B\SL_{2}^{r+1} \to B\SL_{d+1}$$ of the $(r+1)$-fold product of the canonical map $BN \to B\SL_{2}$ with the map $ B\SL_{2}^{r+1} \to B\SL_{d+1}$ classifying the external direct sum $\oplus_{i=1}^{r+1} \Stautbun_2$. There is an isomorphism \begin{equation}\label{Sym^n_as_sum}f^* \Sym^{n} \Stautbun_{d+1}^* \cong  \bigoplus_{\substack{(a_1,\ldots,a_{r+1}) \in \Z_{\geq 0}^{r+1} \\ \sum_i a_i = n}} \boxtimes_{i=1}^r \Sym^{a_i}\tO(1).\end{equation} By inspection, the symmetric powers of the tautological bundle $\tO(1)$ on $BN$ split into a sum of bundles of rank $\leq 2$, cf. \cite[p. 38]{Levine-Witt}: \begin{equation}\label{SymtOa}
\Sym^{a}\tO(1) \cong
\begin{cases}
\oplus_{l = 0}^b \tO^{(-1)^l} (i-2l) , \text{ when } a= 2b+1 \text{ is odd}\\
\oplus_{l = 0}^{b-1} \tO^{(-1)^l} (i-2l) \oplus \calO  ,\text{ when } a = 2b \text{ is even and } b\text{ is even }\\
\oplus_{l = 0}^{b-1} \tO^{(-1)^l} (i-2l) \oplus \gamma  ,\text{ when } a = 2b \text{ is even and } b\text{ is odd },
\end{cases}
\end{equation} where $\gamma$ is the line bundle corresponding to the representation $N \to \GL_1$ sending the torus to $1$ and $\begin{pmatrix}
0 & 1  \\
-1 & 0 
\end{pmatrix}$ to $-1$.

Combining Equations~\eqref{Sym^n_as_sum} and \eqref{SymtOa}, we can decompose $f^* \Sym^{n} \Stautbun_{d+1}^*$ into a direct sum with summands which have various numbers of factors of rank $2$. Separate these summands into those with at least two rank $2$ factors and those with only one rank $2$ factor, if any of the latter sort appear. (This occurs when we can take all but one $a_i$ to be even.) The direct sum of the latter such terms can alternatively be expressed as a sum of pullbacks of $\Sym^{a_i} \tO(1)$ under some projection $N^{r+1} \to N$ tensored with some $\gamma$'s pulled back from other projections. We may ignore the factors of $\gamma$ by \cite[\S 10 p. 78 (2)]{Levine-EC} because $e(\gamma) = 0$ as $\gamma$ is a bundle of odd rank. Since $\tO$ is rank $2$, and the characteristic of $k$ does not divide $2 a_i$, we may apply \cite[Theorem 8.1]{Levine-Witt} and conclude that $e(\Sym^{a_i} \tO(1))$ is an integer multiple of a power of $e(\tO(1))$. Since the summands are symmetric under the permutation action of the symmetric group on $r+1$ letters on $B N^{r+1}$, it follows that the Euler class of these summands is an integer multiple of a power of $e$.

We now consider the Euler class of the rest of the summands. Namely, it suffices to show that the Euler class $\epsilon_1$ of the summands with at least two rank $2$ factors is also in the image of $\Z[p_1,\ldots,p_r,e]$. We may again ignore the factors of $\gamma$, as these do not change the Euler class. By Lemma \ref{eboxtimesOa_is_in_Zei}, $\epsilon_1$ is the image of an element of $\Z[e_1^2,\ldots, e_{r+1}^2]$.  Moreover, because each tuple $(a_1, \ldots, a_{r+1} )$ of the direct sum occurs in every permutation, we may choose an element of $\Z[e_1^2,\ldots, e_{r+1}^2] $ which is invariant under the permutation action of the symmetric group on $r+1$ letters and which maps to $\epsilon_1$. Thus, $\epsilon_1$ is in the image of the map $$\Z[\sigma_1(e_1^2,\ldots, e_{r+1}^2), \ldots, \sigma_{r+1}(e_1^2,\ldots, e_{r+1}^2)] \to \emW^*(B N^{r+1}),$$ where $\sigma_i$ denotes the $i$th elementary symmetric polynomial. Since $\sigma_i((e_1^2,\ldots, e_{r+1}^2)) $ is the pullback to $BN$ of $p_i(\Stautbun_{d+1}^* \to B\SL_{d+1}, \emW)$, we have that $\epsilon_1$ is in the image of $\Z[p_1,\ldots,p_r,e]$ as desired.

\end{proof}

The Pontryagin and Euler classes of $\tautbun_{d+1}^* \to \Gr(d,n)$ are pulled back from those of $\tautbun_{d+1}^* \to B\GL_{d+1}$. The $\emW^*$-cohomology and twisted cohomology of $B\GL_{d+1}$ injects into that of $B\SL_{d+1}$, $$\emW^*(B\GL_{d+1}) \oplus \emW^*(B\GL_{d+1}, \det \Gtautbun) \subseteq \emW^*(B\SL_{d+1}),$$ by \cite[Theorem 4.1]{Levine-Witt}. Under this injection, the Pontrygin and Euler classes of $\Gtautbun_{d+1}^* \to B\GL_{d+1}$ are sent to $p_1, \ldots, p_r, p_{r+1}$ and $e$, respectively, so we will let $p_i$ and $e$ denote the corresponding characteristic classes of $\Gtautbun^* \to B\GL_{d+1}$ and $\tautbun_{d+1}^* \to \Gr(d,n)$ as well. Let $\pi: \Gr(d,n) \to \Spec k$ denote the structure map, and $$\pi_*: \emW^*(\Gr(d,n), (\det \tautbun_{d+1}^*)^{\otimes -(n+1)})\to \W(k)$$ the induced push forward on $\emW^*$-cohomology.

\begin{lemma}\label{monomial_top_degree_as_cen-d}
Let $d = 2r+1$ be odd. For any non-negative integers $a_1, \ldots, a_{r+1},b$ such that $\sum (4i) a_i + b(d+1) = (d+1)(n-d)$ and $b \equiv n+1 \mod 2$, the monomial $e^b \prod_{i=1}^{r+1} p_i^{a_i}$ in $\emW^*(\Gr(d,n), (\det \tautbun_{d+1}^* )^{\otimes b})$ is in the image of $\Z[e]$, or in other words, there exists $c$ in $\Z$ such that $e^b \prod_{i=1}^{r+1} p_i^{a_i} = c e^{n-d}$.
\end{lemma}

\begin{proof}
Let $Q$ denote the quotient bundle on $\Gr(d,n)$, defined by the short exact sequence $$0 \to \tautbun_{d+1} \to \calO^{n+1} \to Q \to 0 .$$ In particular, the rank of $Q$ is $n-d$. The non-vanishing Pontryagin classes of $\tautbun_{d+1}$ are $p_1, \ldots, p_{r}, p_{r+1}$ with $e^2 =  p_{r+1}$. Define $s$ so that $n-d = 2s$ or $n-d=2s+1$ depending on whether $n-d$ is odd or even. Let $p_1^{\perp}, \ldots, p_s^{\perp} $ denote the non-vanishing Pontryagin classes of the dual to the quotient bundle $Q^*$ on $\Gr(d,n)$. By \cite[Lemma 15]{Ananyevskiy-SL_projective_bundle_thm}, $$(1+p_1+ \ldots + p_{r+1})(1+ p_1^{\perp} + \ldots + p_s^{\perp})=1 $$ in $W^*(\Gr(d,n))$. Setting the notation $A$ for the ring $$A =  \Z [p_1, \ldots, p_{r}, p_{r+1}, p_1^{\perp}, \ldots, p_s^{\perp} ]/\langle (1+p_1+ \ldots + p_{r} + p_{r+1})(1+ p_1^{\perp} + \ldots + p_s^{\perp}) - 1\rangle, $$ we therefore have a homomorphism $\tau: A \to \emW^*(\Gr(d,n).$ There is a canonical  isomorphism $$\H^*(\CC\Gr(r,r+s); \Z) \cong A, $$ where $\H^*(\CC\Gr(r,r+s); \Z)$ denotes the singular cohomology of the $\CC$-manifold associated to the $\CC$-points of the Grassmannian $\Gr(r,r+s)$, sending the $i$th Chern class of the dual tautological bundle to $p_i$. The top dimensional singular cohomology $\H^{(r+1)s}(\CC\Gr(r,r+s); \Z)$ is isomorphic to $\Z$ by Poincar\'e duality. Under our chosen isomorphism, the monomial $p_{r+1}^s$ corresponds to the top Chern class $c_{(r+1)s}(\tautbun_{r+1}^{\oplus s} \to \CC\Gr(r,r+s))$ of the direct sum of $s$-copies of the dual tautological bundle, which is a generator (with the usual $\CC$-orientations, $c_{(r+1)s}(\tautbun_{r+1}^{\oplus s} \to \CC\Gr(r,r+s))$ counts the number of linear subspaces of dimension $r$ in a complete intersection of $s$ linear hypersurfaces in $\CC\P^{r+s}$, and this number is $1$, cf. Remark \ref{arithmetic_count_d_planes_complete_intersection} and Lemma \ref{discpi_*monomial=1}). Therefore, for any monomial $\prod_{i=1}^{r+1} p_i^{a'_i}$ with $\sum_{i=1}^{r+1} ia'_i = (r+1)s$, there is an integer $c'$ such that \begin{equation}\label{CGrassmannian_generation} \prod_{i=1}^{r+1} p_i^{a'_i} = c' p_{r+1}^s\end{equation} in  $\H^*(\CC\Gr(r,r+s); \Z).$

Since $d$ is odd, $n+1 \equiv n-d \mod 2$, and therefore $b \equiv n-d \mod 2$. Note that if $n-d$ is odd, $b \geq 0$. We may then define a non-negative integer $b'$ by the rule $$b' = \begin{cases}
b/2, {\text if } n-d \equiv 0 \mod 2\\
(b-1)/2, {\text if } n-d \equiv 1 \mod 2 .
\end{cases} $$
With this notation, $n-d-b= 2(s-b')$. Thus $$\sum (4i) a_i = (d+1)(n-d-b) = (d+1)2(s-b') = 4(r+1)(s-b'),$$ whence $$ \sum i a_i = (r+1)(s-b').$$ By Equation~\eqref{CGrassmannian_generation}, there is an integer $c$ so that we have the equality $$ 
p_{r+1}^{b'}\prod_{i=1}^{r+1} p_i^{a_i} = c  p_{r+1}^s
$$ in  $\H^*(\CC\Gr(r,r+s); \Z).$ Applying $\tau$, we see that $$e^{2b'} \prod_{i=1}^{r+1} p_i^{a_i} = c'  e^{2s},$$ which implies the claim, either immediately in the case that $n-d$ is even, or by multiplying by $e$ if $n-d$ is odd.

\end{proof}

\begin{lemma}\label{discpi_*monomial=1}
Let $d$ be odd. Suppose $k$ is a finite field such that $-1$ is a square. For any non-negative integers $a_1,\ldots,a_{r+1},b$ such that $b(d+1)+\sum_{i=1}^{r+1} 4i a_i= (d+1)(n-d)$ and $b \equiv n+1 \mod 2$, the pushforward  $\pi_*(e^b \prod_{i=1}^{r+1} p_i^{a_i})$ has trivial discriminant. 
\end{lemma}

\begin{remark}
The condition $b \equiv n+1 \mod 2$ ensures that $e^b \prod_{i=1}^{r+1} p_i^{a_i}$ lies in the appropriate twist of the Witt-cohomology of $\Gr(d,n)$, i.e. $e^b \prod_{i=1}^{r+1} p_i^{a_i}$ is in $\emW^*(\Gr(d,n),\omega_{\Gr(d,n)}/k)$, as opposed to $\emW^*(\Gr(d,n))$, so that we may apply $\pi_*$. The condition on the sum  $b(d+1)+\sum_{i=1}^{r+1} 4i a_i$ ensures that $e^b \prod_{i=1}^{r+1} p_i^{a_i}$ lies in the $(d+1)(n-d)$-degree $\emW^*$-cohomology of $\Gr(d,n)$, so the codomain of $\pi_*$ is $\W(k)$.
\end{remark}

\begin{proof}
By Lemma \ref{monomial_top_degree_as_cen-d}, it suffices to show that $\disc \pi_*e^{n-d}=1$. We may identify $\pi_*e^{n-d}$ with the Euler number of $\oplus_{j=1}^{n-d} \tautbun^*,$ $$\pi_*e^{n-d} = n(\oplus_{j=1}^{n-d} \tautbun^*, \emW).$$

Let $x_0, \ldots, x_n$ be coordinates on projective space $\P_k^n = \Proj k[x_0,\ldots, x_n]$. The Euler number $ n(\oplus_{j=1}^{n-d} \tautbun^*, \emW) $ can be calculated with the section $\sigma = \oplus_{i=d+1}^n x_i$ as in \S\ref{subsec:poincare-hopf-euler-number}. There is an analogous section $\sigma_{\Z} = \oplus_{i=d+1}^n x_i$ of $\oplus_{j=1}^{n-d} \tautbun^*_{\Z} \to \Gr(d,n)_{\Z}$ defined over $\Z$. The zero locus $\sigma_{\Z} = 0$ is the single $\Z$-point of the Grassmannian associated to the linear subspace of $\P_k^n$ given by $x_{d+1} = x_{d+2} = \ldots = x_n = 0$. 

The vanishing locus of $\sigma_{\Z}$ is the origin of the affine space $\A^{(d+1)(n-d)}_\Z = \Spec \Z[a_{ij}: i =0,\ldots, d, j=d+1,\ldots, n] \hookrightarrow \Gr(d,n)$ whose point $(a_{ij})$ corresponds to the row space of \begin{equation}\label{Grassmannian_coords_aij} 
\begin{pmatrix}
1&0&\cdots & 0&a_{0,d+1} & a_{0,d+2} & \cdots & a_{0,n} \\
0&1&\cdots & 0&a_{1,d+1} & a_{1,d+2} & \cdots & a_{1,n} \\
\vdots &\vdots &\ddots & \vdots &\vdots  & \vdots  & \ddots & \vdots  \\
0&0&\cdots & 1&a_{d,d+1} & a_{d,d+2} & \cdots & a_{d,n} 
\end{pmatrix} .\end{equation} See for example, \cite[Section 3.2]{EisenbudHarris}, for a description of these coordinates on this open affine of the Grassmannian. Let $e_0,\ldots, e_n$ denote the standard basis of the free module of rank $n+1$. Let $\tilde{e}_0,\ldots, \tilde{e}_n$ denote the basis consisting of the row space of  \eqref{Grassmannian_coords_aij} followed by $e_{d+1},e_{d+1},\ldots,e_{n}$. Let $\tilde{x}_0,\ldots, \tilde{x}_n$ denote the dual basis to $\tilde{e}_0,\ldots, \tilde{e}_n$. Over this $\A^{(d+1)(n-d)}_\Z$, the vector bundle $\tautbun^*_{\Z}$ is trivialized by the basis of sections $\{\tilde{x}_0,\ldots, \tilde{x}_d\}$. Then we may interpret $\sigma$ as a function $\A^{(d+1)(n-d)}_\Z \to \A^{(d+1)(n-d)}_\Z$. Namely, $$x_u ((a_{ij}))= \sum_{l=0}^{d} (x_u(a_{ij}))(\tilde{e_l})  \tilde{x}_l$$ for $$u = d+1, \ldots, n.$$ As a subscheme, $\sigma_{\Z} = 0$ is therefore the zero locus of $$ (x_u(a_{ij}))(\tilde{e_l}) = a_{l,u}$$ for $l=0,\ldots, d$ and $u = d+1, \ldots, n$. Thus the subscheme of $\Gr(d,n)_{\Z}$ given by $\{\sigma_{\Z} = 0\}$ is a section of the structure map $\Gr(d,n)_{\Z} \to \Spec \Z$. In particular, it is finite and \'etale of rank $1$. It follows that the Jacobian of $\sigma$ (which is described further at the beginning of Section \ref{subsection:arithmetic_count_d-planes_ci}) $$\Jac \sigma \in \Hom (\det T\Gr(d,n)\vert_{\{\sigma_{\Z} = 0\}}, \det \oplus_{j=1}^{n-d} \tautbun^*_{\Z}\vert_{\{\sigma_{\Z} = 0\}} )$$ is nowhere vanishing. Thus under the relative orientation $$\Hom (\det T\Gr(d,n), \det \oplus_{j=1}^{n-d} \tautbun^*_{\Z} ) \cong L^{\otimes 2},$$ we have that $\Jac \sigma$ is a nowhere vanishing section of the restriction of $L^{\otimes 2}$. Thus $\langle \Jac \sigma_{k}  \rangle$ is either $\langle -1 \rangle$ or $\langle 1 \rangle$; but $\lra{-1} = 1$ by assumption. Since $n(\oplus_{j=1}^{n-d} \tautbun_k^*, \emW) = \langle \Jac \sigma_{k}  \rangle$ (cf. Example \ref{F=0_is_one_point}), this proves the claim.

\end{proof}
\subsection{An arithmetic count of the $d$-planes on a complete intersection in projective space}\label{subsection:arithmetic_count_d-planes_ci}

A complete intersection of hypersurfaces $$X=\{F_1 =F_2 = \ldots = F_j=0\} \subset \mathbb{P}^n$$ with $F_i$ of degree $n_i$ gives rise to a section $\sigma$ of $V= \oplus_{i=1}^j \Sym^{n_i} \tautbun^*$, defined by $\sigma(\mathbb{P}L) = \oplus_{i=1}^j F_i\vert_L$, where $L$ is any $d+1$ dimensional linear subspace of $\mathbb{A}^{n+1}$ containing the origin, and $\mathbb{P}L$ denotes the corresponding point of $\Gr(d,n)$. The zeros of $\sigma$ are then precisely the $d$-planes in $X$. See for example \cite{Debarre_Manivel}. By \cite[Th\'eor\`eye 2.1]{Debarre_Manivel}, the closed subscheme $\{ \sigma = 0\}$ of $\Gr(d,n)$ is smooth for general $X$ and either of the expected dimension $(d+1)(n-d) - \Rrank V = (d+1)(n-d) - \sum_{i=1}^j {n_i + d \choose d}$ or empty, with the empty case occurring exactly when one or both of $(d+1)(n-d) - \Rrank V$ and $n-2r-j$ is less than $0$. In particular, when $(d+1)(n-d) - \Rrank V =0$, the zeros of $\sigma$ are isolated and \'etale over $k$ for a general complete intersection $X$. The canonical relative orientation (Remark \ref{canonical_rel_orientation_sym}) of $V$ determines an isomorphism $\Hom(\det T\Gr(d,n), \det V) \cong L^{\otimes 2}$ for a line bundle $L$ on $\Gr(d,n)$. The Jacobian determinant $\Jac \sigma$ at a zero $p$ of $\sigma$ is an element of the fiber of the vector bundle $\Hom(\det T\Gr(d,n), \det V)$ at $p$. Choosing any local trivialization of $L$, we have a well-defined element $\Jac \Sigma (p)$ in $k(p)/(k(p)^*)^2$, which can also be computed by choosing a local trivialization of $V$ and local coordinates of $\Gr(d,n)$ compatible with the relative orientation and computing $\Jac \Sigma (p) = \det \big( \frac{\partial \sigma_k}{\partial x_l} \big)$.

\begin{corollary}\label{arithmetic_count_d_planes_complete_intersection}
Let $X=\{F_1 =F_2 = \ldots = F_j=0\} \subset \mathbb{P}^n$ be a general complete intersection of hypersurfaces $F_i = 0$ of degree $n_i$ in $\mathbb{P}^n_k$ projective space over a field $k$. Suppose that $\sum_{i=1}^j {n_i + d \choose d}= (d+1)(n-d) $ and $\sum_{i=1}^j \frac{n_i}{d+1}{n_i + d \choose d} + n+1$ is even. Then $$ \sum_{d\text{-planes } \mathbb{P}L \text{ in }X} \tr_{k(\mathbb{P}L)/k} \langle \Jac~\sigma(\mathbb{P}L) \rangle  =  \frac{ n_{\CC} +n_{\RR} }{2} + \frac{n_{\CC} +  n_{\RR}}{2}\lra{-1} $$ where \begin{itemize}
\item $k(\mathbb{P}L)$ denotes the residue field of $\mathbb{P}L$ viewed as a point on the Grassmannian
\item  $n_{\CC}$ (respectively $n_{\RR}$) is the topological Euler number of the complex (respectively real) vector bundle associated to the algebraic vector bundle $V= \oplus_{i=1}^j \Sym^{n_i} \tautbun^*$ given the canonical relative orientation (\ref{canonical_rel_orientation_sym}) 
\item and $\Jac \sigma$ is the Jacobian determinant.
\end{itemize}
\end{corollary}

\begin{proof}
By \cite[Th\'eor\`eme 2.1]{Debarre_Manivel}, the zeros of $\sigma$ are isolated and \'etale over $k$. It follows \cite[p.18, Proposition 34]{CubicSurface} that for a zero of $\sigma$ corresponding to the $d$-plane $\mathbb{P}L$, the local index is computed $\ind^{\PH}_{\mathbb{P}L} \sigma = \tr_{k(L)/k} \langle \Jac ~\sigma(\mathbb{P}L) \rangle$. See Section \ref{subsec:poincare-hopf-euler-number} for the definition of the notation $\ind^{\PH}$. Thus $ n^{\PH}(V, \sigma) = \sum_{d\text{-planes } \mathbb{P}L \text{ in }X} \tr_{k(L)/k} \langle \Jac ~\sigma(\mathbb{P}L) \rangle$ by Definition \ref{def:ePH}. Corollary \ref{co:eoplussym} computes  $ n^{\PH}(V, \sigma)$. Proposition \ref{nRRnCCaretopEnums} shows that $n_{\RR}$ and $n_{\CC}$ are the claimed topological Euler numbers.
\end{proof}

\begin{remark}
Note that Corollary \ref{arithmetic_count_d_planes_complete_intersection} is a weighted count of the dimension $d$ hyperplanes on the complete intersection $X$, depending only on $n_1,\ldots,n_j$ and not on the choice of polynomials $F_1,\ldots, F_j$ as long as these are chosen generally. 
\end{remark}

\begin{example}
Examples where Corollary \ref{arithmetic_count_d_planes_complete_intersection} applies include: i) lines on a degree $2n-1$ hypersurface of dimension $n$, ii) $3$-planes on a degree $d$ hypersurface of dimension $2 + \frac{1}{3}{d+3 \choose 3}$, when this is an integer. iii) Lines on a complete intersection of two degree $n-2$ polynomials in $\mathbb{P}^n$ for $n$ odd.
\end{example}

Matthias Wendt's oriented Schubert calculus shows that enriched intersections of Schubert varieties are determined by the $\RR$ and $\CC$ realizations in the same manner \cite[Theorem 8.6]{Wendt-oriented_schubert}, as well as giving enumerative applications \cite[Section 9]{Wendt-oriented_schubert}. 

For $d=j=1$ and $n_1=3$, Corollary \ref{co:eoplussym} is work of Kass and the second named author \cite{CubicSurface} over a general field. For $d=1$ and general $j$ and $n_i$, it is work of M. Levine \cite{Levine-Witt} over a perfect field either of characteristic $0$ or of characteristic prime to $2$ and the odd $n_i$s. Our result eliminates the assumption on the characteristic, and generalizes to arbitrary relatively orientable $d,n,n_1,\ldots,n_j$. 

In order to obtain an enumerative theorem whose statement is independent of $\A^1$-homotopy theory, one needs an arithmetic-geometric interpretation of the local indices:

\begin{question}
Can the local indices $\ind^{\PH}_{\mathbb{P}L} \sigma = \tr_{k(L)/k} \langle \Jac ~\sigma(\mathbb{P}L) \rangle$ be expressed in terms of the arithmetic-geometry of the $d$-plane $\mathbb{P}L$ on $X$?
\end{question}

Such expressions are available over $\mathbb{R}$ for $d=j=1$ \cite{FK-Segre_indices}, and over a field $k$ of characteristic not $2$, for lines on a cubic surface \cite{CubicSurface}, lines on a quintic $3$-fold \cite{Pauli-Quintic_3_fold}, and points on a complete intersection of hypersurfaces \cite{McKean-Bezout}. S. Pauli has interesting observations on such results for lines on the complete intersection of two cubics in $\P^5$. Dropping the assumption that the zeros of $\sigma$ are isolated, she can compute contributions from infinite families of lines on a quintic $3$-fold in some cases \cite{Pauli-Quintic_3_fold}. An alternative point of view in terms of (S)pin structures for $d=j=1$ and $n_1=5$, as well as computations of the real Euler number is discussed \cite[Example 1.6, Theorem 8.8]{solomon06}. 

\begin{example}\label{3-plane-counts}
The computations of the Euler classes of $\mathbb{C}$ and $\mathbb{R}$-points in Finashin and Kharlamov's paper \cite[p. 190]{Finashin-Kharlamov-3_planes} imply the following enriched counts of $3$-planes on hypersurfaces over any field $k$.

$ n( \Sym^3 \tautbun^* \to \Gr(3,8))= 160839 \lra{1}+ 160650 \lra{-1}$ corresponds to an enriched count of $3$-planes in a $7$-dimensional cubic hypersurface. Namely, for a general degree $3$ polynomial $F$ in $9$ variables, the corresponding cubic hypersurface $X \subset \mathbb{P}^8$ contains finitely many $3$-planes as discussed above and $$\sum_{3-\text{planes } P \subset X} \tr_{k(L)/k} \langle \Jac ~\sigma_F(P) \rangle=160839 \lra{1}+ 160650 \lra{-1},$$ where $\sigma_F$ is the section of  $\Sym^3 \tautbun^*$ defined by $\sigma_F[P] = F\vert_P$.

Similarly, $ n( \Sym^5 \tautbun^* \to \Gr(3,17))=$ $$32063862647475902965720976420325 \lra{1}+ 32063862647475902965683320692800 \lra{-1}$$ corresponds to an enriched count of $3$-planes in a $16$-dimensional degree $5$-hypersurface.
\end{example} 

\section{Indices of sections of vector bundles and $\A^1$-degrees} \label{sec:A1-degrees}
\subsection{$\A^1$-degrees} \label{subsec:A1-degree}
Recall the following.
\begin{definition}[local $\A^1$-degree]
Let $S$ be a scheme, $X \in \Sch_S$ and $F: X \to \A^n_S$ be a morphism.
We say that $F$ has an isolated zero at $Z \subset X$ if $Z$ is a clopen subscheme of $Z(F)$ such that $Z/S$ is finite.

Now let $X \subset \A^n_S$ be open and suppose that $F$ has an isolated zero at $Z \subset X$.
We define the \emph{local $\A^1$-degree of $F$ at $Z$} \[ \ideg_Z(F) \in [\P^n/\P^{n-1}, \P^n/\P^{n-1}]_{\SH(S)} \wequi [\1, \1]_{\SH(S)} \] as the morphism corresponding to the unstable map \[ \P^n/\P^{n-1} \to \P^n / \P^n \setminus Z \wequi X / X \setminus Z \hookrightarrow X/X \setminus Z(F) \xrightarrow{F} \A^n/\A^n \setminus 0 \wequi \P^n/\P^{n-1}. \]
Here we use that by assumption $Z(F) \wequi Z \amalg Z'$, and hence $X/X \setminus Z(F) \wequi X/X \setminus Z \amalg X/X \setminus Z'$.
\end{definition}

\begin{example}
If $S = Spec(k)$ is the spectrum of a field, then an isolated zero $z \in \A^n_k$ of $F: \A^n_k \to \A^n_k$ in the usual sense is also an isolated zero $\{z\} \subset \A^n_k$ in the above sense, and $\ideg_z(F) \in \GW(k)$ is the usual local $\A^1$-degree of \cite[Definition 11]{kass2016class}.
\end{example}

\begin{lemma} \label{lemm:voevodsky-transfers}
Let $X, Y \in \Sm_S$ and $(Z, U, \phi, g)$ be an equationally framed correspondence from $X$ to $Y$ \cite[Definition 2.1.2]{EHKSY}; in other words $Z \subset \A^n_X$, $U$ is an \'etale neighbourhood of $Z$, $g: U \to Y$ and $\phi: U \to \A^n$ is a framing of $Z$.
Then the following two morphisms are stably homotopic \[ T^n \wedge X_+ \wequi (\P^1)^{\wedge n} \wedge X_+ \to (\P^1)^{\times n}_X/(\P^1)^{\times n}_X \setminus Z \wequi U/U \setminus Z \xrightarrow{\phi,g} T^n \wedge Y_+ \] \[ T^n \wedge X_+ \wequi \P^n/\P^{n-1} \wedge X_+ \to \P^n_X / \P^n_X \setminus Z \wequi U/U \setminus Z \xrightarrow{\phi,g} T^n \wedge Y_+. \]
\end{lemma}
\begin{proof}
For $E \in \SH(S)$, precomposition with the first morphism (desuspended by $T^n$) induces a map $E(Y) \to E(X)$ known as the \emph{Voevodsky transfer}.
Precomposition with the second map induces an ``alternative Veovedsky transfer''.
It suffices (by the Yoneda lemma) to show that these transfer maps have the same effect (even just on $\pi_0$), for every $E$.
In \cite[Theorem 3.2.11]{EHKSY2}, it is shown that the Voevodsky transfer coincides with a further construction known as the \emph{fundamental transfer}.
In that proof, all occurrences of $(\P^1)^{\wedge n}$ can be replaced by $\P^n/\P^{n-1}$; one deduces that the alternative Voevodsky transfer also coincides with the fundamental transfer.

The result follows.
\end{proof}

\begin{corollary} \label{cor:deg-vs-correspondence}
Let $U \subset \A^n_S$, $F: U \to \A^n$ have an isolated zero along $Z \subset U$.
Then $\deg_Z(F) \in [\1, \1]_S$ is the same as the endomorphism given by the equationally framed correspondence defined by $F$.
\end{corollary}
\begin{proof}
By definition, $\deg_Z(F)$ is given by the second morphism of Lemma \ref{lemm:voevodsky-transfers}, whereas the equationally framed correspondence is given by the first morphism.
The result follows.
\end{proof}

Recall also the following.
\begin{proposition}[\cite{EHKSY2}, Theorem 3.3.10] \label{prop:corr-transfer}
Let $\alpha: S \xleftarrow{\varpi,\tau} Z \to S$ be a tangentially framed correspondence from $S$ to $S$ over $S$.
Then the trivialization $\tau$ of $L_{Z/S}$ induces a transfer map \[ \varpi_*: [\1, \1]_Z \to [\1, \1]_S, \] where $\varpi: Z \to S$ is the projection.
The endomorphism of $\1_S$ corresponding to $\alpha$ is given by $\varpi_*(1)$.
\end{proposition}

\subsection{Main result}
Let $X/S$ be smooth, $V/X$ a relatively oriented vector bundle with very non-degenerate section $\sigma$ and zero scheme $Z$ (which is thus finite over $S$).
Let $Z' \subset Z$ be a clopen component and suppose there are coordinates $(\psi, \varphi, \sigma')$ for $(V, X, \sigma, \rho, Z')$ as in Definition \ref{def:coordinates}.
Then $\sigma' = (F_1, \ldots, F_d)$ determines a function $F: \A^d_S \to \A^d_S$, and $\varphi(Z')$ is an isolated zero of $F$.

\begin{theorem} \label{thm:identify-degree}
Assumptions and notations as above.
Let $E \in \SH(S)$ be an $\SL$-oriented ring spectrum with unit map $u: \1 \to E$.
Then \[ \ind_Z(V, \sigma, \rho, E) = u_*\deg_{\varphi(Z')}(F) \in E^0(S). \]
\end{theorem}
\begin{proof}
By Corollary \ref{cor:deg-vs-correspondence}, Proposition \ref{prop:corr-transfer} and \S\ref{subsub:functoriality-E} we have $u_*\deg_{\varphi(Z')}(F) = \ind_Z(\sigma',o',E)$, where $o'$ is the canonical relative orientation of $\scr O_{\A^n}^n/\A^n$.
The result now follows from Proposition \ref{prop:meta-compute}.
\end{proof}

\begin{example}
Suppose $S = Spec(k)$ is a field.
Then for well-chosen $E$ (e.g. $E = \KO$ or $E = \H\tilde\Z$), the unit map \[ u_*: \GW(k) \wequi [\1, \1]_{\SH(k)} \to E^0(k) \] is an isomorphism.
We deduce that $\ind_z(V, \sigma, \rho, E)$ is essentially the same as $\deg_{\varphi(z)}(F).$
\end{example}

\section{Euler numbers in $\KO$-theory and applications}
\label{sec:Euler-KO}
As explained in \S\ref{subsec:13} and \S\ref{subsec:refined-euler}, we have the motivic ring spectrum $\KO$ related to Hermitian $K$-theory made homotopy invariant, and associated theories of Euler classes and Euler numbers.
Using (for example) the construction in \S\ref{app:KO}, we can define $\KO$ even if $1/2 \not\in S$.
There is still a map $\GW(S) \to \KO(S)$, however we do not know if this is an equivalence, even if $S$ is regular (but we do know this if $S$ is regular and $1/2 \in S$).

\begin{example}\label{KO-Euler_class_Koszul}
Let $V \to X$ be a vector bundle and $\sigma$ a section.
Suppose $1/2 \in S$.
Then the refined Euler class of $V$ in $\KO$-theory is given by $e(V, \sigma, \KO) = [K(V, \sigma)]$, the class of the Koszul complex.
Indeed via Remark \ref{rmk:refined-euler-levine} it suffices to show the analogous result for Thom classes, which is stated on \cite[p. 34]{levine2018motivic}.
\end{example}

Recall that for any lci morphism $f$ we put $\widetilde \omega_f = \widetilde \det L_f$ and $\omega_f = \det L_f$.
We show in Proposition \ref{prop:f!-Lf} that for $f: X \to Y$ an lci morphism, we have $f^!(\scr O) \wequi \widetilde \omega_f$.
Via Proposition \ref{prop:f!-Lf}, coherent duality (i.e. the adjunction $f_* \dashv f^!$) thus supplies us with a canonical \emph{trace map} \[ \eta_f: f_* \widetilde\omega_f \wequi f_*f^! \scr O_Y \to \scr O_Y \in D(Y), \] provided that $f$ is also \emph{proper}.
One expect that this can be used to build a map \[ \hat f_*: \GW(X, f^! \scr L) \to \GW(Y, \scr L), \] and moreover that the following diagram commutes
\begin{equation*}
\begin{CD}
\GW(X, f^! \scr L) @>>> \KO(X, f^! \scr L) \\
@V{\hat{f}_*}VV                 @V{f_*}VV \\
\GW(Y, \scr L) @>>> \KO(Y, \scr L).
\end{CD}
\end{equation*}
If we assume that $1/2 \in S$, replace $\GW$ by $\W$ and $\KO$ by $\KW$, then maps $\hat f_*$ can be defined and studied using the ideas from \S\ref{subsec:local-indices}; see also \cite{calmes2008push}.
Levine-Raksit \cite{levine2018motivic} show that the (modified) diagram commutes provided that $X$ and $Y$ are smooth over a common base with $1/2 \in S$.

If instead we assume that $f$ is finite syntomic, then the analogous result is proved (for $\GW \to \KO$ and without $1/2 \in S$) in Corollary \ref{cor:KO-transfer}.
This is the only case that we shall use in the rest of this section.
Recall the construction of the Scheja-Storch form $\SSform$ from Definition \ref{def:SSform}.

\begin{corollary} \label{cor:KO-scheja}
Let $X \in \Sm_S$, $V/X$ a relatively oriented vector bundle with a very non-degenerate section $\sigma$, and $Z$ a clopen component of the zero scheme $Z(\sigma)$.
Suppose there exists coordinates $(\psi,\varphi,\sigma')$ around $Z$, as in Definition \ref{def:coordinates}.

Then \[ \ind_Z(\sigma, \rho, \KO) = [\SSform(\varphi(U), \sigma', S)] \in \KO^0(S). \]
\end{corollary}
\begin{proof}
By Proposition \ref{prop:meta-compute}, we may assume that $\psi = \id$ and so on; so in particular $X \subset \A^n_S$.
The result now follows from the identification of the transfers in Corollary \ref{cor:KO-transfer} (telling us that the index is given by the trace form from coherent duality) and Theorem \ref{thm:introd-2} (identifying the coherent duality form with the Scheja-Storch form).
\end{proof}

\begin{corollary} \label{cor:framed-endo}
Let $S$ be regular semilocal scheme over a field $k$ of characteristic $\ne 2$.
\begin{enumerate}
\item Let $\varpi: S' \to S$ be a finite syntomic morphism, and $\tau$ a trivialization of $L_{\varpi} \in K(S')$.
  Then the associated endomorphism of the sphere spectrum over $S$ is given under the isomorphism $[\1, \1]_S \wequi GW(S)$ of \cite[Theorem 10.12]{bachmann-norms} by the symmetric bilinear form \[ \varpi_*(\scr O_{S'}) \otimes \varpi_*(\scr O_{S'}) \to \varpi_*(\scr O_{S'}) \stackrel{\det(\tau)}{\wequi} \varpi_*(\omega_{S'/S}) \xrightarrow{\eta_{\varpi}} \scr O_S. \]
\item Let $U \subset \A^n_S$ be open, $F: U \to \A^n$ have an isolated zero along $Z \subset U$.
  Then \[ \deg_Z(F) = [\SSform(U, F, S)] \in [\1, \1]_S \wequi GW(S). \]
\end{enumerate}
\end{corollary}
\begin{proof}
The proof of \cite[Theorem 10.12]{bachmann-norms} shows that the unit map $u: \1 \to \KO$ induces the isomorphism $[\1, \1]_S \wequi GW(S)$.
(1) is an immediate consequence of Proposition \ref{prop:corr-transfer}, Lemma \ref{lemm:gysin-compat}(1) and Corollary \ref{cor:KO-transfer}.
Via Corollary \ref{cor:deg-vs-correspondence} and Theorem \ref{thm:introd-2}, (2) is a special case of (1).
This concludes the proof.
\end{proof}

\begin{corollary}
Let $X$ be essentially smooth over a field $k$ of characteristic $\ne 2$.
Let $\varpi: X' \to X$ be a finite syntomic morphism, and suppose given an orientation $\omega_{X'/X} \stackrel{\rho}{\wequi} \scr L^{\otimes 2}$.
Consider the induced transfer \[ \varpi_*: \H\tilde\Z^0(X') \stackrel{\rho}{\wequi} \H\tilde\Z^0(X', \omega_{X'/X}) \to \H\tilde\Z^0(X) \wequi \ul{GW}(X). \]
Then $\varpi_*(1)$ is given by the image in $\ul{GW}(X)$ of the symmetric bilinear form on \[ \varpi_*(\scr L) \otimes \varpi_*(\scr L) \to \varpi_*(\scr L^{\otimes 2}) \stackrel{\rho}{\wequi} \varpi_*(\omega_{X'/X}) \xrightarrow{\eta_{\varpi}} \scr O_X. \]
\end{corollary}
\begin{proof}
Using unramifiedness of $\ul{GW}$ \cite[Lemma 6.4.4]{morel2005stable}, we may assume that $X$ is the spectrum of a field.
Then $X'$ is semilocal, so $\scr L \wequi \scr O$ and we obtain (up to choosing such an isomorphism) $\omega_{X'/X} \stackrel{\rho'}{\wequi} \scr O$.
Since $L_{X'/X}$ has constant rank (namely $0$), it follows from \cite[Lemma 1.4.4]{bruns1998cohen} that $L_{X'/X} \wequi 0 \in K(X')$.
The set of homotopy classes of such trivializations is given by $K_1(X')$, which maps surjectively (via the determinant) onto $\scr O^\times(X')$.
It follows that there exists a trivialization $\tau: 0 \wequi L_{X'/X} \in K(X')$ such that $\det(\tau) = \rho'$.
Hence we have a commutative diagram
\begin{equation*}
\begin{CD}
\varpi_*(\scr L) \otimes \varpi_*(\scr L) @>>> \varpi_*(\scr L^{\otimes 2}) @>{\rho}>> \varpi_*(\omega_{X'/X}) @>{\eta_\varpi}>> \scr O \\
@|                                                @|                                   @|                                      @|    \\
\varpi_*(\scr O) \otimes \varpi_*(\scr O) @>>> \varpi_*(\scr O) @>{\det \tau}>> \varpi_*(\omega_{X'/X}) @>{\eta_\varpi}>> \scr O. \\
\end{CD}
\end{equation*}
It follows from Corollary \ref{cor:framed-endo} that the bottom row is the form $\varpi_*(1)$ arising from the orientation $\rho$; by what we just said this is the same as the top row, which is the form we were supposed to obtain.

This concludes the proof.
\end{proof}

We also point out the following variant.
\begin{corollary}
Let $l/k$ be a finite extension of fields, $1/2 \in k$.
Them Morel's absolute transfer \cite[\S5.1]{A1-alg-top} $\tr_{l/k}: GW(l, \omega_{l/k}) \to GW(k)$ is given as follows.
Let $\phi: V \otimes_l V \to l$ be an element of $GW(l)$, $\alpha \in \omega_{l/k}^\times$.
Then \[ \tr_{l/k}(\phi \otimes \alpha) = [V \otimes_k V \to V \otimes_l V \xrightarrow{\phi} l \stackrel{\alpha}{\wequi} \omega_{l/k} \xrightarrow{\eta_{l/k}} k], \] where $\eta_{l/k}: \omega_{l/k} \to k$ is the ($k$-linear) trace map of coherent duality (see \S\ref{subsec:preliminaries-cotangent}).
\end{corollary}
\begin{proof}
Immediate from \cite[Proposition 4.3.17]{EHKSY2} (telling us that Morel's transfer coincides with the one from \S\ref{sec:gysin}) and Corollary \ref{cor:KO-transfer}.
\end{proof}

\begin{remark} \label{rmk:KW-reprove}
Corollary \ref{cor:framed-endo}(2) generalizes the main result of \cite{kass2016class}, at least for fields of characteristic $\ne 2$.
\end{remark}

\begin{remark}
We expect that all of the results in this section extend to fields of characteristic $2$ as well.
This should be automatic as soon as $\KO$ is shown to represent $\GW$ in this situation (over regular bases, say).
\end{remark}

\appendix

\section{$\KO$ via framed correspondences} \label{app:KO}
In this section we will construct a strong orientation on $\KO$, and identify some of the transfers.
We would like to thank M. Hoyois for communicating these results to us.
For another approach to parts of the results in this section see \cite{lopez2017ring}.

We shall make use of the technology of framed correspondences  \cite{EHKSY}.
We write $\Cor^\fr(S)$ for the symmetric monoidal $\infty$-category of smooth $S$-schemes and tangentially framed correspondences.
Denote by $\FSynor \in \CAlg(\PSh_\Sigma(\Cor^\fr(S)))$ the stack of finite syntomic schemes $Y/X$ together with a choice of trivialization $\det L_{Y/X} \wequi \scr O$, with its standard structure of framed transfers.
This is constructed in \cite[Example 3.3.4]{EHKSY3}.

Write $\Bil \in \CAlg(\PSh_\Sigma(\Sch_S))$ for the presheaf sending $X$ to the $1$-groupoid of pairs $(V, \phi)$ with $V \to X$ a vector bundle and $\phi: V \otimes V \to \scr O_X$ a non-degenerate, symmetric bilinear form.
The commutative monoid structure is given by $\otimes$.
If $f: X \to Y$ is a finite syntomic morphism then a choice of trivialization $\widetilde\det L_f \wequi \omega_f \wequi \scr O_X$ induces an additive map $\hat f_*: \Bil^\wequi(X) \to \Bil^\wequi(Y)$; see e.g. \S\ref{subsec:local-indices}.

\begin{theorem}
There exists a lift $\Bil \in \CAlg(\PSh_\Sigma(\Cor^\fr(S)))$ with the transfers given by maps of the form $\hat f_*$, together with a morphism $\FSynor \to \Bil \in \CAlg(\PSh_\Sigma(\Cor^\fr(S)))$.
\end{theorem}
The morphism $\FSynor \to \Bil$ is informally described as follows: a pair ($f: X \to Y$ finite syntomic, $\omega_f \wequi \scr O$) is sent to $f_*(\scr O)$, where $\scr O \in \Bil(X)$ denotes the vector bundle $\scr O_X$ with its canonical symmetric bilinear pairing $\scr O_X \otimes \scr O_X \to \scr O_X$.
\begin{proof}
Denote by $K^\circ \in \PSh_\Sigma(\Sm_S)$ the rank $0$ part of the $K$-theory presheaf and by $\Cor^\fr((\Sm_S)_{/K^\circ})$ the subcategory of the category constructed in \cite[\S B]{EHKSY3} on objects $(X, \xi)$ with $X \in \Sm_S$, $\xi$ of rank $0$, and morphisms those spans whose left leg is finite syntomic.
There are symmetric monoidal functors $\gamma: (\Sm_S)_{/K^\circ} \to \Cor^\fr((\Sm_S)_{/K^\circ})$ and $\delta: \Cor^\fr(S) \to \Cor^\fr((\Sm_S)_{/K^\circ})$.

We first lift $\Bil$ to $\PSh_\Sigma(\Cor^\fr((\Sm_S)_{/K^\circ}))$; we let $\Bil(X, \xi)$ be the $1$-groupoid of vector bundles with a symmetric bilinear form for the duality $\iHom(\ph, \det \xi)$.
Since $\Bil$ is $1$-truncated, we only need to specify a finite amount of coherence homotopies, so this can be done by hand.
Since $\delta$ is symmetric monoidal $\delta^*$ is lax symmetric monoidal and hence $\delta^*(\Bil)$ produces the desired lift.

Let $K' \to K^\circ$ denote the fiber of the determinant map; in other words this is the rank $0$ part of $K^\SL$.
This defines an object of $\CAlg(\PSh_\Sigma(\Sm_S)_{/K^\circ}) \wequi \CAlg(\PSh_\Sigma((\Sm_S)_{/K^\circ}))$ and $\gamma K' \wequi \FSynor$ \cite[Example 3.3.4 and after Example 3.3.6]{EHKSY3}.
To conclude the proof it hence suffices to construct a map $K' \to \gamma^*(\Bil) \in \CAlg(\PSh_\Sigma(\Sm_S)_{/K^\circ})$.
Again this only needs a finite amount of coherences; the desired map sends $(\xi, \phi)$ with $\xi \in K^\circ(X)$ and $\phi: \det(\xi) \wequi \scr O_X$ to $(\scr O, \phi')$ where $\phi': \scr O \otimes \scr O \wequi \scr O \stackrel{\phi}{\wequi} \det \xi$.
\end{proof}

Since group-completion and Zariski localization commute with the forgetful functor $\PSh_\Sigma(\Cor^\fr_S) \to \mathrm{CMon}(\PSh_\Sigma(\Sm_S)) \to \PSh_\Sigma(\Sm_S)$ \cite[Proposition 3.2.14]{EHKSY}, we deduce that $(L_\mathrm{Zar} \Bil^{gp})(X) \wequi \KO(X)$ \cite[Definitions 1.5 and 2.2]{hornbostel2005a1}; we denote this presheaf by $\GW$.
There is thus a canonical Bott element $\beta \in [T^4, \GW]$.
\begin{proposition} \label{prop:framed-KO}
There is a canonical equivalence $(\Sigma^\infty_\fr \GW)[\beta^{-1}] \wequi \KO$, at least if $1/2 \in S$.
\end{proposition}
\begin{proof}
The spectrum $(\Sigma^\infty_\fr \GW)[\beta^{-1}] \in \SH^\fr(S)$ can be modeled by the framed $T^4$-prespectrum $(\GW, \GW, \dots)$ with the bonding maps given by multiplication by $\beta$.
Under the equivalence $\SH^\fr(S) \wequi \SH(S)$ \cite{hoyois2018localization}, this corresponds to the same prespectrum with transfers forgotten.
This is $\KO$ by definition.
\end{proof}

In particular we have constructed an $\scr E_\infty$-structure on $\KO$.

\begin{corollary} \label{cor:KO-orient}
There is a morphism $\MSL \to \KO \in \CAlg(\SH(S))$, at least if $1/2 \in S$.
\end{corollary}
\begin{proof}
Take the morphism $\Sigma^\infty_\fr \FSynor \to \Sigma^\infty_\fr \Bil \to \Sigma^\infty_\fr \GW \to \KO$ and use that $\Sigma^\infty_\fr \FSynor \wequi \MSL$ \cite[Theorem 3.4.3(i)]{EHKSY3}.
\end{proof}

\begin{corollary} \label{cor:KO-transfer}
Let $f: Z \to S$ be a finite syntomic morphism, with $1/2 \in S$, $\tau: L_{Z/S} \wequi 0 \in K(Z)$ a trivialization, $\phi: V \otimes V \to \scr O_Z$ be a non-degenerate, symmetric bilinear form (defining an element $[\phi] \in \KO^0(Z)$).
Then \[ [\hat f_*(\phi)] = f_*([\phi]) \in \KO^0(S). \]
Here $f_*$ denotes the transfer arising from the six functors formalism, and $\hat f_*$ denotes the transfer constructed above using coherent duality.
\end{corollary}
\begin{proof}
We first give a simplified proof assuming that $S$ is affine.
Denote by $\KO^\fr \in \SH^\fr(S)$ a lift of $\KO$.
By \cite[Theorem 3.3.10]{EHKSY2}, for any morphism $p: Z \to Y$ with $Y \in \Sm_S$ and a form $\psi$ on $Y$, the six functors transfer of $p^*([\psi])$ along $f$ coincides with the framed transfer of $[\psi]$ along the correspondence $S \xleftarrow{\tau} Z \to Y$.
By Proposition \ref{prop:framed-KO} we can take $\KO^\fr = \Sigma^\infty_\fr \GW[\beta^{-1}]$; it follows that there is a map $\Bil \to \Omega^\infty_\fr \KO^\fr \in \PSh_\Sigma(\Cor^\fr)$.
We thus deduce that \[ \hat f_*(p^* \psi) = f_*(p^*[\psi]) \in \KO^0(S). \]
The result would follow if there exist $p$, $\psi$ with $p^* \psi = \phi$.
Under our simplifying assumption that $S$ is affine, this is always the case; see \cite[Proposition A.0.4 and Example A.0.6(5)]{EHKSY3}.

To deal with the general case, we begin with some constructions.
Given $F \in \PSh(\Sm_S)$, denote by $f^* F \in \PSh(\Sm_Z)$ the left Kan extension.
If $F$ comes from a presheaf with framed transfers, then the transfers along (base changes of) $f$ assemble to a map $f_*f^* F \to F$.
Given $E \in \SH(S)$, we obtain $\Omega^\infty E \in \PSh(\Sm_S)$ and can apply this construction.
On the other hand there is a map $f_* \Omega^\infty f^* E \to \Omega^\infty E$ coming from the six functors transfer as well as a map $f^* \Omega^\infty E \to \Omega^\infty f^* E$, and \cite[Theorem 3.3.10]{EHKSY2} implies that the following diagram commutes
\begin{equation*}
\begin{tikzcd}
f_* f^*  \Omega^\infty E \ar[r] \ar[d] & f_* \Omega^\infty f^* E \ar[dl] \\
\Omega^\infty E.
\end{tikzcd}
\end{equation*}

Now we continue with the proof.
We have a commutative diagram
\begin{equation*}
\begin{tikzcd}
f_*f^* \Bil_S \ar[r] \ar[d] & f_* \Bil_Z \ar[dl] \\
\Bil_S;
\end{tikzcd}
\end{equation*}
here $\Bil_T \in \PSh(\Sm_T)$ denotes the stack of symmetric bilinear forms, the map $f_*f^* \Bil_S \to \Bil_S$ comes from the above construction with $F = \Bil_S$, the map $f_* \Bil_Z \to \Bil_S$ is the transfer $\hat f_*$, and commutativity holds by construction.
Applying the above construction with $E = \KO$ and using its naturality in $F$, we obtain all in all the following commutative diagram
\begin{equation*}
\begin{tikzcd}
f_* \Bil_Z \ar[dr] & f_*f^* \Bil_S \ar[r] \ar[l, "w"] \ar[d] & f_* f^* \Omega^\infty \KO \ar[r] \ar[d] & f_* \Omega^\infty f^* \KO \ar[dl] \\
 & \Bil_S \ar[r] & \Omega^\infty \KO.
\end{tikzcd}
\end{equation*}
It follows from \cite[Proposition A.0.4 and Example A.0.6(5)]{EHKSY3} that $f^* \Bil_S \to \Bil_Z$ is a Zariski equivalence, and hence the map labelled $w$ is a Nisnevich equivalence ($f$ being finite).
Since $f_* \Omega^\infty f^* \KO$ and $\Omega^\infty \KO$ are Nisnevich local, we may invert $w$ above to obtain the following commutative diagram
\begin{equation*}
\begin{CD}
f_* \Bil_Z @>>> f_* f^* \Omega^\infty \KO \\
@V{\hat f_*}VV         @V{f_*}VV  \\
\Bil_S @>>> \Omega^\infty \KO.
\end{CD}
\end{equation*}
The result follows.
\end{proof}

\begin{remark}
If $1/2 \not\in S$, then we could \emph{define} $\KO$ as $(\Sigma^\infty_\fr \GW)[\beta^{-1}]$.
Then Corollaries \ref{cor:KO-orient} and \ref{cor:KO-transfer} (as well as Proposition \ref{prop:framed-KO}) remain true.
The problem is that we no longer know what theory $\KO$ represents.
\end{remark}

\section{Miscellaneous}
We collect some results which we believe are well-known, but for which we could not find convenient references.

\subsection{Cotangent complexes and dualizing complexes}
\label{subsec:preliminaries-cotangent}
For any morphism of schemes $f: X \to Y$, there is the cotangent complex $L_f \in D(X)$ (see e.g. \cite[Tag 08P5]{stacks-project}).
If $f$ is lci, then $L_f$ is perfect \cite[Tag 08SH]{stacks-project} and hence defines a point $L_f \in K(X)$.
Consequently in this case we can make sense of the graded determinant $\widetilde\det L_f \in Pic(D(X))$; this is (locally) a shift of a line bundle.

\begin{proposition} \label{prop:f!-Lf}
Let $f: X \to Y$ be lci.
Then there is a canonical isomorphism $f^! \scr O_Y \wequi \widetilde\det L_f$.
\end{proposition}
Recall that by our conventions, $X$ and $Y$ are separated and of finite type over some noetherian scheme $S$; in particular they are themselves noetherian.
We strongly believe that these assumptions are immaterial.
\begin{proof} 
Both sides are compatible with passage to open subschemes of $X$.
Locally on $X$, $f$ factors as a regular immersion followed by a smooth morphism, say $f=pi$.
For either a regular immersion or a smooth morphism $g$, we have isomorphisms \begin{equation} \label{eq:f!-triv} g^! \scr O_Y \wequi \widetilde\det L_g, \end{equation} as desired \cite[Definition III.2, Proposition 7.2]{hartshorne1966residues}.
For composable lci morphisms $p, i$, we have \begin{equation} \label{eq:gf!} (pi)^!(\scr O) \wequi i^*p^!(\scr O) \otimes i^!(\scr O).\end{equation}
Similarly we have a canonical cofiber sequence $i^* L_p \to L_{pi} \to L_i$ and hence \begin{equation} \label{eq:detLgf} \widetilde \det L_{pi} \wequi i^* \widetilde \det L_p \otimes \widetilde \det L_i.\end{equation}
Combining \eqref{eq:f!-triv}, \eqref{eq:gf!} and \eqref{eq:detLgf}, we thus obtain an isomorphism \[ \alpha_{p,i}: f^! \scr O_X \wequi i^!p^! \scr O_X \wequi \widetilde \det L_f. \]

We have thus shown that $f^! \scr O_Y$ is locally isomorphic to $\widetilde\det L_f$ (via $\alpha_{p,i}$), and hence that $A := f^! \scr O_Y \otimes (\widetilde\det L_f)^{-1}$ is an $\scr O_X$-module concentrated in degree $0$.
Exhibiting an isomorphism as claimed is the same as exhibiting $A \wequi \scr O_X$, or equivalently a section $a \in \Gamma(X, A)$ which locally on $X$ corresponds to an isomorphism.
Since $A$ is $0$-truncated, we may construct $a$ locally.
In other words, we need to exhibit a cover $\{U_n\}_n$ of $X$ and isomorphisms $\alpha_n: f^! \scr O_Y|_{U_n} \wequi \widetilde\det L_f|_{U_n}$ such that on $U_n \cap U_m$ we have $\alpha_n \wequi \alpha_m$.
Hence, since we are claiming to exhibit a \emph{canonical} isomorphism, we may do so locally on $X$.
We may thus assume that $f$ factors as $pi$, for a smooth morphism $p: V \to Y$ and a regular immersion $i: X \to V$.
We have already found an isomorphism in this situation, namely $\alpha_{p,i}$.
What remains to do is to show that this isomorphism is independent of the choice of factorization $f = pi$.

Thus let $i': X \to V'$ and $p': V' \to Y$ be another such factorization.
We need to show that $\alpha_{p,i} = \alpha_{p',i'}$.
By considering $V' \times_Y V$, we may assume given a commutative diagram
\begin{equation*}
\begin{CD}
X @>i'>> V' @>p'>> Y \\
@|   @VqVV        @| \\
X @>i>> V   @>p>>  Y,
\end{CD}
\end{equation*}
where $q$ is smooth.
If in \eqref{eq:detLgf} both $f$ and $g$ are smooth, then the isomorphism arises from the first fundamental exact sequence of Kähler differentials \cite[Proposition III.2.2]{hartshorne1966residues}, and hence is the same as the isomorphism \eqref{eq:gf!}.
It follows that we may assume that $p=\id$.
The isomorphism $i'^! q^! \wequi i^!$ is explained in \cite[Proposition III.8.2]{hartshorne1966residues} and reduces via formal considerations (that apply in the same way to $\widetilde \det L_{\ph}$) to the case of a smooth morphism with a section.

We are thus reduced to the following problem.
Let $p: V \to X$ be smooth and $i: X \to V$ a regular immersion which is a section of $p$.
The coherent duality formalism provides us with an isomorphism $\omega_{X/V} \otimes i^* \omega_{V/X} \wequi \scr O_X$; we need to check that this is the same as the isomorphism $\det L_i \otimes i^* \det L_p \wequi \det L_{\id} = \scr O_X$ coming from \eqref{eq:detLgf}.
By \cite[Lemma III.8.1, Definition III.1.5]{hartshorne1966residues}, the first isomorphism arises from the second fundamental exact sequence of Kähler differentials.
This is the same as the second isomorphism.

This concludes the proof.
\end{proof}
\subsection{Grothendieck--Witt rings and Witt rings}
Let $R$ be a commutative ring.
A \emph{symmetric space over $R$} means a finitely generated projective $R$-module $M$ together with a non-degenerate, symmetric bilinear form $\varphi: M \times M \to R$.
The Grothendieck group on the semiring of isomorphism classes of symmetric spaces over $R$ (with operations given by direct sum and tensor product) is denoted $\GW(R)$ and called the \emph{Grothendieck--Witt ring of $R$}.
A symmetric space $(M, \varphi)$ is called \emph{metabolic} if there exists a summand $N \subset M$ with $N = N^\perp$.
The quotient of $\GW(R)$ by the subgroup (which is an ideal) generated by metabolic spaces is denoted $\W(R)$ and called the \emph{Witt ring of $R$}.

\begin{lemma} \label{lemm:metabolic-witt}
Let $R$ be a commutative ring and $(M, \varphi)$ a symmetric space.
The following are equivalent:
\begin{enumerate}
\item $M$ is metabolic
\item $M$ contains an isotropic subspace of half rank: there exists a summand $N \subset M$ such that $\dim M = 2 \dim N$ and $\varphi|_N = 0$.
\end{enumerate}

Now suppose that all finitely generated projective $R$-modules are free of constant rank, e.g. $R$ a local ring.
In this case the image of $M$ in $\GW(R)$ is given by $nh$, where $\dim M = 2n$ and $h$ denotes the hyperbolic plane, corresponding to the matrix $\begin{pmatrix} 0 & 1 \\ 1 & 0 \end{pmatrix}$, which also satisfies $h = 1 + \lra{-1} \in \GW(R)$.
Moreover we have a pullback square
\begin{equation*}
\begin{CD}
\GW(R) @>>> \Z \\
@VVV      @VVV \\
\W(R) @>>> \Z/2,
\end{CD}
\end{equation*}
where the horizontal maps are given by rank.
In particular $\GW(R) \to \W(R) \times \Z$ is injective.
\end{lemma}
\begin{proof}
The equivalence of (1) and (2) is \cite[Corollary I.3.2]{knebusch-bilinear}.
The fact that the image in $\GW(R)$ is given by $nh$ follows from \cite[Proposition I.3.2 and Corollary I.3.1]{knebusch-bilinear}.

Consider the symmetric space $M=\lra{1} \oplus h$.
Thus $M \wequi R^3$ has basis $e_1, e_2, e_3$ with $\lra{e_1,e_1} = 1$, $\lra{e_1, e_2} = 0$, and so on.
Direct computation shows that $f_1 = e_1 + e_2$, $f_2 = e_1 - e_3$, $f_3 = e_3 - e_1 - e_2$ is an orthogonal basis exhibiting $M \wequi \lra{1} \oplus \lra{1} \oplus \lra{-1}$.
Thus $h = 1 + \lra{-1} \in \GW(R)$, as claimed.

We have $\W(R) = \GW(R)/J$, where the ideal $J$ is generated by metabolic spaces.
By the previous assertion, $J=\Z\cdot h$, and so the rank homomorphism maps $J$ isomorphically onto $2\Z$.
The pullback square follows formally.
\end{proof}

\begin{lemma} \label{lemm:GW-gens}
Let $R$ be a local ring.
\begin{enumerate}
\item Let $(M,\varphi)$ be a symmetric space over $R$.
  Then $M$ admits an orthogonal basis if and only if there exists $m \in M$ with $\varphi(m,m) \in R^\times$.
\item $\GW(R)$ is generated by elements of the form $\langle a \rangle$, with $a \in R^\times$.
\end{enumerate}
\end{lemma}
\begin{proof}
(1) If $e_1, \dots, e_n$ is an orthogonal basis then $\varphi(e_i, e_i) = a_i$ whereas $\varphi(e_i, e_j) = 0$ for $i \ne j$; it follows now from non-degeneracy that $a_1 \in R^\times$, and so the condition is necessary.
Now suppose that $m \in M$ with $\varphi(m,m) \in R^\times$.
By \cite[Theorem I.3.2]{milnor1973symmetric} we get $M \simeq Rm \oplus (Rm)^\perp$.
Consider an isomorphism $M \simeq Re_1 \oplus \dots \oplus Re_n \oplus N$, where $n$ is maximal.
We wish to show that $N=0$.
We know that $n > 0$ (by existence of $m$) and if $x \in N$ then $(*)$ $\varphi(x,x) \not\in R^\times$ (because else we could split off $Rx$ as before, contradicting maximality).
Replacing $M$ by $e_1 R \oplus N$, we may assume that $n=1$.
Since $\varphi|_N$ is non-degenerate  and $R$ is local, if $N \ne 0$ there exist $y,z \in N$ with $\varphi(y,z) = 1$.
Set $e_1' = e_1+y$ and $f = e_1 + \lambda z$ (with $\lambda \in R$).
Then $\varphi(e_1',e_1') = \varphi(e_1,e_1) + \varphi(y,y) \in R^\times$ by $(*)$ (and using that $R$ is local), and similarly $\varphi(f,f) \in R^\times$.
On the other hand $\varphi(e_1',f) = \varphi(e_1,e_1) + \lambda$, and hence there exists a (unique) value of $\lambda$ such that $f \in (Re_1')^\perp$.
It follows that $M \simeq Re_1' \oplus Rf \oplus N'$, in contradiction of maximality of $n$.

(2) If $M$ is an inner product space then $M \oplus \langle 1 \rangle$ admits an orthogonal basis by (1), and hence $[M] = [M \oplus \langle 1 \rangle] - \langle 1 \rangle \in GW(R)$ can be expressed in terms of elements of the form $\langle a \rangle$.
\end{proof}

\begin{corollary} \label{cor:GW(F2)}
We have $G\W(\FF_2) = \Z$.
\end{corollary}
\begin{proof}
It is immediate from Lemma \ref{lemm:GW-gens} that $\Z \to \GW(\FF_2)$ is surjective.
Since the rank provides a retraction, the map is also injective, hence an isomorphism.
\end{proof}

\subsection{Regular sequences}
\begin{lemma} \label{lemm:automatic-regularity}
Let $S$ be a scheme, $X \to S$ a smooth morphism, $f_1, \dots, f_n \in \scr O_X(X)$ and put $Z=Z(f_1, \dots, f_n)$.
Assume that for all $s \in S$, either $Z_s$ is empty or else $\dim Z_s \le \dim X_s - n$.
Then $Z \to S$ is flat and for each $z \in Z$, $f_1, \dots, f_n \in \scr O_{Z,z}$ is a strongly regular sequence.
In particular $f_1, \dots, f_n$ is a regular sequence and $Z \to X$ is a regular immersion.
\end{lemma}
\begin{proof}
For $x \in X$ there exist affine open subschemes $x \in U \subset X$, $S' \subset S$ and a factorization $U \to \A^d_{S'} \to S' \to S$ of $U \to S$ with $U \to \A^n_{S'}$ étale \cite[Tag 01V4]{stacks-project}.
Then $Z' := Z \cap U \to S'$ still has fibers of dimension $\le d-n$ and hence equal to $d-n$ by Krull's principal ideal theorem \cite[Tag 0BBZ]{stacks-project}.
This implies that $Z' \to S$ is a relative global complete intersection (in the sense of \cite[Tag 00SP]{stacks-project}) \cite[Lemma 2.1.15]{EHKSY}, whence flat \cite[00SW]{stacks-project}, and $(f_1, \dots, f_n)$ is a strongly regular sequence in $\scr O_{Z',z'}$ for any $z' \in Z'$ \cite[Tag 00SV(1)]{stacks-project}.
Since $x$ was arbitrary, it follows that $Z \to S$ is flat and $(f_1, \dots, f_n)$ form a strongly regular sequence in $\scr O_{Z,z}$ for any $z\in Z$.
For the last statement, we note that if $x \in X \setminus Z$ then $f_i \in \scr O_{X,x}$ is a unit for some $i$, and hence $K(f_\bullet)_x \wequi 0 \wequi (\scr O_Z)_x$, whereas if $x \in Z$ then $K(f_\bullet)_x \wequi \scr O_{Z,z}$ by strong regularity \cite[Tag 062F]{stacks-project}.
\end{proof}

Competing interests: The authors declare none.

\bibliographystyle{alpha}
\bibliography{bibliography}

\newcommand{\etalchar}[1]{$^{#1}$}
\begin{thebibliography}{EHK{\etalchar{+}}20b}

\bibitem[AB84]{Atiyah-Bott}
M.~F. Atiyah and R.~Bott.
\newblock The moment map and equivariant cohomology.
\newblock {\em Topology}, 23(1):1--28, 1984.

\bibitem[AF16]{Asok-Fasel_comparing_Euler_classes}
A.~Asok and J.~Fasel.
\newblock Comparing {E}uler classes.
\newblock {\em Q. J. Math.}, 67(4):603--635, 2016.

\bibitem[Ana15]{Ananyevskiy-SL_projective_bundle_thm}
Alexey Ananyevskiy.
\newblock The special linear version of the projective bundle theorem.
\newblock {\em Compos. Math.}, 151(3):461--501, 2015.

\bibitem[Ana20]{ananyevskiy2019sl}
Alexey Ananyevskiy.
\newblock S{L}-oriented cohomology theories.
\newblock In {\em Motivic homotopy theory and refined enumerative geometry},
  volume 745 of {\em Contemp. Math.}, pages 1--19. Amer. Math. Soc.,
  Providence, RI, [2020] \copyright 2020.

\bibitem[Bac17]{bachmann-very-effective}
Tom Bachmann.
\newblock The generalized slices of hermitian k-theory.
\newblock {\em Journal of Topology}, 10(4):1124--1144, 2017.
\newblock \href{https://arxiv.org/abs/1610.01346}{arXiv:1610.01346}.

\bibitem[Bac18]{bachmann-gwtimes}
Tom Bachmann.
\newblock Some remarks on units in grothendieck–witt rings.
\newblock {\em Journal of Algebra}, 499:229 -- 271, 2018.
\newblock \href{https://arxiv.org/abs/1612.04728}{arXiv:1707.08087}.

\bibitem[BBM{\etalchar{+}}21]{Brazeltontrace}
Thomas Brazelton, Robert Burklund, Stephen McKean, Michael Montoro, and Morgan
  Opie.
\newblock The trace of the local {$\Bbb A^1$}-degree.
\newblock {\em Homology Homotopy Appl.}, 23(1):243--255, 2021.

\bibitem[BF17]{bachmann-criterion}
Tom Bachmann and Jean Fasel.
\newblock On the effectivity of spectra representing motivic cohomology
  theories.
\newblock 2017.
\newblock \href{https://arxiv.org/abs/1710.00594}{arXiv:1710.00594}.

\bibitem[BGI{\etalchar{+}}71]{SGA6}
Pierre Berthelot, Alexandre Grothendieck, Luc Illusie, et~al.
\newblock {\em Th{\'e}orie des intersections et th{\'e}oreme de Riemann-Roch},
  volume 225.
\newblock Springer, 1971.

\bibitem[BH98]{bruns1998cohen}
Winfried Bruns and H~J{\"u}rgen Herzog.
\newblock {\em Cohen-macaulay rings}.
\newblock Cambridge university press, 1998.

\bibitem[BH17]{bachmann-norms}
Tom Bachmann and Marc Hoyois.
\newblock Norms in motivic homotopy theory.
\newblock 2017.
\newblock \href{https://arxiv.org/abs/1711.03061}{arXiv:1711.03061}.

\bibitem[BH20]{bachmann-eta}
Tom Bachmann and Michael~J. Hopkins.
\newblock $\eta$-periodic motivic stable homotopy theory over fields.
\newblock \href{https://arxiv.org/abs/2005.06778}{arXiv:2005.06778}, 2020.

\bibitem[BKW20]{Bethea}
Candace Bethea, Jesse~Leo Kass, and Kirsten Wickelgren.
\newblock Examples of wild ramification in an enriched {R}iemann--{H}urwitz
  formula.
\newblock In {\em Motivic homotopy theory and refined enumerative geometry},
  volume 745 of {\em Contemp. Math.}, pages 69--82. Amer. Math. Soc.,
  Providence, RI, [2020] \copyright 2020.

\bibitem[BM00]{Barge_Morel-Euler_classes}
Jean Barge and Fabien Morel.
\newblock Groupe de {C}how des cycles orient\'{e}s et classe d'{E}uler des
  fibr\'{e}s vectoriels.
\newblock {\em C. R. Acad. Sci. Paris S\'{e}r. I Math.}, 330(4):287--290, 2000.

\bibitem[BW20]{Benoist_WittenbergI}
Olivier Benoist and Olivier Wittenberg.
\newblock On the integral {H}odge conjecture for real varieties, {I}.
\newblock {\em Invent. Math.}, 222(1):1--77, 2020.

\bibitem[CD19]{triangulated-mixed-motives}
Denis-Charles Cisinski and Fr\'{e}d\'{e}ric D\'{e}glise.
\newblock {\em Triangulated categories of mixed motives}.
\newblock Springer Monographs in Mathematics. Springer, Cham, [2019] \copyright
  2019.

\bibitem[CDH{\etalchar{+}}20]{CDHHLMNNS-3}
Baptiste Calm\`es, Emanuele Dotto, Yonatan Harpaz, Fabien Hebestreit, Markus
  Land, Kristian Moi, Denis Nardin, Thomas Nikolaus, and Wolfgang Steimle.
\newblock Hermitian k-theory for stable $\infty$-categories iii.
\newblock \href{https://arxiv.org/abs/2009.07225}{ArXiv 2009.07225}, 2020.

\bibitem[CF17a]{calmes2017comparison}
Baptiste Calm{\`e}s and Jean Fasel.
\newblock A comparison theorem for mw-motivic cohomology.
\newblock {\em arXiv preprint arXiv:1708.06100}, 2017.

\bibitem[CF17b]{calmes2014finite}
Baptiste Calm{\`e}s and Jean Fasel.
\newblock Finite chow--witt correspondences.
\newblock {\em arXiv preprint}, 2017.
\newblock \href{https://arxiv.org/abs/1412.2989}{arXiv:1412.2989}.

\bibitem[CH09]{calmes2009tensor}
Baptiste Calm{\`e}s and Jens Hornbostel.
\newblock Tensor-triangulated categories and dualities.
\newblock {\em Theory Appl. Categ}, 22(6):136--200, 2009.

\bibitem[CH11]{calmes2008push}
Baptiste Calm\`es and Jens Hornbostel.
\newblock Push-forwards for {W}itt groups of schemes.
\newblock {\em Comment. Math. Helv.}, 86(2):437--468, 2011.

\bibitem[Con00]{conrad2000grothendieck}
Brian Conrad.
\newblock {\em Grothendieck duality and base change}.
\newblock Number 1750. Springer Science \& Business Media, 2000.

\bibitem[DJK18]{DJK}
Fr\'ed\'eric D\'eglise, Fangzhou Jin, and Adeel~A. Khan.
\newblock Fundamental classes in motivic homotopy theory.
\newblock {\em Preprint}, available at \url{https://arxiv.org/abs/1805.05920},
  2018.

\bibitem[DM98]{Debarre_Manivel}
Olivier Debarre and Laurent Manivel.
\newblock Sur la vari\'{e}t\'{e} des espaces lin\'{e}aires contenus dans une
  intersection compl\`ete.
\newblock {\em Math. Ann.}, 312(3):549--574, 1998.

\bibitem[EH16]{EisenbudHarris}
David Eisenbud and Joe Harris.
\newblock {\em 3264 and all that---a second course in algebraic geometry}.
\newblock Cambridge University Press, Cambridge, 2016.

\bibitem[EHK{\etalchar{+}}17]{EHKSY}
Elden Elmanto, Marc Hoyois, Adeel~A. Khan, Vladimir Sosnilo, and Maria
  Yakerson.
\newblock Motivic infinite loop spaces.
\newblock arXiv preprint 1711.05248, 2017.

\bibitem[EHK{\etalchar{+}}20a]{EHKSY2}
Elden Elmanto, Marc Hoyois, Adeel~A. Khan, Vladimir Sosnilo, and Maria
  Yakerson.
\newblock Framed transfers and motivic fundamental classes.
\newblock {\em J. Topol.}, 13(2):460--500, 2020.

\bibitem[EHK{\etalchar{+}}20b]{EHKSY3}
Elden Elmanto, Marc Hoyois, Adeel~A. Khan, Vladimir Sosnilo, and Maria
  Yakerson.
\newblock Modules over algebraic cobordism.
\newblock {\em Forum Math. Pi}, 8:e14, 44, 2020.

\bibitem[Eis95]{Eisenbud_CommutativeAlgebra}
David Eisenbud.
\newblock {\em Commutative algebra}, volume 150 of {\em Graduate Texts in
  Mathematics}.
\newblock Springer-Verlag, New York, 1995.
\newblock With a view toward algebraic geometry.

\bibitem[EL77]{eisenbud77}
David Eisenbud and Harold~I. Levine.
\newblock An algebraic formula for the degree of a {$C^{\infty }$} map germ.
\newblock {\em Ann. of Math. (2)}, 106(1):19--44, 1977.
\newblock With an appendix by Bernard Teissier, ``Sur une in{\'e}galit{\'e}
  {\`a} la Minkowski pour les multiplicit{\'e}s''.

\bibitem[Fas13]{Fasel-projective_bundle}
Jean Fasel.
\newblock The projective bundle theorem for {${\bf I}^j$}-cohomology.
\newblock {\em J. K-Theory}, 11(2):413--464, 2013.

\bibitem[FK13]{finashin13}
Sergey Finashin and Viatcheslav Kharlamov.
\newblock Abundance of real lines on real projective hypersurfaces.
\newblock {\em Int. Math. Res. Not. IMRN}, (16):3639--3646, 2013.

\bibitem[FK15]{Finashin-Kharlamov-3_planes}
S.~Finashin and V.~Kharlamov.
\newblock Abundance of 3-planes on real projective hypersurfaces.
\newblock {\em Arnold Math. J.}, 1(2):171--199, 2015.

\bibitem[FK21]{FK-Segre_indices}
Sergey Finashin and Viatcheslav Kharlamov.
\newblock Segre {I}ndices and {W}elschinger {W}eights as {O}ptions for
  {I}nvariant {C}ount of {R}eal {L}ines.
\newblock {\em Int. Math. Res. Not. IMRN}, (6):4051--4078, 2021.

\bibitem[Ful84]{fulton-intersection}
W.~Fulton.
\newblock {\em Intersection theory}.
\newblock Ergebnisse der Mathematik und ihrer Grenzgebiete. Springer-Verlag,
  1984.

\bibitem[GP99]{GP-loc}
T.~Graber and R.~Pandharipande.
\newblock Localization of virtual classes.
\newblock {\em Invent. Math.}, 135(2):487--518, 1999.

\bibitem[GP18]{gwilliam2018enhancing}
Owen Gwilliam and Dmitri Pavlov.
\newblock Enhancing the filtered derived category.
\newblock {\em Journal of Pure and Applied Algebra}, 222(11):3621--3674, 2018.

\bibitem[Har66]{hartshorne1966residues}
Robin Hartshorne.
\newblock {\em Residues and duality}, volume~20.
\newblock Springer, 1966.

\bibitem[Har13]{hartshorne2013algebraic}
Robin Hartshorne.
\newblock {\em Algebraic geometry}, volume~52.
\newblock Springer Science \& Business Media, 2013.

\bibitem[Hor05]{hornbostel2005a1}
Jens Hornbostel.
\newblock A1-representability of hermitian k-theory and witt groups.
\newblock {\em Topology}, 44(3):661--687, 2005.

\bibitem[Hoy17]{hoyois-equivariant}
Marc Hoyois.
\newblock The six operations in equivariant motivic homotopy theory.
\newblock {\em Advances in Mathematics}, 305:197--279, 2017.

\bibitem[Hoy21]{hoyois2018localization}
Marc Hoyois.
\newblock The localization theorem for framed motivic spaces.
\newblock {\em Compos. Math.}, 157(1):1--11, 2021.

\bibitem[HWXZ19]{Hornbostel_real_cycle_class_map}
Jens Hornbostel, Matthias Wendt, Heng Xie, and Marcus Zibrowius.
\newblock The real cycle class map.
\newblock {\em arXiv preprint arXiv:1911.04150}, 2019.

\bibitem[Ill71]{illusie1971complexe}
L.~Illusie.
\newblock {\em Complexe cotangent et d{\'e}formations}.
\newblock Number v. 1 in Lecture notes in mathematics. Springer-Verlag, 1971.

\bibitem[Jac17]{jacobson-fundamental-ideal}
Jeremy Jacobson.
\newblock Real cohomology and the powers of the fundamental ideal in the witt
  ring.
\newblock {\em Annals of K-Theory}, 2(3):357--385, 2017.

\bibitem[Khi77]{khimshiashvili}
G.~N. Khimshiashvili.
\newblock The local degree of a smooth mapping.
\newblock {\em Sakharth. SSR Mecn. Akad. Moambe}, 85(2):309--312, 1977.

\bibitem[Kne77]{knebusch-bilinear}
Manfred Knebusch.
\newblock Symmetric bilinear forms over algebraic varieties.
\newblock In G.~Orzech, editor, {\em Conference on quadratic forms}, volume~46
  of {\em Queen's papers in pure and applied mathematics}, pages 103--283.
  Queens University, Kingston, Ontario, 1977.

\bibitem[Knu91]{knus}
Max-Albert Knus.
\newblock {\em Quadratic and {H}ermitian forms over rings}, volume 294 of {\em
  Grundlehren der Mathematischen Wissenschaften [Fundamental Principles of
  Mathematical Sciences]}.
\newblock Springer-Verlag, Berlin, 1991.
\newblock With a foreword by I. Bertuccioni.

\bibitem[KW19]{kass2016class}
Jesse~Leo Kass and Kirsten Wickelgren.
\newblock The class of {E}isenbud--{K}himshiashvili--{L}evine is the local
  {A}1-brouwer degree.
\newblock {\em Duke Mathematical Journal}, 168(3):429--469, 2019.

\bibitem[KW21]{CubicSurface}
Jesse~Leo Kass and Kirsten Wickelgren.
\newblock An arithmetic count of the lines on a smooth cubic surface.
\newblock {\em Compos. Math.}, 157(4):677--709, 2021.

\bibitem[L{\'A}17]{lopez2017ring}
Alejo L{\'o}pez-{\'A}vila.
\newblock {\em $E_\infty$-ring Structures in Motivic Hermitian K-theory}.
\newblock PhD thesis, Universit{\"a}t Osnabr{\"u}ck, 2017.

\bibitem[Lam05]{Lam-intro_quadratic_forms_over_fields}
T.~Y. Lam.
\newblock {\em Introduction to quadratic forms over fields}, volume~67 of {\em
  Graduate Studies in Mathematics}.
\newblock American Mathematical Society, Providence, RI, 2005.

\bibitem[Lev19]{Levine-Witt}
Marc Levine.
\newblock Motivic {E}uler characteristics and {W}itt-valued characteristic
  classes.
\newblock {\em Nagoya Mathematical Journal}, 236:251 -- 310, 2019.
\newblock Special Issue: Celebrating the 60th birthday of Shuji Saito.

\bibitem[Lev20]{Levine-EC}
Marc Levine.
\newblock Aspects of enumerative geometry with quadratic forms.
\newblock {\em Doc. Math.}, 25:2179--2239, 2020.

\bibitem[LR20]{levine2018motivic}
Marc Levine and Arpon Raksit.
\newblock Motivic {G}auss-{B}onnet formulas.
\newblock {\em Algebra Number Theory}, 14(7):1801--1851, 2020.

\bibitem[McK21]{McKean-Bezout}
Stephen McKean.
\newblock An arithmetic enrichment of {B}\'{e}zout's {T}heorem.
\newblock {\em Math. Ann.}, 379(1-2):633--660, 2021.

\bibitem[MH73]{milnor1973symmetric}
John~Willard Milnor and Dale Husemoller.
\newblock {\em Symmetric bilinear forms}, volume~60.
\newblock Springer, 1973.

\bibitem[Mor05]{morel2005stable}
Fabien Morel.
\newblock The stable $\mathbb{A}^1$-connectivity theorems.
\newblock {\em K-theory}, 35(1):1--68, 2005.

\bibitem[Mor12]{A1-alg-top}
Fabien Morel.
\newblock {\em $\mathbb{A}^1$-Algebraic Topology over a Field}.
\newblock Lecture Notes in Mathematics. Springer Berlin Heidelberg, 2012.

\bibitem[OT14]{okonek14}
Christian Okonek and Andrei Teleman.
\newblock Intrinsic signs and lower bounds in real algebraic geometry.
\newblock {\em J. Reine Angew. Math.}, 688:219--241, 2014.

\bibitem[Pau20]{Pauli-Quintic_3_fold}
Sabrina Pauli.
\newblock Quadratic types and the dynamic euler number of of lines on a quintic
  $3$-fold.
\newblock ArXiv preprint 2006.12089, 2020.

\bibitem[PW10a]{panin2010algebraic}
Ivan Panin and Charles Walter.
\newblock On the algebraic cobordism spectra msl and msp.
\newblock {\em arXiv preprint arXiv:1011.0651}, 2010.

\bibitem[PW10b]{panin2010quaternionic}
Ivan Panin and Charles Walter.
\newblock Quaternionic grassmannians and pontryagin classes in algebraic
  geometry.
\newblock {\em arXiv preprint arXiv:1011.0649}, 2010.

\bibitem[PW18]{panin2010motivic}
I.~Panin and C.~Walter.
\newblock On the motivic commutative ring spectrum {${\bf BO}$}.
\newblock {\em Algebra i Analiz}, 30(6):43--96, 2018.

\bibitem[Sch10]{schlichting2010mayer}
Marco Schlichting.
\newblock The mayer-vietoris principle for grothendieck-witt groups of schemes.
\newblock {\em Inventiones mathematicae}, 179(2):349--433, 2010.

\bibitem[Sol06]{solomon06}
Jake~P. Solomon.
\newblock Intersection theory on the moduli space of holomorphic curves with
  {L}agrangian boundary conditions.
\newblock {\em PhD Thesis}, available at
  \url{https://arxiv.org/abs/math/0606429}, 2006.

\bibitem[SS75]{scheja}
G{\"u}nter Scheja and Uwe Storch.
\newblock \"{U}ber {S}purfunktionen bei vollst\"andigen {D}urchschnitten.
\newblock {\em J. Reine Angew. Math.}, 278/279:174--190, 1975.

\bibitem[ST15]{schlichting2015geometric}
Marco Schlichting and Girja~S Tripathi.
\newblock Geometric models for higher grothendieck--witt groups in
  $\mathbb{A}^1$-homotopy theory.
\newblock {\em Mathematische Annalen}, 362(3-4):1143--1167, 2015.

\bibitem[{Sta}18]{stacks-project}
The {Stacks Project Authors}.
\newblock {\itshape Stacks Project}.
\newblock \url{http://stacks.math.columbia.edu}, 2018.

\bibitem[SW21]{FourLines}
Padmavathi Srinivasan and Kirsten Wickelgren.
\newblock An arithmetic count of the lines meeting four lines in {$P^3$}.
\newblock {\em Trans. Amer. Math. Soc.}, 374(5):3427--3451, 2021.

\bibitem[Wen20]{Wendt-oriented_schubert}
Matthias Wendt.
\newblock Oriented {S}chubert calculus in {C}how--{W}itt rings of
  {G}rassmannians.
\newblock In {\em Motivic homotopy theory and refined enumerative geometry},
  volume 745 of {\em Contemp. Math.}, pages 217--267. Amer. Math. Soc.,
  Providence, RI, [2020] \copyright 2020.

\end{thebibliography}

\end{document}